\documentclass[11pt, dvipsnames]{amsart}
\usepackage[utf8]{inputenc}
\usepackage{amsmath}
\usepackage{amsfonts}
\usepackage{amssymb}
\usepackage{amsthm}
\usepackage{graphicx}
\usepackage{hyperref}
\usepackage{xcolor}
\usepackage{todonotes}
\usepackage[T1]{fontenc}
\usepackage[utf8]{inputenc}
\allowdisplaybreaks
\usepackage[left=2.5cm,top=3cm,right=2.5cm,bottom=3cm,bindingoffset=0.5cm]{geometry}
\definecolor{darkpastelgreen}{rgb}{0.01, 0.75, 0.24}

\usepackage{comment}
\title{
Rapid stabilization of general linear systems with 
$F$-equivalence
}
\author[A.~Hayat]{Amaury Hayat}
\thanks{CERMICS, \'{E}cole des Ponts ParisTech, 6 - 8, Avenue Blaise Pascal, Cité Descartes—Champs sur Marne, 77455 Marne la Vall\'{e}e, France. E-mail: \texttt{amaury.hayat@enpc.fr.}}

\author[E.~Loko]{Epiphane Loko}
\thanks{CERMICS, \'{E}cole des Ponts ParisTech, 6 - 8, Avenue Blaise Pascal, Cité Descartes—Champs sur Marne, 77455 Marne la Vall\'{e}e, France \& L2S- Paris Saclay University, CentraleSupelec, bât. Bréguet, 3, rue Joliot Curie, 91190 Gif-sur-Yvette. E-mail: \texttt{dagbegnon.loko@enpc.fr.}}

\date{}

\newtheorem{theorem}{THEOREM}[section]
\newtheorem{rmk}[theorem]{REMARK}
\newtheorem{lem}[theorem]{LEMMA}
\newtheorem{defn}[theorem]{DEFINITION}
\newtheorem{prop}[theorem]{PROPOSITION}
\newtheorem{corollary}[theorem]{COROLLARY}
\newtheorem{claim}[theorem]{CLAIM}

\begin{document}
\maketitle

\begin{abstract}
We study the rapid stabilization of general linear systems, when the differential operator $\mathcal{A}$ has a Riesz basis of eigenvectors.
We find simple sufficient conditions for the rapid stabilization and the construction of a relatively explicit feedback operator. We use an $F$-equivalence 
approach \textcolor{black}{relying on Fredholm transformation}
to show a stronger result:
under these sufficient conditions 
the system is equivalent to a simple exponentially stable system, with arbitrarily large decay rate. In 
particular, our conditions improve the existing conditions of rapid stabilization for non-parabolic operators such as skew-adjoint systems. 

\end{abstract}

\section{Introduction}

Stabilization is one of the three main problems in control theory, together with controllability and optimal control. The goal is to drive a system of differential equations to a given long-term behavior by acting on it with a control, with the specificity that the control has to be a function of the state. This creates a feedback loop (the control influences the system which in turn influences the control), which can make the achieved stability relatively robust. However, this feedback loop makes the mathematical problem of finding a suitable control difficult and even the well-posedness can be a challenge.
From a practical point of view, stabilization has many applications \cite{BastinCoronBook}, in engineering \cite{astrom1995pid}, transportation \cite{delle2019feedback,hayat2023dissipation,Goatin_stabilization_2024}, chemistry \cite{viel1997global}, biology \cite{pouchol2018global}, population dynamics \cite{caponigro2013sparse}, fluid mechanics \cite{BastinCoronBook,HS}, economics \cite{kellett2019feedback} etc.

From an abstract point of view, the stabilization problem in a linear framework is the following: given a system
\begin{equation}
\label{eq:sys0}
    \partial_{t}u = \mathcal{A}u + Bw(t),
\end{equation}
where $\mathcal{A}$ is a differential operator and $B$ is a given operator, we would like to find a control feedback law $w(t)=K(u(t,\cdot))$ such that the system \eqref{eq:sys0} is exponentially stable. That is, we want all solutions of the system to converge exponentially quickly to 0. 
A more ambitious problem is to require that for any $\lambda>0$ there exists a control feedback law $w(t)=K_{\lambda}(u(t,\cdot))$ such that all solutions of the system converge exponentially quickly to 0 with decay rate at least $\lambda>0$. This is the so-called \emph{rapid stabilization} (or \emph{complete stabilization}). Solving this problem is all the more challenging when $w$ is finite-dimensional (in particular this implies that $w$ is not a function of the space variable $x$) and $B$ is unbounded. This is the case, for instance, when the control is located at a boundary, or when the control is a distributed force where only the amplitude can be controlled (see for instance \cite{beauchard2010local,Coron2014-local,alabau2021superexponential,Coron2022-stabilization}
for particular examples). This problem is studied since (at least) the work of Slemrod in 1972 \cite{slemrod1972linear} and powerful results concerning the exponential stabilization of generic linear systems were obtained \textcolor{black}{either using frequency-domain criterion method (see for instance \cite{Rebarber2007frequency,Pruss1984spectrum, Huang1985characteristic} and the references therein, see also the frequency-Lyapunov method of \cite{xiang2024quantitative}), or} tools from optimal control and LQ theory by, among others, Lions, Barbu, Lasiecka and Triggiani \cite{lions1971optimal,barbu2018controllability,LT1,LT2}\textcolor{black}{. Several of these results can be extended to rapid stabilization and an unbounded operator $B$ \cite{urquiza2005rapid,vest2013rapid,badra2014fattorini,trelat2019characterization,liu2022characterizations,ma2023feedback,kunisch2025frequency}} and semilinear systems \cite{badra2011stabilization,breiten2014riccati,barbu2018controllability,nguyen2024stabilization,badra2020local}. Nevertheless, the control feedback laws obtained often rely either on the knowledge of the semigroup $e^{\mathcal{A}^{*}t}$, which makes it not always explicit, or on solving a minimization problem and an algebraic Riccati equation \textcolor{black}{or a Hamilton-Jacobi-Bellman equation}, 
which can be sometimes complicated to achieve \cite{komornik1997rapid}, \textcolor{black}{especially for non-parabolic systems}.
For this reason, many works on the stabilization problem have focused on particular systems, especially when $B$ is unbounded, see for instance (such as \cite{BastinCoronBook,komornik1997rapid,zuazua1990uniform,raymond2007feedback}).\\

In recent years, a new approach has been introduced to tackle this problem more generally: the 
\textcolor{black}{\emph{feedback equivalence ($F$-equivalence)}}. 
Instead of trying directly to find a feedback $K$, this method consists in solving a different mathematical problem: finding
an isomorphism-feedback pair $(T,K)$ such that $T$ maps the original system of interest to an exponentially stable target system. If such a pair exists, then the original system with feedback $K$ is exponentially stable, thanks to the isomorphism property of $T$. In the particular case where $T$ is a Volterra \textcolor{black}{transformation} of the second kind, this method corresponds to the well investigated \emph{backstepping method} introduced by Krstic and his collaborators in \cite{BaloghKrstic,BBK2003,KrsticBook} (and inspired from the finite-dimensional method \cite{backstepping1,backstepping2,backstepping3} and the adaptation \cite{backsteppingadapt}) which has been very successful in the last 20 years for many different classes of systems (see for instance \cite{hu2015control}). For this reason the $F$-equivalence \textcolor{black}{in infinite dimension} is sometimes called \emph{generalized backstepping} \textcolor{black}{or \emph{Fredholm backstepping}, when relying on Fredholm transformation, although no proper backstepping phenomena occur.}

\textcolor{black}{For finite-dimensional systems, the concept of $F$-equivalence goes back to Brunosvky \cite{brunovsky1970classification} for linear systems and was extended to nonlinear systems in many subsequent work (see for instance \cite{jakubczyk1993remarks,fliess1999lie,martin2001flat,cheng2005feedback}). For infinite-dimensional systems, a similar approach}
relying on Fredholm transformations was introduced 
in \cite{Coron2014-local} for the special case of the Korteweg-de Vries equation and the Kuramoto-Sivashinsky equation \cite{CoronLu15}. Other systems with boundary controls were investigated using \textcolor{black}{Fredholm transformations} as a generalization of the Volterra backstepping in \cite{coron2017finite,deutscher2019fredholm,redaud2022stabilizing}. Since then,
several works in a similar spirit have successfully extended and formalized the $F$-equivalence to many different frameworks: the (linear) Schrödinger equation in \cite{Coron2018-rapid}, a degenerate parabolic equation \cite{gagnon2021fredholm,lissy2023rapid}, the 1-D heat equation in \cite{Gagnon2022-fredholm-laplacien}, the transport equation in \cite{zhang2022internal}, the linearized Saint-Venant equations in \cite{Coron2022-stabilization}. The hope that this approach could be used in a general framework, rather than in specific cases, materialized in \cite{Gagnon2022-fredholm} where the authors introduced a compactness-duality method to generalize the approach to any (linear) skew-adjoint system when the differential operator has eigenvalues 
$\lambda_{n}$ scaling as $n^{\alpha}$ with $\alpha >1$. 
This impressive result resulted from overcoming the structural limitations of the original method in \cite{Coron2014-local} and as a side result allowed to deal with systems like the capillary gravity water-wave equations that was an open question presented in \cite{CoronCDF}. Nevertheless, these results are restricted to skew-adjoint systems which enjoy very nice properties (such as a basis of orthonormal eigenvectors, pure imaginary eigenvalues, generates a $C^{0}$ group, etc.)

In this paper we generalize the $F$-equivalence to a much more generic class of systems: 
any system where $\mathcal{A}$ generates a $C^{0}$ semigroup, has a Riesz basis of eigenvectors with eigenvalues with finite multiplicities and satisfying a growth assumption (see \eqref{ln+1_cond}--\eqref{ln-lp_cond}). 
We illustrate these results on several linear and nonlinear (semilinear) examples: Schrödinger equation, parabolic systems and a general diffusion equation, \textcolor{black}{Burgers} equation, and a Gribov system that is neither self nor skew adjoint.

The main point of this approach is to give explicit feedback controls and conditions under which the PDE system under consideration is equivalent to a simpler PDE system, rather than trying to improve the sufficient conditions of stabilizability coming from the successful existing abstract approaches (such as \cite{LT1,LT2,trelat2019characterization,vest2013rapid,ma2023feedback,liu2022characterizations,nguyen2024stabilization}). Nevertheless, in some cases, our method results in very weak and generic sufficient conditions for rapid stabilization 
which, to the best of our knowledge, were yet unknown.

In particular, for skew-adjoint systems we show that the system can be rapidly stabilized even if the control operator $B$ is \textbf{not} admissible \textbf{and} the system is \textbf{not} exactly null controllable \textcolor{black}{in the space of stabilization}. This is discussed in Section \ref{sec:discussion}.

\section{Notation and assumptions}
\label{sec:notations}
{\color{black}
\noindent \textbf{Notation}\;\;
In what follows, 
$\mathbb{N}^{*}$ stands for the set of positive integers
that is
$\mathbb{N}^*=\mathbb{N}\setminus\{0\}$ and $\mathbb{Z}^*=\mathbb{Z}\setminus\{0\}.$ 
Similarly we denote $\mathbb{R}^*=\mathbb{R}\setminus\{0\}$ and $\mathbb{R}_+=[0, +\infty)$. 
The symbol $\lesssim$ means lower or equal up to a multiplicative constant.}\\

Consider the control system described by \eqref{eq:sys0}. Let $(\textcolor{black}{X},\langle\cdot,\cdot\rangle)$ be a Hilbert space. The operator $\mathcal{A}: \mathcal{D}(\mathcal{A})\rightarrow \textcolor{black}{X}$ is assumed to be the generator of a $C^0$ semigroup $\{e^{At}\}_{t\geq 0}$ with finite growth bound in $\mathcal{D}(\mathcal{A})$ where $\mathcal{D}(\mathcal{A})$ is defined as 
\begin{equation*}
    \mathcal{D}(\mathcal{A}):=\{f\in \textcolor{black}{X}, \ \mathcal{A}f\in \textcolor{black}{X}\}.
\end{equation*} 
We assume the following
\begin{itemize}
\item[($A_{1}$)] 
$\mathcal{A}$ is diagonalizable, i.e. there exists a Riesz basis of $\textcolor{black}{X}$ consisting of eigenvectors of $\mathcal{A}$ (see \cite{weiss2011eigenvalues}), and the associated eigenvalues have bounded multiplicities.
\end{itemize}
 Let us denote by $m\in\mathbb{N}^{*}$ the highest multiplicity, one can decompose the space $\textcolor{black}{X}$ into
\begin{equation}
\label{eq:sumHi}
    \textcolor{black}{X} = \textcolor{black}{X}_{1}+ ... + \textcolor{black}{X}_{m},
\end{equation}
such that for any $t\geq 0$, $e^{t \mathcal{A}}\textcolor{black}{X}_{i}\subset \textcolor{black}{X}_{i}$, and on which $\mathcal{A}$ has a spectral decomposition with simple eigenvalues (see \cite[Appendix G]{Gagnon2022-fredholm} for an explicit construction of such a decomposition\footnote{the only difference being that \eqref{eq:sumHi} is not a direct sum when the basis of eigenvectors of $\mathcal{A}$ is not orthogonal and the $(\varphi_{n}^{i+1})$ are chosen outside $Span((\varphi_{n}^{i})_{n\in\mathbb{N}^{*}})$ but not necessarily orthogonal}).
 In other words, on each of the $\textcolor{black}{X}_i$, there exists a Riesz basis (of $\textcolor{black}{X}_{i}$) of eigenvectors $(\varphi^{i}_{n})_{n\in\mathbb{N}^{*}}$ with associated eigenvalues $(\lambda_{n}^{i})_{n\in\mathbb{N}^{*}}$ such that $\lambda_{n}^{i}\neq \lambda_{p}^{i}$ for any $n\neq p$, and $(\varphi_{n}^{i})_{(i,n)\in\{1,...,m\}\times\mathbb{N}^{*}}$ is a Riesz basis of $\textcolor{black}{X}$.
In what follow, we assume without loss of generality that any $\textcolor{black}{X}_i$ is generated by \textcolor{black}{an infinite} number of $\varphi_n^i.$ Otherwise the spaces $\textcolor{black}{X}_i$ generated by finite number of $\varphi_n^i,$ can be easily treated as finite dimension spaces \cite{coron2015stabilization} and the following argument works with a finite set $\mathcal{I}$ instead of $\mathbb{N}^{*}$.
We define for any $s\in \mathbb{R},$ the space
\begin{equation*}
\mathcal{H}^{s}:=\mathcal{H}_{1}^{s}+...+\mathcal{H}_{m}^{s},
\end{equation*}
where
\begin{equation*}
\mathcal{H}_{i}^{s}=\left\{f = \sum\limits_{n\in\mathbb{N}^{*}}f_{n}^i
\varphi_{n}^{i}
,\ 
\sum_{n\in \mathbb{N^*}}n^{2s}|f_{n}^i|^{2} 
<\infty\right\}.
\end{equation*} 

These spaces can be endowed with a natural Hilbert space structure (the formal proof is given in \ref{app:Riesz}):
\begin{lem}
\label{lem:structure}
For any $s \in\mathbb{R}$, and $i \in \{1,...,m\}$ , the 
function
\begin{equation}
\|\cdot\|_{\mathcal{H}^{s}_{i}}\;:\;f\rightarrow \|f\|_{\mathcal{H}^{s}_{i}} := \left(\sum_{n\in \mathbb{N^*}}n^{2s}|f_{n}^i|^{2}\right)^{1/2},
\end{equation}
is a norm of $\mathcal{H}^{s}_{i}$ and ($\mathcal{H}^{s}_{i}$,$\|\cdot\|_{\mathcal{H}^{s}_{i}}$) is a Hilbert space, with inner product
\begin{equation}
\label{eq:inner}
    \langle f,g\rangle _{\mathcal{H}^s_{i}} :=\sum_{n\in \mathbb{N^*}}n^{2s}f_{n}^{i}\overline{g_{n}^{i}},\;\forall f =\sum\limits_{n\in\mathbb{N}^{*}}f_{n}^{i}\varphi_{n}^{i},\; g = \sum\limits_{n\in\mathbb{N}^{*}}g_{n}^{i}\varphi_{n}^{i}\in \mathcal{H}^s_{i}.
\end{equation}
\end{lem}

As a consequence, for any $s=(s_{1},...,s_{m})\in\mathbb{R}^{m}$ we can define the Hilbert space
\begin{equation*}
    \mathcal{H}^{\vec s}:=\mathcal{H}_1^{s_1}+ ... +\mathcal{H}_m^{s_m}.
\end{equation*}
These spaces are intimately related to the operator $\mathcal{A}$. Indeed, if the assumption \ref{ln+1_cond} (see below) holds, then $\mathcal{H}^{\alpha} = D(\mathcal{A})$, and for any $s\in\mathbb{Z}$, $\mathcal{H}^{s\alpha} = D(\mathcal{A}^{s})$ (this can even be extended to $s\in\mathbb{R}$ if $\mathcal{A}$ is sectorial, see \cite[Section 2.6]{PazyBook} for more details). These spaces also coincide with the usual Sobolev spaces $H^{s}$ whenever $(\varphi_{n})_{n\in\mathbb{N}}$ is an orthonormal basis of $L^{2}$. Similarly to the usual Sobolev spaces, these spaces satisfy the following embedding property
\begin{lem}
\label{lem:embedding}
    For any $s\in\mathbb{R}$ and $\varepsilon>0$,
    \begin{equation*}
    \mathcal{H}^{s+\varepsilon}\text{ is compactly embedded in }\mathcal{H}^{s}.
    \end{equation*}
\end{lem}
This actually holds for the $\mathcal{H}^{s}_{i}$ taken separately and the proof is identical to \cite[Section 1.8.1]{Gagnon2022-fredholm}.
 Recall that if $s=0$, $\mathcal{H}^{s} = \textcolor{black}{X}$. Since $(\varphi_{n})_{n\in\mathbb{N}}$ is not necessarily an orthonormal basis of $\textcolor{black}{X}$ with the canonical norm $\|\cdot\|_{\textcolor{black}{X}}$, then $\|\cdot\|_{\textcolor{black}{X}}$ is different from $\|\cdot\|_{\mathcal{H}^{0}}$. However, we have the following
\begin{lem}
\label{lem:equivnorm}
    If $s=0$, then $\|\cdot \|_{\mathcal{H}^{s}}$ and $\|\cdot\|_{\textcolor{black}{X}}$ are equivalent norms.
\end{lem}
This is a direct consequence of the fact that $\varphi_{n}$ is a Riesz basis. Indeed, by definition there exists positive constants $c$ and $C$ such that for any $(f_{n})_{n\in\mathbb{N}^{*}}\in l^{2}$,
\begin{equation}
c\left(\sum\limits_{n\in\mathbb{N}}|f_{n}|^{2}\right)^{1/2}\leq \|\sum\limits_{n\in\mathbb{N}}f_{n}\varphi_{n}\|_{\textcolor{black}{X}}\leq C\left(\sum\limits_{n\in\mathbb{N}}|f_{n}|^{2}\right)^{1/2}.
\end{equation}
Since  $\|\sum\limits_{n\in\mathbb{N}}f_{n}\varphi_{n}\|_{\mathcal{H}^{0}}=(\sum\limits_{n\in\mathbb{N}}|f_{n}|^{2})^{1/2}$, Lemma \ref{lem:equivnorm} follows.
\\

 From \cite{beauchard2010local}, for each $i\in\{1,...,m\}$, the family $(\varphi_n^{i})_{n\in \mathbb{N^*}}$ admits a unique bi-orthonormal family $(\widetilde{\varphi_n}^{i})_{n\in \mathbb{N^*}}$ in $\textcolor{black}{X}_{i}$ which forms a Riesz basis of $\textcolor{black}{X}_{i}$ and corresponds to the family of eigenvectors of the adjoint $\mathcal{A}^*$ of $\mathcal{A}$ (with respect to the scalar product $\langle \cdot,\cdot \rangle$ of $\textcolor{black}{X}$). Thus any $f\in \textcolor{black}{X}_{i}$ admits a unique decomposition in $\textcolor{black}{X}_{i}$ as follows
\begin{equation*}
    f=\sum_{n\in \mathbb{N^*}}f_n^i\varphi_n^{i}=\sum_{n\in \mathbb{N^*}}\langle f,\widetilde{\varphi_n}^{i}\rangle \varphi_n^{i},
\end{equation*}
and, as a consequence $\langle f,\widetilde{\varphi_n}^{i}\rangle = \langle f,\varphi_n^{i}\rangle_{\mathcal{H}^{0}_{i}}$ and this extends to any function $f\in \mathcal{H}_{i}^{s}$ for $s\in \mathbb{R}$. 

\noindent \textbf{Assumptions}\;\; Through the paper, we assume the following assumptions:
there exists $(\alpha_1,\cdots,\alpha_m)\in (1,+\infty)^m,$ such that
\begin{itemize}
    \item 
    \begin{align}
n^{\alpha_i}\lesssim|\lambda_n^i|+1\lesssim 
        n^{\alpha_i},\quad \forall n\in \mathbb{N}^*, i\in\{1,\cdots,m\} \label{ln+1_cond}
    \end{align}
    \item there exists $C>0$ 
    such that for any $n,p\in \mathbb{N}^*, i\in\{1,\cdots,m\},$
    \begin{equation}\label{ln-lp_cond}
        |\lambda_n^i-\lambda_p^i|\geq C n^{\alpha_i-1}|n-p|.
    \end{equation}
    
\end{itemize}

{\color{black}
We now introduce the definition of exponential stability 
\begin{defn}[Exponential stability]
\label{def:expstab}
Let $\mathcal{A}\in \mathcal{L}(D(\mathcal{A}),X)$ satisfying assumptions $(A_{1})$ and \eqref{ln+1_cond}--\eqref{ln-lp_cond}, $B\in \mathcal{H}_{1}^{-\alpha_{1}}\times...\times \mathcal{H}_{m}^{-\alpha_{m}}$
and $K\in \mathcal{L}(D(A);\mathbb{C}^{m})$. The system 
\begin{equation}
\label{eq:sysdef}
    \partial_{t}u = \mathcal{A}u+BK(u),
\end{equation}
is said exponentially stable with decay rate $\mu>0$ in $X$ if the system is well-posed in $X$, i.e. $\mathcal{A}+BK$ generates a $C^{0}$-semigroup on $X$, and if there exists $C>0$ such that for any $u_{0}\in X$, the unique solution $u\in C^{0}([0,+\infty),X)$ with initial condition $u_{0}$ to the system \eqref{eq:sysdef} satisfies
\begin{equation}
    \|u(t)\|_{X}\leq Ce^{-\mu t}\|u_{0}\|_{X},\;\forall t\in[0,+\infty).
\end{equation}
\end{defn}
\begin{rmk}
    Note that it is not assumed a priori that $K$ belongs to $\mathcal{L}(X,\mathbb{C}^{m})$, and there could be cases where $K$ does not belong to this space. This situation may happen, for instance, with controls depending only on point-wise or boundary values (see for instance \cite{BastinCoronBook}) where typically $X=L^{2}(0,1)$, $K\in \mathcal{L}(H^{1}(0,1);\mathbb{R}^{m})$ but $K\notin \mathcal{L}(L^{2}(0,1),\mathbb{R}^{m})$.
\end{rmk}
}

\section{Main results}
Our main result generates the exponential stabilization of system \eqref{eq:sys0} for any operator $\mathcal{A}$ satisfying ($A_{1}$), \eqref{ln+1_cond}--\eqref{ln-lp_cond}.

\begin{theorem}\label{Theo_main}
Consider the control system \eqref{eq:sys0} with operator $\mathcal{A}$ satisfying the assumptions $(A_{1})$, \eqref{ln+1_cond} and \eqref{ln-lp_cond}. If $B=(B_1,\cdots,B_m)\in\mathcal{H}_{\textcolor{black}{1}}^{-\alpha_1/2}\times \cdots\times \mathcal{H}_{\textcolor{black}{m}}^{-\alpha_m/2}$ is such that 
\begin{align}
  c_1\leq |\langle B_i,\widetilde{\varphi_n}^i\rangle|\leq c_2 n^{\gamma_{i}}, \quad \forall n\in \mathbb{N}^*,\ i\in\{1,\cdots,m\} \label{control_cond} 
\end{align} 
for some constant $c_1,c_2>0$, and some $\gamma_{i}\in [0,(\alpha_{i}-1)/2)$ for $i\in\{1,...,m\}$
and set \textcolor{black}{$\vec{\alpha}=(\alpha_{1},...,\alpha_{m})$.} 
Then, 
for any $\lambda_{0}>0,$ 
there exist $\lambda>\lambda_{0}$ and a bounded linear feedback $K\in \mathcal{L}(\textcolor{black}{\mathcal{H}^{\vec{\alpha}/2}},\mathbb{C}^m)$ and a mapping $T
$ that is an isomorphism from $\mathcal{H}^{\textcolor{black}{\vec r}}$ to itself \textcolor{black}{for any $\vec r=(r_{1},..., r_{m})$} with \textcolor{black}{$r_{i}\in(1/2-\alpha_{i}+\gamma_{i},\alpha_{i}-1/2-\gamma_{i})$} such that $T$ maps the system
\begin{equation}
    \partial_t u=\mathcal{A}u+BK(u)\label{closed-loop-system}
\end{equation} to the system
\begin{equation}
    \partial_t v=\mathcal{A}v-\lambda v.\label{target-system}
\end{equation}
In particular, for any $\mu>0$, $\lambda_{0}$ can be chosen sufficiently large such that the closed loop system \eqref{closed-loop-system} is exponentially stable with decay rate $\mu>0$ in $\mathcal{H}^{\vec r}$ for any 
\textcolor{black}{$\vec r=(r_{1},..., r_{m})$} with \textcolor{black}{$r_{i}\in(1/2-\alpha_{i}+\gamma_{i},\alpha_{i}-1/2-\gamma_{i})$}.
\end{theorem}

\begin{rmk}[Regularity of the feedback operator]
While the linear feedback $K$ belongs \emph{a priori} to $\mathcal{L}(\textcolor{black}{\mathcal{H}^{\vec{\alpha}/2}},\mathbb{C}^m)$, the feedback we construct is actually more regular and can actually be extended on $\mathcal{L}(\textcolor{black}{\mathcal{H}^{1/2+\varepsilon}},\mathbb{C}^{m})$ for any $\varepsilon>0$. (see \eqref{K_n-expression} and Lemma \ref{Lemma-K_n})
\end{rmk}

In fact, Theorem \ref{Theo_main} can be generalized as follows:
\begin{theorem}\label{Theo_main_general}
Consider the control system \eqref{eq:sys0} with operator $\mathcal{A}$ satisfying the assumptions $(A_{1})$ \eqref{ln+1_cond} and \eqref{ln-lp_cond}. Let $\vec{\beta}=(\beta_1,\cdots,\beta_m)\in\mathbb{R}^m$, $(\gamma_{1},...,\gamma_{m})\in [0,(\alpha_{1}-1)/2)\times ...\times [0,(\alpha_{m}-1)/2)$, and $B=(B_1,\cdots,B_m)\in\mathcal{H}_{\textcolor{black}{1}}^{\beta_1-\frac{\alpha_1}{2}}\times \cdots\times \mathcal{H}_{\textcolor{black}{m}}^{\beta_m-\frac{\alpha_m}{2}}$ such that there exist $c_1, c_2>0$ satisfying
\begin{align}
    c_1n^{-\beta_i}\leq |\langle B_i,\widetilde{\varphi_n}^i\rangle|\leq c_2 n^{-\beta_i+\gamma_{i}}, \quad \forall n\in \mathbb{N}^*,\ i\in\{1,\cdots,m\}.\label{control-cond-general}
\end{align}
 Then, for any $\lambda_0>0$ there exist $\lambda>\lambda_{0}$ and a bounded linear feedback $K\in \mathcal{L}(\textcolor{black}{\mathcal{H}^{\vec{\beta}+\vec{\alpha}/2}},\mathbb{C}^m)$ and a mapping $T$ that is an isomorphism from $\mathcal{H}^{\textcolor{black}{\vec{\beta}+}\vec r}$ to itself with 
\textcolor{black}{$\vec{r} := (r_{1},...,r_{m})$} for any
 \begin{align}
 \label{eq:condr}
     r_i\in(1/2-\alpha_i+\gamma_{i},\alpha_i-1/2-\gamma_{i}).
 \end{align}
Moreover $T$ maps the system
\begin{equation*}
    \partial_t u=\mathcal{A}u+BK(u)
\end{equation*} to the system
\begin{equation*}
    \partial_t v=\mathcal{A}v-\lambda v.
\end{equation*}
In particular, for any $\mu>0$, $\lambda_{0}$ can be chosen sufficiently large such that the closed loop system \eqref{closed-loop-system} is exponentially stable with decay rate $\mu$ in $\mathcal{H}^{\textcolor{black}{\vec{\beta}+}\vec r},$ for any $\vec r$ satisfying \eqref{eq:condr}. 
\end{theorem}

We give some illustrations from this Theorem in Sections \ref{sec:schro}--\ref{sec:diffusion}: 
an application to the Schrödinger equation is considered and the water wave equations, and an extension of the linear Laplacian operator studied in \cite{Gagnon2022-fredholm-laplacien}.\\
\textcolor{black}{\begin{rmk}
    Theorem \ref{Theo_main} and Theorem \ref{Theo_main_general} do not only provide the existence of feedback law $K$ such that the closed-loop $\partial_t u=\mathcal{A}u+BK(u)$ is exponentially stable, but they ensure the existence of transformation $T$ which transforms any solution $u$ of the closed-loop into exponentially decaying $v$ where $\partial_t v=\mathcal{A}v-\lambda v.$ This yields the following observations
    \begin{itemize}
        \item F-equivalence can be extended to ensure stability of some quasi-linear systems (see \cite{Boulard2025f} for parabolic systems, for generic systems the question is largely open) and also some nonlinear systems (see subsection \ref{sec-nonlinear}). 
        \item Instead of relying on a frequency-domain criterion involving the resolvent of the operator $\mathcal{A}+BK,$ the concept of $F$-equivalence offers the advantage of establishing the exponential stability of the semigroup generated by $\mathcal{A}+BK$ through the transformation $T$ and the stability properties of the shifted operator $\mathcal{A}-\lambda I,$ for sufficiently large $\lambda>0$ (see Section \ref{sec:wellposed}) and, in turn, obtain a more explicit construction of $K$ which brings hope of being able to provide explicit estimates (in particular quantitative estimates on $K$ with respect to the decay rate $\lambda$).
        \item Moreover since $u=Tv$ with
    \begin{align*}
    \|v(t)\|_{X}\leq C\|v_0\|e^{-\mu t},\quad \forall t\geq 0,
    \end{align*}
    we derive
    \begin{align*}
    \|u(t)\|_{X}\leq C\|T\|\|T^{-1}\|\|u_0\|_Xe^{-\mu t},\quad \forall t\geq 0,
    \end{align*}
    where $C$ only depends on $\mathcal{A}$ (and in particular not on $\lambda$ or $K$). A recent preprint \cite{Gagnon2025abstract} suggests that the $K$ we construct with this approach (see Lemma \ref{lem:regk} and Section \ref{sec:step4}) can be given by an explicit formula leading to sharp estimates on $\|T\|$ and $\|T^{-1}\|$ with respect to the decay rate $\lambda$. This would facilitate numerical simulations, error estimates and could lead to finite-time stabilization results. 
    \end{itemize}   
\end{rmk}}

In the particular case where $\mathcal{A}$ is skew-adjoint operator, this coincides with the result of \cite{Gagnon2022-fredholm}  but with strictly less restrictive conditions
\textcolor{black}{(they coincide with the condition of \cite{Gagnon2022-fredholm} when $\gamma=0$)}. One can also recover the result on the linearized Schrödinger equation of \cite{Coron2018-rapid}, here again with less restrictive conditions of stabilization since \cite[Hypothesis 1.1]{Coron2018-rapid} is not necessarily needed depending on the regularity of the control operator), \textcolor{black}{see Section \ref{sec:schro}}. \\
 
From Theorem \ref{Theo_main_general} we deduce the following conditions for the rapid stabilization of system \ref{eq:sys0}.
\begin{corollary}[\textcolor{black}{Rapid stabilization}]
\label{cor:expstab}
Consider the control system \eqref{eq:sys0} with operator $\mathcal{A}$ satisfying the assumptions $(A_{1})$  \eqref{ln+1_cond} and \eqref{ln-lp_cond}.  Let $(\gamma_{1},...,\gamma_{m})\in [0,(\alpha_{1}-1)/2)\times ...\times [0,(\alpha_{m}-1)/2)$, $r_{i}\in(1/2-\alpha_{i}+\gamma_{i},\alpha_{i}-1/2-\gamma_{i})$, and $B=(B_1,\cdots,B_m)\in\mathcal{H}_{\textcolor{black}{1}}^{-\alpha_1}\times \cdots\times \mathcal{H}_{\textcolor{black}{m}}^{-\alpha_m}$ such that there exist $c_1, c_2>0$ satisfying
\begin{equation}
\label{eq:condbsup}
   c_1n^{r_{i}}
   \leq |\langle B_i,\widetilde{\varphi_n}^i\rangle|\leq c_2 n^{r_{i}+\gamma_{i}},
\end{equation}
then for any $\mu>0$ there exists a (constructive) bounded linear feedback $K_{\gamma}\in \mathcal{L}(\textcolor{black}{\mathcal{H}^{\vec{\alpha}/2-\vec{r}}},\mathbb{C}^m)$ such that $K_{\gamma}\in \mathcal{L}(\mathcal{H}^{\textcolor{black}{1/2+\varepsilon-\vec{r}}},\mathbb{C}^m)$ for any $\varepsilon>0$, \textcolor{black}{where
$1/2+\varepsilon-\vec{r} := (1/2-\varepsilon -r_{1},...,1/2-\varepsilon-r_{m})$},
and the system \eqref{closed-loop-system} is exponentially stable in $\textcolor{black}{X}$ with decay rate $\mu$.
\end{corollary}

\textcolor{black}{
\begin{rmk}
    The conditions on the eigenvalues are asymptotic (see \eqref{ln+1_cond},\eqref{ln-lp_cond}), and  
    the knowledge of the asymptotic behaviors may suffice (Corollary \ref{corollary-diffusion}, Corollary \ref{cor-Gribov}). However, we need to know at least partially the value of $\langle B_i, \widetilde{\varphi_n}^i\rangle:$ the precise estimate required here on the growth of $\langle B_i, \widetilde{\varphi_n}^i\rangle$ is also asymptotic, but we need to ensure that it is non-zero for all $n$. 
    Nevertheless, the fact that it is non-zero for every $n$ means that \textcolor{black}{what} we really require is approximate controllability, which is essential for stabilizing a system without any further assumption (e.g. parabolicity). So while this does impose restrictions, it would be hard to do without it in any case.
\end{rmk}
}

\subsection{Link with controllability and existing conditions for stabilizability}
\label{sec:discussion}

Usually, a typical ``good'' condition for a generic rapid stabilization result (in $\textcolor{black}{X}$) would be admissibility of the control operator $B$ (in $\textcolor{black}{X}$), and exact null controllability (in $\textcolor{black}{X}$), sometimes reformulated as an observability condition.
These are the conditions 
typically required when using Riccati or Gramian methods for rapid stabilization (see \cite{slemrod1972linear,komornik1997rapid,urquiza2005rapid,LT1,LT2}).
The result of \cite[Proposition 1]{trelat2019characterization} using an observability approach to investigate rapid stabilizability in a general framework with also 
uses these as assumptions (see also \cite{slemrod1974note,nguyen2024stabilization} in the case where  $\mathcal{A}$ generates a $C^{0}$ group and not only a semigroup). 

In our case, because we require $BK$ to shift all the spectrum\footnote{since the system should be mapped to $\partial_{t}v = (\mathcal{A}-\lambda Id) v$} of $\mathcal{A}$ our statement is stronger than the sole rapid stabilization and we could expect to require a stronger condition\footnote{In fact our problem is closer to the pole placement problem in infinite dimension for which the conditions are usually stronger than the condition of rapid stabilization.}.
Surprisingly, our condition \eqref{eq:condbsup} can in fact be less restrictive: 
in our framework the system is not necessarily exactly null controllable in $\textcolor{black}{X}$ and $B$ is not necessarily admissible \textcolor{black}{in $X$}.
This is particularly striking when the system is skew-adjoint: in this case our sufficient conditions \eqref{eq:condbsup} are
less restrictive than those of 
the existing results so far. In particular, the system \eqref{eq:sys0} can be stabilized rapidly in $\textcolor{black}{X}$ even if $B$ is not admissible (in $\textcolor{black}{X}$) and the system is not exactly (null) controllable in \textcolor{black}{the state space $X$ (but only in a more regular space)}.\\ 

To go more in detail, let us recall that being admissible (in $\textcolor{black}{X}$) means that there exists $T>0$ and $C_{T}>0$ such that for every $z\in D(\mathcal{A}^{*})$ (see \cite[Section 2.3]{coron2007control} or \cite{tucsnak2009observation})
\begin{equation}
\label{eq:admissible0}
\int_{0}^{T}|B^{*}S(t)^{*}z|^{2}\leq C_{T}\|z\|_{\textcolor{black}{X}}^{2},\text{  \emph{(admissibility condition)} }
\end{equation}
where $\mathcal{A}^{*}$ is the adjoint of $\mathcal{A}$ and $S(t)^{*}$ is the adjoint of $S(t)$, where $\{S(t)\}_{t\geq0}$ is the $C^{0}$ semigroup generated by $\mathcal{A}$. On the \textcolor{black}{other hand} being exactly controllable in $\textcolor{black}{X}$ at time $T$ is equivalent to the existence of $c_{T}$ such that
\begin{equation}
\label{eq:exactcont}
    \int_{0}^{T}|B^{*}S(t)^{*}z|^{2}\geq c_{T}\|z\|_{\textcolor{black}{X}}^{2},\text{  \emph{(exact controllability)}}.
\end{equation}
In the particular case of skew-adjoint systems, \eqref{eq:exactcont} is also equivalent to exact null controllability.
In our framework, since $\mathcal{A}$ has a Riesz basis of eigenvectors we have
the following results from \cite{weiss2011eigenvalues} (see also \cite{russell1994general})
\begin{lem}[\cite{weiss2011eigenvalues}]
\label{lem:weiss}
If $B$ is admissible then there exists $C>0$ such that
\begin{equation}
\label{eq:admissible}
|\langle  B_i,\textcolor{black}{\widetilde{\varphi_{n}}}^{i} \rangle|\leq C(1+|\text{Re}(\lambda_{n}^{i})|)^{1/2},\;\forall\;(i,n)\in\{1,...,m\}\times\mathbb{N}^{*},
\end{equation}
and if the system is in addition exactly controllable in $\textcolor{black}{X}$ then there exists $c>0$ such that
\begin{equation}
\label{eq:exactcontspect}
|\langle  B_i,\textcolor{black}{\widetilde{\varphi_{n}}}^{i} \rangle|\geq c(1+|\text{Re}(\lambda_{n}^{i})|)^{1/2},\;\forall\;(i,n)\in\{1,...,m\}\times\mathbb{N}^{*}.
\end{equation}
\end{lem}
These conditions are necessary (and not \textcolor{black}{always} sufficient \cite{hansen1997new} in general\footnote{Note that following \cite{weiss2011eigenvalues} \eqref{eq:exactcontspect} is necessary and sufficient for exact controllability if $(\lambda_{n}^{i})_{n,i}$ are \emph{properly spaced} (see \cite[Proposition 3.4]{weiss2011eigenvalues}), but they are not always here.}) 
to admissibility and exact controllability \textcolor{black}{in $X$}. \textcolor{black}{Note that in the case of a skew-adjoint operator, there is an orthonormal basis of eigenvectors and one can assume that $\widetilde{\varphi_{n}} = \varphi_{n}$.}
Comparing with our condition \eqref{eq:condbsup} we see the following:
\begin{itemize}
\item If the $\lambda_{n}^{i}$ have bounded real part (for instance if $\mathcal{A}$ is a skew-adjoint operator, but not only). 
\begin{itemize}
\item For $r_{i}=0$, $\gamma_{i}=0$ we recover the conditions of admissibility \textcolor{black}{and exact controllability in $X$} which coincide with the regularity and observability conditions of \cite{komornik1997rapid,urquiza2005rapid,vest2013rapid,nguyen2024stabilization} (and \cite{Coron2018-rapid,nguyen2024rapid} for the bilinear Schrödinger equation).
\item For \textcolor{black}{$r_{i}\in(0,\alpha_{i}-1/2-\gamma_i)$}, $B$ is not necessarily admissible \textcolor{black}{in $X$}, in particular the regularity condition required in \cite{vest2013rapid,urquiza2005rapid,komornik1997rapid,nguyen2024stabilization} is not satisfied. 
\item For \textcolor{black}{$r_{i}\in(1/2-\alpha_{i}+\gamma_{i},0)$}, then 
\textcolor{black}{the system is not necessarily exactly (null) controllable in the state space $X$ and for $r_{i}\in(1/2-\alpha_{i}+\gamma_{i},-\gamma_{i})$ it cannot be exactly (null) controllable}. This could seem surprising given the usual result  given in \cite[Theorem 16.5]{zabczyk2020mathematical} (see also \cite{trelat2019characterization}) according which rapid stabilizability \textcolor{black}{in $X$} implies exact (null) controllability \textcolor{black}{in $X$} for skew-adjoint systems. In fact, since in these references, the definition of rapid stabilizability additionally requires that $K\in\mathcal{L}(X,\mathbb{C}^{m})$, 
which is helpful in the analysis, in particular to ensure that the system is well-posed, but is not assumed here (compared to the case $r_{i}\in(1/2,\alpha_{i}-1/2)$): we only require the weaker condition that $\mathcal{A}+BK$ generates a $C^{0}$ semigroup on $\textcolor{black}{X}$. \textcolor{black}{This allows to have a weaker assumption on the controllability of the system: note that in this case the system is still exactly controllable in a stronger space $\mathcal{H}^{-\vec{r}}$ (recall that $r_{i}<0$ in this case).} This less restrictive and 
relaxed
definition \textcolor{black}{of stabilizability} can also be found in  \cite{liu2022characterizations,ma2023feedback} where the authors provide very nice necessary conditions in terms of observability when $\mathcal{A}$ is a skew-adjoint operators (but under the assumption that $B$ is admissible, which is not the case here).\\
\end{itemize}

\noindent \textcolor{black}{Note that both $r_{i}\in(1/2-\alpha_i+\gamma_{i},0)$ and $r_{i}+\gamma_{i}\in(0,\alpha_{i}-1/2)$ can be satisfied at the same time. As a consequence, our sufficient condition for rapid stabilization \eqref{eq:condbsup} can hold even in cases where the system is \textbf{not} exactly controllable in \textcolor{black}{the state space $X$} \textbf{and} $B$ is not admissible.}\\

\item If $\text{Re}(\lambda_{n}^{i})\geq c |\lambda_{n}|$ for some constant $c>0$ (for instance if $\mathcal{A}$ is a self-adjoint operator, but not only)
\begin{itemize}
\item If $r_{i}+\gamma_{i}\in(\alpha_{\textcolor{black}{i}}/2,\alpha_{\textcolor{black}{i}}-1/2)$ then 
$B$ does not necessarily need to be admissible.
\end{itemize}
\end{itemize}

Of course, if $B$ is admissible and the system exactly controllable, then our condition on $B$ is always satisfied, as one could expect.

\subsection{Application to Schrödinger equation around the ground state}
\label{sec:schro}

We consider the bilinear Schrödinger equation linearized around the ground state studied in \cite{Coron2018-rapid,nguyen2024rapid} (see also \cite{beauchard2010local,beauchard2014local} and \cite{machtyngier1994stabilization,dehman2006stabilization} for earlier works):
\begin{equation}
\label{eq:linearschro}
    \begin{split}
        &i\partial_{t}\Psi = -\Delta \Psi-\textcolor{black}{\sigma_{1}}\Psi+u(t)\mu(x)\textcolor{black}{\Phi_{1}}(x),\\
        &\Psi(t,0)=\Psi(t,1)=0,
    \end{split}
\end{equation}
where $\textcolor{black}{\sigma_{1}}=\pi^{2}$ and $\textcolor{black}{\Phi_{1}} = \sqrt{2}\sin(\pi x)$ are respectively the first eigenvalue and eigenvector of the Laplace operator with Dirichlet boundary conditions. The total family of eigenvalues and eigenvectors are the solution of 
\begin{equation}
\begin{split}
    &-\Delta \textcolor{black}{\Phi_{n}} = \textcolor{black}{\sigma_{n}}\textcolor{black}{\Phi_{n}}\\
    &\textcolor{black}{\Phi_{n}}(0) = \textcolor{black}{\Phi_{n}}(1) = 0,
    \end{split}
\end{equation}
and is given by
\begin{equation}
\textcolor{black}{\sigma_{n}}=\pi^{2}n^{2},\;\;\textcolor{black}{\Phi_{n}}(x) = \sqrt{2}\sin(n\pi x),\;\forall n\in\mathbb{N}^{*}.
\end{equation}
Here, the space under consideration is, as in \cite{Coron2018-rapid,nguyen2024rapid}
\begin{equation}
\textcolor{black}{X}=\{f\in H^{3}((0,1);\mathbb{C}),\ f(0)=f(1)=f''(0)=f''(1)=0,\; \text{Re}\langle f,\textcolor{black}{\Phi_{1}}\rangle_{L^{2}(0,1)} = 0\},
\end{equation}
One readily checks that the associated operator $\mathcal{A} = i(\Delta +\textcolor{black}{\sigma_{1}}Id)$ is diagonal with eigenvectors 
$\textcolor{black}{\varphi_{n}} = (\pi n)^{-3}\textcolor{black}{\Phi_{n}}$ associated to eigenvalues $\lambda_{n} = -i(\textcolor{black}{\sigma_{n}-\sigma_{1}})$, for $n\in\mathbb{N}^{*}$ on $\textcolor{black}{X}$ 
and satisfy \eqref{ln+1_cond}--\eqref{ln-lp_cond} with $\alpha=2$. \textcolor{black}{The space $X$} is naturally equipped with the following inner product and associated norm (see \cite{Coron2018-rapid}): 
\begin{equation}
\begin{split}
\langle f, g\rangle_{\textcolor{black}{X}}:= \sum\limits_{n\in\mathbb{N}^{*}}
\textcolor{black}{\sigma_{n}^{3}}\langle f,\textcolor{black}{\Phi_{n}}\rangle_{L^{2}}\overline{\langle g,\textcolor{black}{\Phi_{n}}\rangle_{L^{2}}},
\end{split}
\end{equation}
which, in fact, corresponds to the $H^{3}$ norm (see \cite[Section 1]{nguyen2024rapid}).

One can check that $(\textcolor{black}{\varphi_{n}})_{n\in\mathbb{N}^{*}}$ is an orthonormal basis of $\textcolor{black}{X}$, indeed
\begin{equation}
\textcolor{black}{\langle \varphi_{n},\varphi_{k}\rangle_{\textcolor{black}{X}}} = \sum\limits_{m\in\mathbb{N}^{*}}\textcolor{black}{\sigma_{m}^{3}}\pi^{-6}n^{-3}k^{-3}
\langle \textcolor{black}{\Phi_{n},\Phi_{m}\rangle_{L^{2}}\overline{\langle \Phi_{k},\Phi_{m}\rangle_{L^{2}}}}
=\delta_{n,k}\textcolor{black}{\sigma_{n}^{3}}\pi^{-6}n^{-6},
\end{equation}
thus for $k\neq n$, $\textcolor{black}{\langle \varphi_{n},\varphi_{k}\rangle_{\textcolor{black}{X}}}=0$ and $\|\textcolor{black}{\varphi_{n}}\|_{\textcolor{black}{X}}=1$. Finally, let $f\in \textcolor{black}{X}$, since $\textcolor{black}{X}\subset L^{2}(0,1)$ and $(\textcolor{black}{\Phi_{n}})_{n\in\mathbb{N}^{*}}$ is a basis of $L^{2}(0,1)$ there exists $\textcolor{black}{(f_{n})_{n\in \mathbb{N}^{*}}}$ such that
\begin{equation}
f = \sum\limits_{n\in\mathbb{N}^{*}} f_{n} \textcolor{black}{\Phi_{n}} = \sum\limits_{n\in\mathbb{N}^{*}} \textcolor{black}{\sigma_{n}^{3/2}}f_{n} \varphi_{n},
\end{equation}
and since $f\in \textcolor{black}{X}$, $(\textcolor{black}{\sigma_{n}^{3/2}}f_{n})_{n\in \mathbb{N}^{*}}\in l^{2}$.\\

In \cite{Coron2018-rapid,nguyen2024rapid} it was \textcolor{black}{shown} 
that the system is rapidly stabilizable, i.e. for any $\lambda>0$ there exists a feedback operator $K$ such that the system \eqref{eq:linearschro} with $u(t) = K(\Psi(t))$ is exponentially stable with decay rate $\lambda$ \textcolor{black}{in $X$}, provided that the following sufficient condition is satisfied

\begin{equation}
\label{eq:condexisting}
\begin{split}
    &\mu\text{ belongs to }H^{3}(0,1)\\
    &|\langle\mu\textcolor{black}{\Phi_{1}},\textcolor{black}{\Phi_{n}}\rangle_{L^{2}(0,1)}|\geq c n^{-3},\;\forall n \in\mathbb{N}^{*}
    \end{split}
\end{equation}
for some $c>0$.
This implies in particular (see \cite[Remark 1.2]{Coron2018-rapid}) that 
\begin{equation}
\label{eq:rmk12}
    c n^{-3}\leq |\langle\mu\textcolor{black}{\Phi_{1}},\textcolor{black}{\Phi_{n}}\rangle_{L^{2}(0,1)}|\leq C n^{-3},
\end{equation}
for some $c,C >0$ (that might change between lines but are independent on $n$). The left inequality is equivalent to the exact (null) controllability of the system \eqref{eq:linearschro} \textcolor{black}{(in $X$)} \cite{Coron2018-rapid}. Since $(\textcolor{black}{{\varphi}_{n}})_{n\in\mathbb{N}^{*}}$ is a Riesz (and in fact orthonormal) basis of $\textcolor{black}{X},$ \textcolor{black}{the condition \eqref{eq:rmk12} is equivalent to:}
\begin{equation}
    c \leq |\langle\mu\textcolor{black}{\Phi_{1}},\textcolor{black}{{\varphi}_{n}}\rangle_{\textcolor{black}{X}}|\leq C,
\end{equation}
which corresponds to our condition \eqref{eq:condbsup} with $\gamma=\beta=0$. From Corollary \ref{cor:expstab} we can relax the condition \eqref{eq:condexisting} of \cite{Coron2018-rapid,Gagnon2022-fredholm} and show that the system \eqref{eq:linearschro} with $u(t) = K(\Psi(t))$ is rapidly stabilizable \textcolor{black}{in $X$} as soon as there exists $\gamma\in [0,1/2)$ and $r\in(-3/2+\gamma,3/2-\gamma)$ such that
\begin{equation}
        c n^{r} \leq |\langle\mu\textcolor{black}{\Phi_{1}},\textcolor{black}{\varphi_{n}}\rangle_{\textcolor{black}{X}}|\leq C n^{r+\gamma}.
\end{equation}
In particular, a sufficient condition for rapid stabilization is \textcolor{black}{now}
\begin{equation}
\begin{split}
    &\mu\text{ belongs to }H^{3}(0,1)\\
    &|\langle\mu\textcolor{black}{\Phi_{1}},\textcolor{black}{\Phi_{n}}\rangle_{L^{2}(0,1)}|\geq c n^{-7/2+\varepsilon},\;\forall n \in\mathbb{N}^{*}, \text{ for some }\varepsilon>0,
\end{split}
\end{equation}
\textcolor{black}{
\noindent\textbf{Stabilization in weaker spaces} In \cite{Coron2018-rapid,nguyen2024rapid} the space considered is $\textcolor{black}{X}\subset H^{3}(0,1)$ because the system was shown to be exactly (null) controllable  under the assumption \eqref{eq:condexisting} (see \cite{beauchard2010local}) and if $\mu$ is regular the system cannot be exactly controllable in weaker spaces such as $H^{1}_{0}(0,1;\mathbb{C})$ or $H^{2}(0,1;\mathbb{C})\cap H^{1}_{0}(0,1;\mathbb{C})$ (this is a consequence of \eqref{eq:rmk12} and \ref{lem:weiss})}.
However, since with our result we do not need anymore the exact controllability to be able to stabilize the system, we can consider the stabilisation in weaker spaces. In particular, under the condition \eqref{eq:condexisting} of \cite{beauchard2010local,Coron2018-rapid,nguyen2024rapid} the system is rapidly stabilizable in $H^{2}(0,1;\mathbb{C})\cap H^{1}_{0}(0,1;\mathbb{C})$: \textcolor{black}{let us now set $X:= L^{2}(0,1)$, the following corollary holds}
\begin{corollary}
Assume that \eqref{eq:condexisting} holds, then
    for any $\lambda>0$ there exists $\mathcal{K}\in \mathcal{L}(\mathcal{H}^{7/2+\varepsilon}(0,1);\mathbb{C})$ for any $\varepsilon>0$ such that the system \eqref{eq:linearschro} 
    \textcolor{black}{with $u(t)=\mathcal{K}(\Psi(t))$} is exponentially stable in $\mathcal{H}^{r}$ for any $r\in(3/2,9/2)$, \textcolor{black}{where}
        \begin{equation}
        \label{eq:mathcalHsschro}
        \mathcal{H}^{s} :=\{f \in H^{s}(0,1;\mathbb{C})\;|\;\sum\limits_{n\in\mathbb{N}^{*}}|\langle f,\textcolor{black}{\Phi_{n}} \rangle_{L^{2}(0,1)}|^{2}n^{2s}<+\infty\}.
    \end{equation}
     \textcolor{black}{In} particular it is exponentially stable in $H^{2}(0,1;\mathbb{C})\cap H^{1}_{0}(0,1;\mathbb{C})$.
\end{corollary}
\begin{proof}
It suffices to use Theorem \ref{Theo_main_general} and to note that $\mathcal{H}^{2} = H^{2}(0,1;\mathbb{C})\cap H^{1}_{0}(0,1;\mathbb{C})$ (see \cite[Section 1.1]{Coron2018-rapid} for this last point, noting that $\mathcal{H}^{2}$ in our case corresponds exactly to $H_{(0)}^{2}$ in \cite{Coron2018-rapid}).
\end{proof}
In general we have 
\begin{corollary}
    Assume that there exists $\beta\in \mathbb{R}$ and $\gamma\in[0,1/2)$ such that
    \begin{equation}
        c n^{-\beta}\leq |\langle \mu\textcolor{black}{\Phi_{1}}, \textcolor{black}{\Phi_{n}}\rangle_{L^{2}(0,1)}|\leq C n^{-\beta+\gamma}, \text{ for any }n\in\mathbb{N}^{*},\label{co}
    \end{equation}
    for some positive constants $c,C>0$. 
    Then for any $\lambda>0$ there exists an (explicitly computable) linear feedback $\mathcal{K}\in \mathcal{L}(\mathcal{H}^{\beta+1/2+\varepsilon},\mathbb{C})$ for any $\varepsilon>0$ such that the system \eqref{eq:linearschro} \textcolor{black}{with $u(t)=\mathcal{K}(\Psi(t))$} is exponentially stable in $\mathcal{H}^{\beta+r}$ for any $r\in(-3/2+\gamma,3/2-\gamma)$ with decay rate at least $\lambda$ where $\mathcal{H}^{s}$ is given by \eqref{eq:mathcalHsschro}
\end{corollary}

\textcolor{black}{
\begin{rmk}[Comparison with \cite{Gagnon2022-fredholm}]
    Since the operator involved in the example is skew-adjoint, the results in \cite{Gagnon2022-fredholm} can be also invoked and the obtained condition would be \eqref{co} with $\gamma=0.$ This makes our condition more relaxed than the one proposed in \cite{Gagnon2022-fredholm}. 
\end{rmk}
}

\begin{rmk}
    Note that in \cite{nguyen2024rapid} the author uses a very different method and shows an impressive quantitative estimate of $K$ with respect to $\lambda$, which allows a finite-time stabilization of this system. Since the $F$-equivalence constructs relatively explicitly the control $K$ and the isomorphism $T$, it would be interesting to see if it is possible to obtain in our case a quantitative estimate on $K$ and $T$ and, in turn, a finite time stabilization.
\end{rmk}

\subsection{Application to parabolic systems}
\label{sec:diffusion}
\subsubsection{Heat equation}\label{sec:diff1}
Let us first consider the heat equation on a torus $\mathbb{T} = \mathbb{R}/2\pi\mathbb{Z}$ with two scalar controls
\begin{equation}
\label{eq:heat}
    \partial_{t} u = \Delta u + \phi w(t),
\end{equation}
where $\phi = (\phi_{1},\phi_{2})^{T}$ such that $\phi_1$ is even and $\phi_{2}$ is odd and $w(t)\in\mathbb{R}^{2}$ (the system is not controllable in $H^{m}(\mathbb{T})$ with a single control, see \cite{Gagnon2022-fredholm-laplacien}) . The space of consideration is $H^{m}(\mathbb{T})$ and the eigenvectors and associated eigenvalues of the operator $\mathcal{A} = \Delta$ are
\begin{equation}\label{phi-n-i-heat}
    \lambda_{n} = -n^{2},\;\varphi_{n}^{1}(x) = n^{-m}c_{1,n}\sin(n x),\;\varphi_{n+1}^{2}(x) = n^{-m}c_{2,n}\cos(n x), \forall n\in\mathbb{N}^{*},
\end{equation}
and the constant function $\varphi_{0}^{2}(x) = (2\pi)^{-1}$ associated to the eigenvalue $0$. Note that all eigenvalues have multiplicity 2 except $0$. In \cite{Gagnon2022-fredholm-laplacien} it is shown that for arbitrarily large $\lambda>0$ there exists a feedback operator $K\in \mathcal{L}(H^{m+1/2+\varepsilon};\mathbb{R}^{2})$ for any $\varepsilon>0$ such that this system with $w(t)=K(u(t))$ can be invertibly mapped to the exponentially stable system 
\begin{equation}
\label{eq:heatstab}
    \partial_{t}u = \Delta u-\lambda u,
\end{equation}
in $\mathbb{H}^{m+r}(\mathbb{T})$ for any $r\in(-1/2,1/2)$ provided that the following sufficient condition holds:
\begin{equation}
    c \leq |\langle \phi_{i},\varphi_{n}^{i}\rangle_{H^{m}(\mathbb{T})}|\leq C,\;\forall (i,n)\in \{1,2\}\times\mathbb{N}^{*}
\end{equation}
We recover the same with our Theorem \ref{Theo_main_general}, but with a weaker condition:
\begin{corollary}
\label{cor:heat}
    Let $\gamma\in[0,1/2)$, if
    \begin{equation}
    c \leq |\langle \phi_{i},\varphi_{n}^{i}\rangle_{H^{m}(\mathbb{T})}|\leq C n^{\gamma},\;\forall (i,n)\in \{1,2\}\times\mathbb{N}^{*},
    \end{equation}
    then for arbitrarily large $\lambda>0$ there exists a feedback operator $K\in \mathcal{L}(H^{m+1/2+\varepsilon};\mathbb{R}^{2})$ such that the system \textcolor{black}{\eqref{eq:heat}} with $w(t)=K(u(t))$ can be invertibly mapped to the exponentially stable system \eqref{eq:heatstab} in $H^{m+r}(\mathbb{T})$, and in particular it is exponentially stable \textcolor{black}{in $H^{m+r}(\mathbb{T})$}, for any $r\in(-3/2+\gamma,3/2-\gamma)$.
\end{corollary}
\subsubsection{A nonlinear system: \textcolor{black}{Burgers equation}}\label{sec-nonlinear}
Our result can in fact generalize to a nonlinear (semilinear) system such as the \textcolor{black}{Burgers} equation. Let us consider
\begin{equation}
\label{eq:burgers}
    \partial_{t} u = \Delta u - u\partial_{x}u+\phi w(t),
\end{equation}
where $\phi=(\phi_{1},\phi_{2})^{T}$ is taken as in the previous example. The space of consideration is $L^{2}(\mathbb{T})$ and we consider the functions $(\varphi_n^i)$ defined in \eqref{phi-n-i-heat} and get the following
\begin{corollary}
\label{cor:burgers}
Let $\gamma\in [0,1/2)$, and $r\in (-1/4,1/2-\gamma).$ If  
    \begin{equation}
    \label{eq:condburgers}
    c n^r\leq |\langle \phi_{i},\varphi_{n}^{i}\rangle_{L^{2}(\mathbb{T})}|\leq C n^{r+\gamma},\;\forall (i,n)\in \{1,2\}\times\mathbb{N}^{*},
    \end{equation}
    then for any arbitrary large $\lambda>0$ there exists a feedback operator $K\in \mathcal{L}(H^{3/4};\mathbb{R}^{2})$ (and in fact $K\in\mathcal{L}(H^{1/2-r+\varepsilon};\mathbb{R}^{2})$ for any $\varepsilon>0$) such that the nonlinear system \eqref{eq:burgers} with $w(t)=K(u(t))$ is locally exponentially stable in $L^{2}(\mathbb{T})$.
\end{corollary}
This can be shown leveraging the $F$-equivalence on the linear system and using the same on the nonlinear system exactly as in \cite{Gagnon2022-fredholm-laplacien}. We give a proof in Appendix \ref{sec:appendix-burgers}. To our knowledge this stabilization result is new and is an improvement of the results of \cite{Gagnon2022-fredholm-laplacien} (which correspond to the particular case of Corollary \ref{cor:burgers} when $r=\gamma=0$). 

\subsubsection{General diffusion equation}
Our Theorem \ref{Theo_main_general} also allows to generalize this to the classical diffusion equation:
\begin{equation}
\label{eq:sysdiff1}
    \partial_{t}u=\partial_{x}(a\partial_{x}u) +b u+\phi K(u),\;\;\text{ on }\ [0,+\infty)\times [0,L],
\end{equation}
\begin{equation}
\begin{split}
\label{eq:bcsturm}
    c_{1}u(t,0) + c_{2}\partial_{x}u(t,0)=0\\
    c_{3}u(t,L) + c_{4}\partial_{x}u(t,L)=0,
    \end{split}
\end{equation}
where 
$a\in C^{2}([0,L];\mathbb{R}_{+}^*)$, $b\in L^{2}(0,L)$,
and
 $c_{1}^{2}+c_{2}^{2}>0$, $c_{3}^{3}+c_{4}^{2}>0$.   
This framework could not be considered with the $F$-equivalence until now, while stabilization results \textcolor{black}{were} known with other approaches for some time: this system was for instance considered with Dirichlet boundary condition and $\phi$ a bounded operator in the reference space $\textcolor{black}{X}$ (here $\textcolor{black}{X}=L^{2}(0,L)$) in \cite{banks1984linear} and the authors obtained a stabilization with an optimal feedback using a Riccati approach. 
In particular, the following holds 
\begin{corollary} \label{corollary-diffusion}
    If there exist $c_1, c_2>0$ and $\gamma\in[0,1/2)$ such that $\phi$ satisfies
 \begin{equation}
 \label{eq:cond-cor}
c_1\leq \left|\int_0^L\phi(x)u_n(x)dx\right|\leq c_2 n^{\gamma}, \quad \forall n\in \mathbb{N}^*,
 \end{equation} where $(u_n)_{n\in \mathbb{N}^*}$ is the sequence of eigenvectors of the operator $\mathcal{A}:=\partial_x(a\partial_x)+b,$
then for any $\lambda\in (0,+\infty),$ there exists a bounded linear feedback $K\in \mathcal{L}(H^{1/2+\varepsilon}(0,L),\mathbb{C})$ for any $\varepsilon>0$, such that the diffusion \textcolor{black}{equation} \eqref{eq:sysdiff1}--\eqref{eq:bcsturm} is exponentially stable in $H^r(0,L)$ with decay rate $\lambda$, \textcolor{black}{for any $r\in (-3/2+\gamma,3/2-\gamma)$}.
\end{corollary}
\begin{rmk}[Other boundary conditions]
Depending on the boundary conditions that we consider for the system \eqref{eq:sysdiff1}, the eigenvalues of the operator $\mathcal{A}$ might have simple or double multiplicity (see \cite{Kong1996-eigenvalues,Weidmann2006-spectral,Moller1999-unboundedness, Bailey1996-regular}). Then, by adapting the space $\textcolor{black}{X}$ (see \eqref{eq:sumHi}) depending on the multiplicity of eigenvalues, we can apply our result to get stabilization condition of the system \eqref{eq:sysdiff1} even if we consider different boundary conditions \eqref{eq:bcsturm} (such as Dirichlet boundary conditions for instance).
\end{rmk}

\subsection{Application to a Gribov operator in the \textcolor{black}{Bargmann} space}
The Hilbert space $\textcolor{black}{X}$ considered in this section is the following Bargmann space
\begin{equation*}
    \textcolor{black}{X}:=\left\{f:\mathbb{C}\rightarrow \mathbb{C}\ \text{holomorphic}\ \;|\;\int_{\mathbb{C}}e^{-|z|^2}|f(z)|^2dz<\infty \ \text{and}\ f(0)=0\right\},
\end{equation*} endowed with the scalar product:
\begin{align*}
    \langle f,g\rangle:=\int_{\mathbb{C}}e^{-|z|^2}f(z)\bar g(z)dz,\quad \forall f,g\in \textcolor{black}{X},
\end{align*} 
where $dz$ denotes the $2-$dimensional Lebesgue measure on $\mathbb{C}.$
Let us consider the operators $U$ and $V$ defined as:
\begin{align*}
    \begin{cases}
        U: D(U)\subset \textcolor{black}{X}\rightarrow \textcolor{black}{X}\\
        f\mapsto Uf=\frac{df}{dz}\\
        D(U)=\{f\in \textcolor{black}{X} \;|\; Uf\in \textcolor{black}{X}\},
    \end{cases}\ \text{and}\ \begin{cases}
        V: D(V)\subset \textcolor{black}{X}\rightarrow \textcolor{black}{X}\\
        f\mapsto Vf=zf\\
        D(V)=\{f\in \textcolor{black}{X} \;|\; Vf\in \textcolor{black}{X}\}.
    \end{cases}
\end{align*}
The operator $U$ is called the annihilation operator and $V$ the creation operator. The nonself-adjoint Gribov operator (\cite{Aimar1999-On-an-unconditional}, \cite{Intissar1997-analyse}) is constructed as a polynomial of the operators $U$ and $V$ defined in the Bargmann space $\textcolor{black}{X}$ as follows:
\begin{align}
\mathcal{A}:=(VU)^3+\varepsilon V(U+V)U+\varepsilon^2(VU)^{3d_2}+\cdots+\varepsilon^k(VU)^{3d_k}+\cdots, \label{Gribov-operator}
\end{align}  where $\varepsilon\in \mathbb{C}$ and $(d_k)_{k\in\mathbb{N}}$ is a strictly decreasing sequence with strictly positive terms such that $d_0=2d_1=1.$ This operator appears in the Reggeon field theory which was introduced by Gribov in \cite{Gribov1968-A-reggeon} to study strong interactions between protons and neutrons \textcolor{black}{among} other less stable \textcolor{black}{particles}. Notice that this operator is neither self-adjoint nor skew-adjoint. So the stability analysis results stated in \cite{Gagnon2022-fredholm}, \cite{Gagnon2022-fredholm-laplacien} \textcolor{black}{cannot} cover this class of operator. In view of \cite{Feki2014-riesz}, this operator can be studied on the domain $\mathcal{D}$ defined as
\begin{align*}
    \mathcal{D}:=D((VU)^{3d_2})\cap D(V(U+V)U).
\end{align*} Consider now the following evolution equation associated to it:
\begin{align}
    \partial_t u=-\mathcal{A}u+\phi w(t) \label{eq:Grb-syst}
\end{align} where $\phi$ is a given function and $w$ is the control of the system. Based on \cite[section 4.2]{Feki2014-riesz}, for small value of $\varepsilon,$ there exists a sequence $(\varphi_n^i)_{n,i}$ such that the system of eigenvectors of $\mathcal{A}$ forms a Riesz basis in $\textcolor{black}{X}$ which can be decomposed in entire \textcolor{black}{series} as follows:
\begin{align*}
    \varphi_n(z)&=\frac{z^n}{n!}+p\frac{z^n}{\sqrt{n!}}+\sum_{i=1}^{\infty}\varepsilon^i \varphi_n^i(z),
\end{align*} for any $p\in (0,\frac{1}{\sqrt{n!}}).$ So there exists a Riesz basis $(\widetilde{\varphi_n})_{n\in \mathbb{N}^*}$ of $\textcolor{black}{X}$ which is a bi-orthogonal family of $(\varphi_n)_{n\in \mathbb{N}^*}.$ Applying our Corollary \ref{cor:expstab}, the following holds:
\begin{corollary}\label{cor-Gribov}
    Let $\gamma\in [0,1),$ and $r\in (-5/2+\gamma,5/2-\gamma).$ If $\phi$ is such that there exists $c_1,c_2>0$ satisfying
    \begin{equation*}
        c_1 n^{r}\leq \left|\int_{\mathbb{C}}e^{-|z|^2}\phi(z)\widetilde{\varphi_n}(z)dz\right|\leq c_2 n^{r+\gamma},
    \end{equation*} then if $\varepsilon$ is sufficiently small, for any $\lambda\in (0,+\infty),$ there exists a feedback law $w(t)=K(u(t, \cdot))$ with $K\in \mathcal{L}(\textcolor{black}{\mathcal{H}}^{\frac{3}{2}-r},\mathbb{C})$
    such that 
    the system \eqref{eq:Grb-syst} is exponentially stable in the \textcolor{black}{Bargmann} space $\textcolor{black}{X}$ with decay rate $\lambda.$
\end{corollary}

\begin{proof}
    All we need to show is that: the Gribov operator $-\mathcal{A}$ generates a $C^0$ semigroup and the eigenvalues of $-\mathcal{A}$ satisfy assumptions \eqref{ln+1_cond} and \eqref{ln-lp_cond} and then concludes by Corollary \ref{cor:expstab}. Based on \cite[Section 4.2, Theorem 3.2]{Feki2014-riesz}, there exists small $\varepsilon$ and a sequence $(\lambda_n^i)_{n,i}$ such that the eigenvalues of $-\mathcal{A}$ are simple and can be decomposed as follows:
\begin{align*}
    \lambda_n=-n^3+O_{\varepsilon}(1),
\end{align*} where $O_{\varepsilon}(1):=\displaystyle\sum_{i\in \mathbb{N}^*}\varepsilon^i \lambda_n^i\simeq O(|\varepsilon|,1/n).$
So the conditions \eqref{ln+1_cond} and \eqref{ln-lp_cond} hold with $\alpha=3$. The fact that $-\mathcal{A}$ generates a $C^{0}$ semigroup is a consequence of this and the existence of a Riesz basis of eigenvectors: one can easily check that there exists $\lambda_{0}>0$ such that $\lambda_{n}-\omega<0$ and therefore $-\mathcal{A}-\omega$ is dissipative and similarly thanks to the Riesz basis of eigenvectors $-\mathcal{A}-\lambda_{0}$ is surjective for any $\lambda_{0}>\omega$ and we can apply for instance Lumer-Phillips theorem to conclude to the existence of a $C^{0}$ semigroup.

\end{proof}

\section{Strategy and outline}
\label{sec:strat}
To prove our main result, the main challenge is to show the $F$-equivalence in each of the spaces $\textcolor{black}{X}_{i}$ separately, more precisely:

\begin{prop}
\label{prop:main}
Assume that $\mathcal{A}$ satisfies $(A_{1})$, \eqref{ln+1_cond} and \eqref{ln-lp_cond}. Let $\beta=(\beta_{1},...,\beta_{m})\in \mathbb{R}^{m}$ and $\gamma = (\gamma_{1},...,\gamma_{m})\in [0,(\alpha_{1}-1)/2)\times ...\times [0,(\alpha_{m}-1)/2)$. For any $i\in \{1,...,m\}$, assume that $B_{i}\in \mathcal{H}_{i}^{\beta_i- \frac{\alpha_i}{2}}$ satisfies \eqref{control-cond-general}.
Then for any $\lambda\in\mathbb{R}_{+}\setminus\mathcal{N}$, where $\mathcal{N} =\{\lambda_{n}^{i}-\lambda_{p}^{i} \;|\; (n,p)\in\mathbb{N}^{*},\;i\in\{1,...,m\}\}$,
there exists a bounded linear operator $K_{i}\in \mathcal{L}(\mathcal{H}^{\beta_i+\frac{\alpha_i}{2}}_{i};\mathbb{C})$ such that $K_{i}\in \mathcal{L}(\mathcal{H}^{\beta_i+\frac{1}{2}+\varepsilon}_{i};\mathbb{C})$ for any $\varepsilon>0$ and a linear mapping $T_{i}\in\mathcal{L}(\mathcal{H}_i^{r})$ which is an isomorphism  from $\mathcal{H}_i^{r}$ to itself for any $r\in(\beta_i+ 1/2-\alpha_i+\gamma_{i},\beta_i+ \alpha_i-1/2-\gamma_{i})$ and maps the system, 
\begin{equation}\label{sysdepart1}
\begin{array}{ll}
\partial_t u =\mathcal{A}u +B_{i}K_{i}(u),\;\;u\in \mathcal{H}^{r}_{i},
\end{array}
\end{equation}
to the system, 
\begin{equation}\label{sysarrive1}
    \begin{array}{ll}
\partial_t v =\mathcal{A}v -\lambda v,\;\;v\in \mathcal{H}^{r}_{i}.
\end{array} 
\end{equation}
\end{prop}

The idea is that if Proposition \ref{prop:main} holds, then a candidate feedback to obtain the exponential stability with decay rate $\lambda$ for the total system is $K = (K_{1},....,K_{m})^{T}$. To do so, the main step is to show that the system with feedback $K=(K_{1},...,K_{m})^{T}$ is well-posed, that is
\begin{prop}
\label{th:wellposed}
Under the assumption of Theorem \ref{Theo_main_general},
and with $K=(K_{1},...,K_{m})^{T}$ where the $K_{i}$ are given by Proposition \ref{prop:main}, for any $\vec{r}$ satisfying \eqref{eq:condr}, $\mathcal{A}+BK$ generates a semigroup on $\mathcal{H}^{\vec{r}}$ and in particular the system \eqref{closed-loop-system} has a unique solution $u\in C^{0}([0,+\infty);\mathcal{H}^{\vec{r}})$.
\end{prop}
If Propositions \ref{prop:main} and \ref{th:wellposed} hold, then Theorem \ref{Theo_main_general} (and hence Theorem \ref{Theo_main}) follows by choosing $T= T_{1}+...+T_{m}$. This is detailed in Section \ref{sec:thm}, while the proof of Propositions \ref{prop:main} and \ref{th:wellposed} are done respectively in Sections \ref{sec:step1}--\ref{sec:step6} and Section \ref{sec:wellposed}. In the remaining of this section, we present the strategy to show the main Proposition \ref{prop:main} and the spirit of the \textcolor{black}{$F$-equivalence} approach.
\\

\noindent \textbf{Principle of the approach.} Formally, having $T_{i}$ mapping the system \eqref{sysdepart1} to the system \eqref{sysarrive1} with $v=T_{i}u$ means that $T_{i}$ is an isomorphism which satisfies the following operator equality
\begin{equation}
\label{eq:eqop}
    T_{i}(\mathcal{A}+B_{i}K_{i}) = (\mathcal{A}-\lambda)T_{i},
\end{equation}
in some sense. Indeed, still formally, if \eqref{eq:eqop} holds and $v=T_{i}u$, then
\begin{equation}
    \partial_{t}v = T_{i}\partial_{t}u = T_{i}(\mathcal{A}+B_{i}K_{i})u = (\mathcal{A}-\lambda)T_{i}u = (\mathcal{A}-\lambda)v,
\end{equation}
and conversely, since $T_{i}$ is an isomorphism. Because the operator equality \eqref{eq:eqop} has no uniqueness in the solutions $(T,K)$ --note that if $(T,K)$ is a solution, then $(aT,K)$ is still a solution for any $a\neq0$-- it is tempting to add a kind of normalization on $T$ which would simplify \eqref{eq:eqop}. In this regards, a good approach is often to formally add a condition of the form 
\begin{equation}
\label{eq:tbb}
    T_iB_{i}= B_{i},
\end{equation}
again in a sense to be defined (note that since $B_{i}$ can be unbounded and not belong to the spaces $\mathcal{H}^{r}_{i}$ on which we consider $T_{i}$, this inequality may have to be considered in a weak sense). With \eqref{eq:tbb}, the operator equality \eqref{eq:eqop} becomes (formally) 
\begin{equation}
\label{eq:eqoplin}
    T_{i}\mathcal{A}+B_{i}K_{i} = (\mathcal{A}-\lambda)T_{i}.
\end{equation}
This operator equation is easier to solve than \eqref{eq:eqop} since it is linear in $(T,K)$. When $\textcolor{black}{X}$ is finite-dimensional there exists a unique solution to \eqref{eq:tbb}--\eqref{eq:eqoplin} if $(\mathcal{A},B_{i})$ is controllable in $\textcolor{black}{X}_{i}$ (see \cite{coron2015stabilization,Coron2022-stabilization}) which is an additional motivation to adding the condition \eqref{eq:tbb}. When $\textcolor{black}{X}$ is infinite-dimensional, the situation is much less clear, but 
\eqref{eq:eqoplin} can be translated in terms of actions of $T$ on $(\varphi^{i}_{n})_{i\in\mathbb{N}}$. Indeed,
projecting the operator equality \eqref{eq:eqoplin} on the eigenvector $\varphi_n^i$ of $\mathcal{A},$ we get
\begin{align}
    \lambda_n^i T_i\varphi_n^i+B_iK_i(\varphi_n^i)=(\mathcal{A}-\lambda I)T_i\varphi_n^i.\label{proj_1}
\end{align}
We set $h_n^i:=T_i\varphi_n^i;\ K_n^i:=K_i(\varphi_n^i)$ and we project \eqref{proj_1} on a bi-orthogonal family $(\widetilde{\varphi_p}^i)_{p\in\mathbb{N}^{*}}$ associated to $(\varphi_{n}^{i})_{n\in\mathbb{N}^{*}}$ to obtain
\begin{align*}
    \lambda_n^i \langle h_n^i,\widetilde{\varphi_p}^i\rangle+K_n^i\langle B_i,\widetilde{\varphi_p}^i\rangle&=\langle \mathcal{A}h_n^i,\widetilde{\varphi_p}^i\rangle-\lambda \langle h_n^i,\widetilde{\varphi_p}^i\rangle\\
    &=\lambda_p^i\langle h_n^i,\widetilde{\varphi_p}^i\rangle-\lambda \langle h_n^i,\widetilde{\varphi_p}^i\rangle
\end{align*}
 So we obtain, assuming that $\lambda$ is chosen such that $(\lambda_n^i-\lambda_p^i+\lambda)\neq 0,$ for any $(n,p)\in\mathbb{N}^*$,
 \begin{align*}
     \langle h_n^i,\widetilde{\varphi_p}^i\rangle=-K_n^i\frac{\langle B_i,\widetilde{\varphi_p}^i\rangle}{\lambda_n^i-\lambda_p^i+\lambda}.
 \end{align*}
 This leads to the following
 \begin{equation}
 \label{T_expression}
     T_i\varphi_n^i=h_n^i=\sum_{p\in \mathbb{N}^*}\langle h_n^i,\widetilde{\varphi_p}^i\rangle\varphi_p^i=-K_n^i\sum_{p\in \mathbb{N}^*}\frac{\langle B_i,\widetilde{\varphi_p}^i\rangle}{\lambda_n^i-\lambda_p^i+\lambda}\varphi_p^i.
 \end{equation}
Hence, if $T_{i}$ is a solution to \eqref{eq:eqoplin}, it has to have the decomposition \eqref{T_expression}. In particular, as soon as $K_{i}$ is chosen, the candidate transform $T_{i}$ is fixed and we aim to show that it is indeed an isomorphism from $\mathcal{H}^{r}_{i}$ to itself and maps the system \eqref{sysdepart1} to the system \eqref{sysarrive1}.\\

\noindent \textbf{Outline of the proof of Proposition \ref{prop:main}}
\vphantom{}\\

The general strategy 
is the following
\begin{enumerate}
    \item Show that the operator 
    \begin{equation}
        \label{eq:defS}
        S_{i}:\varphi_{n}^{i} \rightarrow \sum\limits_{p\in\mathbb{N}^{*}}\frac{\varphi_{p}^{i}}{\lambda_{n}^{i}-\lambda_{p}^{i}+\lambda}
    \end{equation}
    is a Fredholm operator of index $0$ from $\mathcal{H}^{r}_{i}$ into itself for any $r\in (1/2-\alpha_{i}, \alpha_{i}-1/2)$.
    \item Show that $\ker(S_{i}) = \text{ker}(S^{*}_{i})=\{0\}$ and consequently $S_{i}$ is an isomorphism from $\textcolor{black}{X}_{i}$ into itself.
    \item Show that $S_{i}$ is an isomorphism from $\mathcal{H}_{i}^{r}$ into itself for any $r\in (1/2-\alpha_{i}, \alpha_{i}-1/2)$\\
    \item Find an explicit candidate $K_{i}\in \mathcal{L}(\mathcal{H}^{1/2+\varepsilon};\mathbb{C})$ for any $\varepsilon>0$, such that $T_{i}B_{i}= B_{i}$ holds in $\mathcal{H}_{i}^{-\alpha_{i}/2}$ with
        \begin{equation}
        \label{eq:defT}
        T_{i}:\varphi_{n}^{i} \rightarrow -K_{i}(\varphi_{n}^{i})\sum\limits_{p\in\mathbb{N}^{*}}\frac{b_p^i\varphi_{p}^{i}}{\lambda_{n}^{i}-\lambda_{p}^{i}+\lambda}
    \end{equation} where $b_p^i:=\langle B_{i},\widetilde{\varphi_p}^{i}\rangle.$
    
    \item Show that the linear operator $T_{i}$ is bounded from $\mathcal{H}_{i}^{r_{i}}$ into itself for $r_{i}\in(1/2-\alpha_{i}+\gamma_{i}, \alpha_{i}-1/2)$ and satisfies the operator equality \eqref{eq:eqoplin} in $\mathcal{L}(\mathcal{H}_{i}^{\alpha_{i}/2+s};\mathcal{H}_{i}^{-\alpha_{i}/2+s})$ for any 
    \textcolor{black}{$s\in(-(\alpha_{i}-1)/2+\gamma_{i},(\alpha_{i}-1)/2-\gamma_{i})$)}.
    \item Show that $T_{i}$ is an isomorphism from
    $\mathcal{H}^{r}_{i}$ into itself for any $r\in(1/2-\alpha_{i}+\gamma_{i}, \alpha_{i}-1/2-\gamma_{i})$.
    
\end{enumerate}
Step 1, 4 and 5 are similar to \cite{Gagnon2022-fredholm}, the main difference with the duality compactness method introduced in \cite{Gagnon2022-fredholm} lies in the steps 2 and 3 where the duality argument used in \cite{Gagnon2022-fredholm} leverages the skew-adjoint properties and does not hold anymore and in the Step 5 and 6, where the boundedness of $T_{i}$ is not straightforward when $\gamma_{i}\neq 0$ while showing directly that $T_{i}$ is an isomorphism from $\mathcal{H}_{i}^{-\textcolor{black}{\alpha_{i}}/2}$ into itself as in \cite{Gagnon2022-fredholm} would be challenge. 
These steps are shown in Section \ref{sec:step1}--\ref{sec:step6}. 

\section{Proof of Theorem \ref{Theo_main_general}}

We first prove Proposition \ref{prop:main}, following the outline described in Section \ref{sec:strat} and we then deduce Theorem \ref{Theo_main_general} (see Section \ref{sec:thm}) and hence Theorem \ref{Theo_main}. Since several of the arguments of the proof (namely steps 1, 4, 5) are similar to \cite{Gagnon2022-fredholm},
we just explicit the novelty and refer the reader to \cite{Gagnon2022-fredholm} for the remaining.\\

Since Proposition \ref{prop:main} consist in showing the $F$-equivalence on each of the spaces $\textcolor{black}{X}_{i}$, we start by fixing $i\in\{1,...,m\}$ and we will work in $\textcolor{black}{X}_{i}$ (resp. $\mathcal{H}^{s}_{i}$). We also adopt the following convention: when an operator on $\textcolor{black}{X}_{i}$ (resp. $\mathcal{H}^{s}_{i}$) is defined by its action on $(\varphi_{n}^{i})_{n\in\mathbb{N}}$, we extend it to $\textcolor{black}{X}$ (resp. $\mathcal{H}^{s}$) by setting its image to $\{0\}$ on  $(\varphi_{n}^{j})_{n\in\mathbb{N}}$ for any $j\neq i$. With this in mind, in the following, we will drop the dependency in $i$ for the different quantities $B_{i}$, $K_{i}$, $T_{i}$, $\beta_{i}$, $\alpha_{i}$, $\gamma_{i}$, $(\varphi_{n}^{i})_{n\in\mathbb{N}}$, $(\lambda_{n}^{i})_{n\in\mathbb{N}}$, for the reader's convenience. We will also assume in the following that $\beta=0$, the extension to the case $\beta\neq0$ is given in Appendix \ref{app:beta}.\\ 

 Let us set, for $n\in\mathbb{N}^*,$
 \begin{align}
     q_n:=\sum_{p\in \mathbb{N}^*}\frac{\varphi_p}{\lambda_n-\lambda_p+\lambda} = S \varphi_{n},\label{q_n}
 \end{align}
\subsection{Step 1: S is a Fredholm operator}
\label{sec:step1}

We are going to show the following:
\begin{prop}
\label{prop:fredholm}
    For any $r\in(1/2-\alpha,\alpha-1/2)$, the operator $S\in\mathcal{L}(\mathcal{H}^{r})$ is a Fredholm operator of index 0. More precisely there exists a compact operator $S_{c}\in\mathcal{L}(\mathcal{H}^{r})$ such that
    \begin{equation}
        S = \lambda^{-1} Id +S_{c}.
    \end{equation}
\end{prop}

\begin{proof}[Proof of Proposition \ref{prop:fredholm}]
    
\end{proof}Let $r\in(1/2-\alpha, \alpha-1/2)$. Showing that $S$ is a Fredholm operator of order 0 in $\mathcal{H}^{r}$ can be done essentially as in \cite{Gagnon2022-fredholm}. The first thing to observe is the following
\begin{lem}
\label{lem:lambda}
    For any $\lambda\in\mathbb{R}_{+}\setminus\mathcal{N}$, there exists $C(\lambda)>0$ such that 
    $|\lambda_n-\lambda_p+\lambda|\geq C(\lambda)|\lambda_n-\lambda_p|$.
\end{lem}
\begin{proof}
For $n=p$ the inequality is true for any constant $C(\lambda)$. For $n\neq p$
      \begin{equation}
      \frac{|\lambda_{n}-\lambda_{p}\pm\lambda|}{|\lambda_{n}-\lambda_{p}|} \geq       1 - \frac{\lambda}{|\lambda_{n}-\lambda_{p}|}.
      \end{equation}
      From \eqref{ln-lp_cond} 
      $$
      \lim\limits_{n^{2}+p^{2}\rightarrow +\infty} \frac{\lambda}{|\lambda_{n}-\lambda_{p}|} \leq \lim\limits_{n^{2}+p^{2}\rightarrow +\infty}\frac{\lambda}{C\max(n,p)^{\alpha-1}} = 0.
      $$
      Since $\lambda_{n}-\lambda_{p}+\lambda \neq 0$ for any $(n,p)\in\mathbb{N}^{*}$ by definition of $\mathcal{N}$, there exists $C(\lambda)>0$ such that
$$|\lambda_n-\lambda_p+\lambda|\geq C(\lambda)|\lambda_n-\lambda_p|$$
\end{proof}
With this in mind, the key of this step is to notice that \cite[Lemma 4.6]{Gagnon2022-fredholm} can still apply despite $\mathcal{A}$ not being skew-adjoint. Indeed, we have 
\begin{lem}
\label{lem:tech}
For any $s<\alpha-1$ we have
\begin{equation} \sum\limits_{n\in\mathbb{N}^{*}\setminus\{p\}}\frac{n^{s}}{|\lambda_{n}-\lambda_{p}+\lambda|}\lesssim p^{1- \alpha+ s}\log(p)+p^{-\alpha}, \; \forall p\in \mathbb{N}^*,
\end{equation}
where $\lesssim$ means lower or equal up to a multiplicative constant that does not depend on $p$ or $n$.
\end{lem}
\begin{proof}
From Lemma \ref{lem:lambda} it is enough to show that
\begin{equation} \sum\limits_{n\in\mathbb{N}^{*}\setminus\{p\}}\frac{n^{s}}{|\lambda_{n}-\lambda_{p}|}\lesssim p^{1- \alpha+ s}\log(p)+p^{-\alpha}, \; \forall p\in \mathbb{N}^*.
\end{equation}
Since condition \eqref{ln+1_cond} and \eqref{ln-lp_cond} hold, this is exactly given by \cite[Lemma 4.2]{Gagnon2022-fredholm}.
\end{proof}
We can now perform as in \cite[Lemma 4.7]{Gagnon2022-fredholm} to show that
\begin{equation}
\label{eq:sc}
    S_c:  \sum\limits_{n\in\mathbb{N}^{*}}a_{n}n^{-r}\varphi_{n} \mapsto \sum\limits_{n\in\mathbb{N}^{*}}a_{n}n^{-r}\left(\sum\limits_{p\in\mathbb{N}^{*}\setminus\{n\}}\frac{\varphi_{p}}{\lambda_{n}-\lambda_{p}+\lambda}\right),
\end{equation}
is a compact operator from $\mathcal{H}^{r}$ to itself: 

{\color{black}
let $\varepsilon>0$ to be chosen, by using Fubini's Theorem we have, 
for any $a = \sum\limits_{n\in\mathbb{N}^{*}} a_{n}n^{-r}\varphi_{n}\in \mathcal{H}^{r}$, 
\begin{align*}
    S_c a=\sum\limits_{n\in\mathbb{N}^{*}}a_{n}n^{-r}\left(\sum\limits_{p\in\mathbb{N}^{*}\setminus\{n\}}\frac{\varphi_{p}}{\lambda_{n}-\lambda_{p}+\lambda}\right)=\sum\limits_{p\in\mathbb{N}^{*}}\varphi_{p}\left(\sum\limits_{n\in\mathbb{N}^{*}\setminus\{p\}}\frac{a_{n}n^{-r}}{\lambda_{n}-\lambda_{p}+\lambda}\right).
\end{align*}
Thus,
\begin{equation}
\begin{split}
    \left\| S_{c} a\right\|_{\mathcal{H}^{r+\varepsilon}}^{2} &= 
    \left\| \sum\limits_{p\in\mathbb{N}^{*}}\varphi_{p}\left(\sum\limits_{n\in\mathbb{N}^{*}\setminus\{p\}}\frac{a_{n}n^{-r}}{\lambda_{n}-\lambda_{p}+\lambda}\right)\right\|_{\mathcal{H}^{r+\varepsilon}}^{2} \\
&=\sum\limits_{p\in\mathbb{N}^{*}}p^{2r+2\varepsilon}\left|\sum\limits_{n\in\mathbb{N}^{*}\setminus\{p\}}
    \frac{a_{n}n^{-r}}{\lambda_{n}-\lambda_{p}+\lambda}
    \right|^{2}.
    \end{split}
\end{equation}
Hence, we can use Lemma \ref{lem:tech} and perform exactly as in \cite[Lemma 4.7] {Gagnon2022-fredholm} to get the following
\begin{lem}\label{lemma-compact-sum}
    Consider $r\in(1/2-\alpha,\alpha-1/2).$ For any $\varepsilon\in \left(0,\min\{\frac{\alpha-1}{2}, \alpha+r-\frac{1}{2}\}\right),$ it holds that
    \begin{align*}
        \sum\limits_{p\in\mathbb{N}^{*}}p^{2r+2\varepsilon}\left|\sum\limits_{n\in\mathbb{N}^{*}\setminus\{p\}}
    \frac{a_{n}n^{-r}}{\lambda_{n}-\lambda_{p}+\lambda}
    \right|^{2}\lesssim \sum\limits_{n\in\mathbb{N}^{*}} |a_{n}|^{2}.
    \end{align*} 
\end{lem}  
Using any $\varepsilon$ provided by this Lemma, the following holds:
\begin{equation}
        \left\| S_{c} a\right\|^2_{\mathcal{H}^{r+\varepsilon}} \lesssim \sum\limits_{n\in\mathbb{N}^{*}} |a_{n}|^{2}. 
\end{equation}
Since 
\begin{equation}
\sum\limits_{n\in\mathbb{N}^{*}} |a_{n}|^{2} = 
\left\|\sum\limits_{n\in\mathbb{N}^{*}} a_{n} (n^{-r}\varphi_{n})\right\|_{\mathcal{H}^{r}}^{2} = \|a\|_{\mathcal{H}^{r}}^{2},
\end{equation}
we conclude that
\begin{equation}
    \left\| S_{c} a\right\|_{\mathcal{H}^{r+\varepsilon}}\lesssim \|a\|_{\mathcal{H}^{r}}
\end{equation}
and from the compact embedding of $\mathcal{H}^{r+\varepsilon}$ in $\mathcal{H}^{r}$ (see Lemma \ref{lem:embedding}),
this implies in particular that $S_{c}$ is a compact operator from $\mathcal{H}^{r}$ to itself. Finally, noting that
\begin{equation}
\label{S_c-expression}
    S = \lambda^{-1}Id + S_{c},
\end{equation}
concludes the proof.
}

\subsection{Step 2: $S$ is an isomorphism from $\mathcal{H}^0$ to itself}
\label{sec:step2}

Using the results of Step 1  for $r=0$, we deduce that $S$ is a Fredholm operator of index $0$ from $\mathcal{H}^0$ to itself. Then, $S$ is an isomorphism if and only if $\ker S=\{0\}.$ \textcolor{black}{The idea of this step is to show that this can be reduced to an equivalence between the density of $(S\varphi_{n})_{n\in\mathbb{N}^{*}}$ and the $\omega$-independence of $(\tilde{S}\varphi_{n})_{n\in\mathbb{N}^{*}}$ where $\tilde{S}$ is the adjoint of $S$ in $\mathcal{H}^{0}$.}
\textcolor{black}{Since $S$ is of index $0,$ showing that $\ker S=\{0\}$} is also equivalent to (see \cite[Chapter 4]{Kato-book})
\begin{equation}
\label{eq:equiv}
\ker \tilde S=\{0\}\text{ where }\tilde S\text{ is the adjoint of }S.
\end{equation}

Let us introduce
\begin{equation}
\label{q_n_tilde}
\widetilde{q_n}:=\sum_{p\in \mathbb{N}^*}\frac{\varphi_p}{\overline{\lambda_p-\lambda_n+\lambda}}.
\end{equation}
We first claim the following
 \begin{lem}
 \label{lem:adjoint}
  The adjoint of the operator $S$ in $\mathcal{H}^0$ is 
  \begin{equation}
\tilde S: \varphi_n \mapsto \widetilde{q_n}.
\end{equation}

 \end{lem}
 \begin{proof}
    Indeed, consider any $f=\displaystyle\sum_{n\in \mathbb{N^*}}f_n\varphi_n,\in \mathcal{H}^0$ and $g=\displaystyle\sum_{n\in \mathbb{N^*}}g_n\varphi_n\in \mathcal{H}^0.$ The following holds
 \begin{align*}
     \langle Sf,g\rangle_{\mathcal{H}^0}
     &=\langle\sum_{n\in \mathbb{N^*}}f_n q_n, \sum_{n\in \mathbb{N^*}}g_n\varphi_n\rangle_{\mathcal{H}^0}\\
     &=\langle \sum_{n\in \mathbb{N^*}}f_n \sum_{p\in \mathbb{N}^*}\frac{\varphi_p}{\lambda_n-\lambda_p+\lambda}, \sum_{n\in \mathbb{N^*}}g_n\varphi_n\rangle_{\mathcal{H}^0}\\
     &=\langle\sum_{p\in \mathbb{N^*}}\varphi_p  \sum_{n\in \mathbb{N}^*}\frac{f_n}{\lambda_n-\lambda_p+\lambda}, \sum_{n\in \mathbb{N^*}} g_n\varphi_n\rangle_{\mathcal{H}^0}\\
     &=\sum_{p\in \mathbb{N^*}}\sum_{n\in \mathbb{N}^*}\frac{\overline{g_{p}} f_n}{\lambda_n-\lambda_p+\lambda}.
 \end{align*}
 In the same way, we have the following
 \begin{align*}
 \langle f,\tilde Sg\rangle_{\mathcal{H}^0}
             &=\langle\sum_{n\in \mathbb{N}^*}f_n \varphi_n, \sum_{n\in \mathbb{N}^*}g_n\widetilde{q_n}\rangle_{\mathcal{H}^0}\\
             &=\langle\sum_{n\in \mathbb{N^*}}f_n\varphi_n , \sum_{p\in \mathbb{N^*}} \varphi_p\sum_{n\in \mathbb{N}^*}\frac{g_n}{\overline{\lambda_p-\lambda_n+\lambda}}\rangle_{\mathcal{H}^0}\\
            &=\sum_{p\in \mathbb{N^*}} \sum_{n\in \mathbb{N}^*}\frac{f_p\overline{g_{n}}}{\lambda_p-\lambda_n+\lambda}
 \end{align*}
 Thus, we have $\langle Sf,g\rangle=\langle f,\tilde Sg\rangle$ for all $f,g\in {\mathcal{H}^0}.$ 
 \end{proof}

In view of \eqref{S_c-expression}, \eqref{eq:equiv}, $S$ is an isomorphism is equivalent to $\ker S=\{0\}$ or $\ker \tilde S=\{0\},$ which in turn is equivalent to: 
\begin{align*}
   \forall (f_{n})_{n}\in l^{2}, \sum_{n\in \mathbb{N}^*}f_n q_n=0 \Leftrightarrow f_n=0\ \forall n\in \mathbb{N}^*,\ \text{or}\ \forall (f_{n})_{n}\in l^{2} \sum_{n\in \mathbb{N}^*} f_n \widetilde{q_n}=0\Leftrightarrow  f_n=0\ \forall n\in \mathbb{N}^*.  
\end{align*}
This means that $S$ is an isomorphism from ${\mathcal{H}^0}$ to itself is equivalent to $(q_n)_{n\in\mathbb{N}^*}$ is $\omega$-independent in ${\mathcal{H}^0}$ or $(\widetilde{q_n})_{n\in\mathbb{N}^*}$ is $\omega$-independent in ${\mathcal{H}^0}$. With a classical argument (see for instance \cite{Coron2018-rapid,Gagnon2022-fredholm-laplacien,Gagnon2022-fredholm}), one can show that $(q_n)_{n\in\mathbb{N}^*}$ is either $\omega$-independent in ${\mathcal{H}^0}$ or ${\mathcal{H}^0}$-dense for any $\lambda\in \mathcal{N}$. This is done in Appendix \ref{sec-appendix-q}. So, \textcolor{black}{to show that $S$ is an isomorphism from $\mathcal{H}^{0}$ into itself, it suffices to show} the following Lemma:
\begin{lem}
    $(q_n)_{n\in\mathbb{N}^*}$ is ${\mathcal{H}^0}$-dense $\Leftrightarrow$ $(\widetilde{q_n})_{n\in\mathbb{N}^*}$ is $\omega$-independent in ${\mathcal{H}^0}$.
\end{lem}
\begin{proof}
    Assume that $(q_n)_{n\in\mathbb{N}^*}$ is not ${\mathcal{H}^0}$-dense. Then there exists a non trivial $(f_n)_{n\in \mathbb{N}^*}\in l^{2}$ such that for any $n\in\mathbb{N}^*$ $\langle q_n,f\rangle_{\mathcal{H}^0}=0$ where $f=\displaystyle\sum_{n\in \mathbb{N}^*}f_n \varphi_n.$
    \begin{align*}
        \langle q_n,f\rangle_{\mathcal{H}^0}=0&\Leftrightarrow \sum_{p\in \mathbb{N}^*}\frac{\overline{f_p}}{\lambda_n-\lambda_p+\lambda}=0\\
        &\Leftrightarrow \sum_{p\in \mathbb{N}^*}\frac{f_p}{\overline{\lambda_n-\lambda_p+\lambda}}=0
    \end{align*}
Thus, $\displaystyle\sum_{p\in \mathbb{N}^*}f_p \widetilde{q_p}=\displaystyle\sum_{p\in \mathbb{N}^*}f_p\displaystyle\sum_{n\in \mathbb{N}^*}\frac{\varphi_n}{\overline{\lambda_n-\lambda_p+\lambda}}=\displaystyle\sum_{n\in \mathbb{N}^*}\varphi_n\displaystyle\sum_{p\in \mathbb{N}^*}\frac{f_p}{\overline{\lambda_n-\lambda_p+\lambda}}=0,$ showing that $(\widetilde{q_n})_{n\in\mathbb{N}^*}$ is not $\omega$-independent.
The converse is straightforward: let us assume that $(\widetilde{q_n})_{n\in\mathbb{N}^*}$ is not $\omega$-independent. Then, there exists non trivial $(f_n)_{n\in \mathbb{N}}\in l^{2}$ such that $\displaystyle\sum_{p\in \mathbb{N}^*} f_p \widetilde{q_p}=0$.
\begin{align*}
  \sum_{p\in \mathbb{N}^*} f_p \widetilde{q_p}=0
  &\Leftrightarrow\sum_{n\in \mathbb{N}^*}\varphi_n\sum_{p\in \mathbb{N}^*}\frac{ f_p}{\overline{\lambda_n-\lambda_p+\lambda}}=0\\
  &\Leftrightarrow \sum_{p\in \mathbb{N}^*}\frac{ f_p}{\overline{\lambda_n-\lambda_p+\lambda}}=0
\end{align*} Thus by setting $f=\sum_{n\in \mathbb{N}^*} f_n \varphi_n,\in {\mathcal{H}^0}$ it holds that
$\langle f,q_n\rangle_{\mathcal{H}^0}=\sum_{p\in \mathbb{N}^*}\frac{ f_p}{\overline{\lambda_n-\lambda_p+\lambda}}=0,$ 
meaning that $(q_n)_{n\in\mathbb{N}^*}$ is not ${\mathcal{H}^0}$-dense. 
\end{proof}

\subsection{Step 3: $S$ is an isomorphism from $\mathcal{H}^{r}$ to itself for any $r\in(1/2-\alpha,\alpha-1/2)$}
\label{sec:step3}

\textcolor{black}{The idea of this step is to first leverage Step 2 to show that $S$ is an isomorphism from $\mathcal{H}^{r}$ to itself for any $r\in[0,\alpha-1/2)$, and then repeat the idea of Step 2 for negative $r$ using that, this time, the duality between the density and $\omega$-independence of $(S\varphi_{n})_{n\in\mathbb{N}^{*}}$ and the adjoint family does not hold in $\mathcal{H}^{r}$ but between $\mathcal{H}^{r}$ and $\mathcal{H}^{-r}$.}\\

Let us now look the operators $S, \tilde S$ as their extension on $\mathcal{H}^r$ defined as
\begin{align*}
    S: n^{-r}\varphi_n\mapsto n^{-r}q_n,\\
    \tilde S: n^{-r}\varphi_n\mapsto n^{-r}\widetilde{q_n}.
\end{align*} Then, $(n^{-r}q_n)_{n\in\mathbb{N}^*}$ is a Riesz basis in $\mathcal{H}^r$ if and only if $S$ is an isomorphism from $\mathcal{H}^r$ to itself. Since $S$ is a Fredholm operator of index $0,$ this is equivalent to $\ker S=\{0\}$. Like in the case $r=0$ this is equivalent to show that for any $(f_n)_{n\in \mathbb{N}^*}\in l^2$
\begin{align}
   \sum_{n\in \mathbb{N}^*}f_n n^{-r}q_n=0 \Leftrightarrow f_n=0\ \forall n\in \mathbb{N}^*.\label{w_indepen_negatif}
\end{align} 
\textcolor{black}{So we have} the following different cases:
\begin{itemize}
\item $r\in[0,\alpha-1/2),$ \\
Then $(n^{-r}f_n)_{n\in \mathbb{N}^*}\in l^2.$ And, as $(q_n)_{n\in \mathbb{N}^*}$ is a Riesz basis in ${\mathcal{H}^0},$ it holds that
\begin{align*}
    \sum_{n\in \mathbb{N}^*}f_n n^{-r}q_n=0 &\Leftrightarrow n^{-r}f_n=0\quad  \forall n\in \mathbb{N}^*\\
    &\Leftrightarrow f_n=0\quad  \forall n\in \mathbb{N}^*,
\end{align*}
thus $S$ is an isomorphism from $\mathcal{H}^{r}$ to itself for any $r\in[0,\alpha-1/2).$\\

\item $r\in(1/2-\alpha,0)$
Showing the property \eqref{w_indepen_negatif} is equivalent to show that $(n^{-r}q_n)_{n\in\mathbb{N}^*}$ is $\omega$-independent in $\mathcal{H}^r.$ We assume here by contradiction that $(n^{-r}q_n)_{n\in\mathbb{N}^*}$ is not $\omega$-independent in $\mathcal{H}^r.$ Then, there exists a non trivial $(f_n)_{n\in \mathbb{N}^*}\in l^2$ such that $\sum_{n\in \mathbb{N}^*}f_n n^{-r}q_n=0.$ Projecting this on the vector $m^{-r}\varphi_m,$ for any $m\in \mathbb{N}^*$ we get
\begin{align}
    0&=\langle\sum_{n\in \mathbb{N}^*}f_n n^{-r}q_n,m^{-r}\varphi_m\rangle_{\mathcal{H}^r}\nonumber\\
    &=\langle\sum_{p\in \mathbb{N}^*}\varphi_p\sum_{n\in \mathbb{N}^*}\frac{f_n n^{-r}}{\lambda_n-\lambda_p+\lambda},m^{-r}\varphi_m\rangle_{\mathcal{H}^r}\nonumber\\
    0&=\sum_{n\in \mathbb{N}^*}\frac{f_n n^{-r}m^r}{\lambda_n-\lambda_m+\lambda}, \quad \forall m\in \mathbb{N}^*.\label{sum=0}
\end{align}

Let us set $f:=\sum_{n\in \mathbb{N}^*}f_n n^{r}\varphi_n.$ As $(f_n)_{n\in \mathbb{N}^*}\in l^2,$ it holds that $f\in \mathcal{H}^{-r}.$ Since $(n^r\varphi_n)_{n\in \mathbb{N}^*}$ is a Riesz basis of $\mathcal{H}^{-r}$ and $(f_n)_{n\in \mathbb{N}^*}$ is non trivial, we have that $f\neq 0$. Since we showed in the previous case that $S$ is an isomorphism on $\mathcal{H}^r$ for $r\in [0,\alpha-1/2),$ then $\tilde S$ is also an isomorphism from $\mathcal{H}^r$ to itself for $r\in [0,\alpha-1/2),$ as the adjoint of $S.$ 
Thus, $(n^{-r}\widetilde{q_n})_{n\in \mathbb{N}^*}$ is a Riesz basis of $\mathcal{H}^r$ for $r\in [0,\alpha-1/2).$ Then for $r\in (1/2-\alpha,0),$  $(n^{r}\widetilde{q_n})_{n\in \mathbb{N}^*}$ is a Riesz basis of $\mathcal{H}^{-r}.$ In particular $( n^{r}\widetilde{q_n})_{n\in \mathbb{N}^*}$ is $\mathcal{H}^{-r}$-dense. This ensures that there exists $m_0\in \mathbb{N}^*$ such that $\langle f,m_0^r\tilde q_{m_0}\rangle_{\mathcal{H}^{-r}}\neq 0.$ It follows that
\begin{align*}
    \langle f,m_0^r\tilde q_{m_0}\rangle_{\mathcal{H}^{-r}}=\sum_{n\in \mathbb{N}^*}\frac{f_n n^{-r}m_0^r}{\lambda_n-\lambda_{m_0}+\lambda}\neq 0
\end{align*} which contradicts \eqref{sum=0}. So $(n^{-r}q_n)_{n\in\mathbb{N}^*}$ is $\omega$-independent in $\mathcal{H}^r$ and we conclude that $S$ is an isomorphism from $\mathcal{H}^r$ to itself for any $r\in(1/2-\alpha,0).$
\end{itemize}
Therefore, $S$ is an isomorphism from $\mathcal{H}^r$ to itself for any $r\in(1/2-\alpha,\alpha-1/2)$.

\subsection{Step 4: an explicit candidate for $K$}
\label{sec:step4}
{\color{black}
\textcolor{black}{Recall that the overall goal is to build $T$ and $K$ such that $T$ is an isomorphism and satisfies the two operator equalities \eqref{eq:tbb} and \eqref{eq:eqoplin}. Since $T$ given by \eqref{eq:defT} depends on $K$ and satisfies (formally) the operator equality \eqref{eq:eqoplin}, the goal of this step is to find $K$ such that $T$ also satisfies the operator equality \eqref{eq:tbb}.} Consider $r\in (1/2-\alpha,\alpha-1/2).$ Let us set 
\begin{equation} \label{eq:deftau}
 \tau: n^{-r}\varphi_n\mapsto b_n n^{-r}\varphi_n,
\end{equation} with $b_n:=\langle B, \widetilde{\varphi_n}\rangle,$ and \begin{align}
   K:\varphi_n\mapsto K(\varphi_n)=K_n.\label{K-phi-n}
\end{align} 

For any $r\in (1/2-\alpha,\alpha-1/2),$ thanks to \eqref{control_cond}, we can see that the operator $\tau$ is an isomorphism from $\mathcal{H}^r$ to $\tau(\mathcal{H}^r).$ Notice that this last space is \emph{a priori} not $\mathcal{H}^r$ when $\gamma\neq0$. Thus, by seeing \textcolor{black}{that $T$ can be written} as follows
\begin{equation}
\label{eq:decompT2}
T : n^{-r}\varphi_{n}\mapsto -K_n n^{-r}\tau q_n,
\end{equation}  the boundedness by above and below of $K_n$ does not suffice to conclude that $T$ is an isomorphism from $\mathcal{H}^r$ to itself as done in \cite{Gagnon2022-fredholm}. Here, we rather see the \textcolor{black}{definition of $T$ as follows}
\begin{equation*}
T : n^{-r}\varphi_{n}\mapsto -K_n b_n (\frac{1}{b_n} n^{-r}\tau q_n).
\end{equation*} 
In Section \ref{sec:step6}, we will show that, under our assumption on $\gamma$, the operator $n^{-r}\varphi_{n}\mapsto  b_n^{-1} n^{-r}\tau q_n$ is an isomorphism from $\mathcal{H}^r$ to itself and then $(b_n^{-1} n^{-r}\tau q_n)_{n\in \mathbb{N}^*}$ forms a Riesz basis of $\mathcal{H}^r$. Note that the latter is not obvious \emph{a priori, since $\tau$ and $S$ do not commute thus $b_n^{-1} n^{-r}\tau q_n\neq n^{-r}q_{n}$.}
From this we will be able to deduce that the operator $T$ is an isomorphism from $\mathcal{H}^r$ to itself if $(K_{n} b_n)_{n\in\mathbb{N}^*}$ is uniformly bounded by above and below.\\
 
 Recall that our goal is to have $T$ that is a solution to \eqref{eq:eqop} and \eqref{eq:tbb}. In this section 
we would like to construct $K_n$ such that the normalization \eqref{eq:tbb} holds and such that $(K_{n}b_n)_{n\in\mathbb{N}}$ is bounded by above. The question of whether $(K_{n} b_n)_{n\in\mathbb{N}}$ is bounded by below is the object of Steps (5)-(6). We first state the following Lemma
\begin{lem}\label{lem:regk}
 There exists a unique sequence $(K_n)_{n\in{\mathbb{N}^*}}$ such that for any $\varepsilon \in (0,\alpha-1)$ the condition \eqref{eq:tbb} holds in $\tau(\mathcal{H}^{-\frac{1}{2}-\varepsilon})$ and $(-K_n b_n n^{-(\frac{1}{2}+\varepsilon)})_{n\in\mathbb{N}^*}\in l^2.$   
\end{lem}
\begin{proof}
    Noticing that $B=\sum_{n\in \mathbb{N}^*}b_n\varphi_n\in \mathcal{H}^{-\alpha/2}$ with $c_1\leq b_n\leq c_2 n^{\gamma},$ (recall that we assume $\beta=0$ to ease the notation, see Appendix \ref{app:beta} for $\beta\neq 0$)
    it holds from definition \eqref{eq:deftau} that $\tau$ is an isomorphism from $\mathcal{H}^r$ to $\tau(\mathcal{H}^r)$ for any $r\in (1/2-\alpha,\alpha-1/2).$ Thus for any $r\in (-1/2-\alpha,\alpha-1/2),$ $\tau (n^{-r}\varphi_n)_{n\in \mathbb{N}^*}$ forms a Riesz basis of $\tau(\mathcal{H}^r).$  Then the expression 
    \begin{align}
        B=\displaystyle\sum_{n\in \mathbb{N}^*} n^{-\frac{1}{2}-\varepsilon}\tau(n^{\frac{1}{2}+\varepsilon}\varphi_n) \label{B-express}
    \end{align} makes sense in $\tau(\mathcal{H}^{-\frac{1}{2}-\varepsilon})$  for any $\varepsilon>0$, since $\left(n^{-\frac{1}{2}-\varepsilon}\right)_{n\in \mathbb{N}^*}\in l^2$. In particular, the equation \eqref{B-express} shows that $B\in \tau(\mathcal{H}^{-\frac{1}{2}-\varepsilon})$ for all $\varepsilon>0$. 
    The condition \eqref{eq:tbb} can be expressed in a weak sense as
    \begin{align}
   \sum_{n\in \mathbb{N}^*}-K_n b_n \tau q_n=\sum_{n\in \mathbb{N}^*}b_n \varphi_n,\label{TB=B-expression}
\end{align}  and we yet know that its right side belongs to $\tau(\mathcal{H}^{-\frac{1}{2}-\varepsilon})$ for all $\varepsilon>0$. This means, in particular, that the left-hand side does too. As a consequence $(-K_n b_n n^{-(\frac{1}{2}+\varepsilon)})_{n\in \mathbb{N}^*}\in l^2$ for all $\varepsilon \in (0,\alpha-1)$.
Indeed, in view of step \ref{sec:step2}, $(n^{\frac{1}{2}+\varepsilon}q_n)$ is a Riesz basis of $\mathcal{H}^{-\frac{1}{2}-\varepsilon}$ for any $\varepsilon \in (0,\alpha-1)$, and since $\tau$ is an isomorphism from $\mathcal{H}^{-\frac{1}{2}-\varepsilon}$ to $\tau(\mathcal{H}^{-\frac{1}{2}-\varepsilon}),$ then $\tau(n^{\frac{1}{2}+\varepsilon} q_n)_n$ is a Riesz basis of $\tau(\mathcal{H}^{-\frac{1}{2}-\varepsilon})$.
\end{proof}
}

The regularity given by Lemma \ref{lem:regk} is not enough \emph{a priori}. To tackle this, we proceed as in \cite{Gagnon2022-fredholm} and exploit the special structure of $(q_{n})_{n\in\mathbb{N}}$ and the fact that $(q_{n})_{n\in\mathbb{N}}-\lambda^{-1}$ is more regular than $(q_{n})_{n\in\mathbb{N}}$ (see Lemma \ref{lem:tech}). We set 
 \begin{equation}
 \label{K_n-expression}
    k_n:=-(K_nb_n+\lambda)\quad \forall n\in \mathbb{N}^*.
\end{equation}

  \begin{lem}\label{Lemma-K_n}
Consider $r\in (1/2-\alpha,\alpha-1/2).$ There exists $\bar\varepsilon>0$ such that the sequence $(k_n n^{\varepsilon})_{n\in \mathbb{N}^*}$ is uniformly bounded for any $\varepsilon\in (0,\bar \varepsilon).$ In particular, the sequence $(K_n b_n)_{n\in \mathbb{N}^*}$ defined by \eqref{K_n-expression} is uniformly bounded.
  \end{lem}

{\color{black}
\begin{proof}[Proof of Lemma \ref{Lemma-K_n}]
For any $\varepsilon >0,$ it holds that $(n^{-\frac{1}{2}-\varepsilon}\lambda)_{n\in\mathbb{N}^{*}}\in l^2,$ thus using Lemma \ref{lem:regk} and \eqref{K_n-expression}, we have for all $\varepsilon\in (0,\alpha-1)$ that
$(k_n n^{-\frac{1}{2}-\varepsilon})_{n\in \mathbb{N}^*}\in l^2,$ and in particular
\begin{align} 
    (k_n n^{-r})_{n\in \mathbb{N}^*}\in l^2,\quad \forall r\in \left(\frac{1}{2},\alpha-\frac{1}{2}\right).\label{kn-regular0}
\end{align}
Still using Lemma \ref{lem:regk}, the equation \eqref{TB=B-expression} makes sense in $\tau(\mathcal{H}^{-\frac{1}{2}-\varepsilon}),$ and by expressing $K_n b_n$ and $\tau q_n$ it becomes
\begin{align*}
   \sum_{n\in \mathbb{N}^*}\lambda \sum_{p\in \mathbb{N}^*}\frac{b_{p}\varphi_p}{\lambda_n-\lambda_p+\lambda}+ \sum_{n\in \mathbb{N}^*}k_{n}\tau q_{n}= \sum_{n\in \mathbb{N}^*}b_n \varphi_n.
\end{align*}
This yields to the following
\begin{align}
   \sum_{n\in \mathbb{N}^*}\lambda \sum_{p\in \mathbb{N}^*\setminus\{n\}}\frac{b_{p}\varphi_p}{\lambda_n-\lambda_p+\lambda} = -\sum_{n\in \mathbb{N}^*}k_{n}\tau q_{n}\label{sum-regular}
\end{align}

which holds a priori in $\tau(\mathcal{H}^{-\frac{1}{2}-\varepsilon})$ for any $\varepsilon \in (0,\alpha-1).$ However, the equality \eqref{sum-regular} \textcolor{black}{now holds in a more regular space} than $\tau(\mathcal{H}^{-\frac{1}{2}-\varepsilon}).$ Indeed, if $\alpha>3/2,$ we can give a sense to \eqref{sum-regular} in $\tau(\mathcal{H}^{\varepsilon})$ for $0\leq \varepsilon\leq \alpha-3/2.$ To prove that, let us first notice by Fubini's Theorem in $\tau(\mathcal{H}^{-\frac{1}{2}-\varepsilon})$ that
\begin{align*}
   \sum_{n\in \mathbb{N}^*}\lambda \sum_{p\in \mathbb{N}^*\setminus\{n\}}\frac{b_{p}\varphi_p}{\lambda_n-\lambda_p+\lambda} = \sum_{p\in \mathbb{N}^*}b_{p}\varphi_p\lambda \sum_{n\in \mathbb{N}^*\setminus\{p\}}\frac{1}{\lambda_n-\lambda_p+\lambda}.
\end{align*}
And it holds that 
\begin{align*}
    \sum_{p\in \mathbb{N}^*}b_{p}\varphi_p\lambda \sum_{n\in \mathbb{N}^*\setminus\{p\}}\frac{1}{\lambda_n-\lambda_p+\lambda}=\sum_{p\in \mathbb{N}^*}\tau(p^{-\varepsilon}\varphi_p) \left(p^{\varepsilon}\lambda\sum_{n\in \mathbb{N}^*\setminus\{p\}}\frac{1}{\lambda_n-\lambda_p+\lambda}\right).
\end{align*}
Since $\tau(p^{-\varepsilon}\varphi_p)_{p\in \mathbb{N}^*}$ is a Riesz basis of $\tau(\mathcal{H}^{\varepsilon}),$ it suffices to show that \[\left(p^{\varepsilon}\lambda\displaystyle\sum_{n\in \mathbb{N}^*\setminus\{p\}}\frac{1}{\lambda_n-\lambda_p+\lambda}\right)_{p\in \mathbb{N}^*}\in l^2\] to conclude that the left hand side of \eqref{sum-regular} belongs to  $\tau(\mathcal{H}^{\varepsilon}).$ Using Lemma \ref{lem:lambda} and Lemma \ref{lem:tech}, we have that
\begin{align*}
    \left\|\left(p^{\varepsilon}\lambda\displaystyle\sum_{n\in \mathbb{N}^*\setminus\{p\}}\frac{1}{\lambda_n-\lambda_p+\lambda}\right)_{p\in \mathbb{N}^*}\right\|^2_{l^2}&=\sum_{p\in \mathbb{N}^*}p^{2\varepsilon}\lambda^2\left|\displaystyle\sum_{n\in \mathbb{N}^*\setminus\{p\}}\frac{1}{\lambda_n-\lambda_p+\lambda}\right|^2\\
    &\leq \frac{\lambda^2}{C(\lambda)^2}\sum_{p\in \mathbb{N}^*}p^{2\varepsilon}\left(\displaystyle\sum_{n\in \mathbb{N}^*\setminus\{p\}}\frac{1}{|\lambda_n-\lambda_p|}\right)^2\\
    &\lesssim\sum_{p\in \mathbb{N}^*} p^{2\varepsilon+2(1-\alpha)}\log^2(p),
\end{align*} 
and this converges for $\varepsilon \in (0,\alpha-3/2).$ So if $\alpha>3/2,$ the left hand side of \eqref{sum-regular} holds in $\tau(\mathcal{H}^{\varepsilon})$ for $\varepsilon \in (0,\alpha-3/2).$\\ 

Consider now the right hand side of \eqref{sum-regular}. We have that
\begin{align}
\label{eq:decompkn1}
   \sum_{n\in \mathbb{N}^*}k_{n}\tau q_{n}=\sum_{n\in \mathbb{N}^*}k_{n}n^{\varepsilon}\tau(n^{-\varepsilon} q_{n}). 
\end{align} 
Recalling that $(n^{-\varepsilon} q_{n})_{n\in \mathbb{N}^*}$ is a Riesz basis of $\mathcal{H}^{\varepsilon}$ for $\varepsilon \in (0,\alpha-3/2),$ it holds by the isomorphism property of $\tau$ that $\tau(n^{-\varepsilon} q_{n})_{n\in \mathbb{N}^*}$ is a Riesz basis of $\tau(\mathcal{H}^{\varepsilon}).$ Thus in view of \eqref{sum-regular} and \eqref{eq:decompkn1}, 
for any $\varepsilon \in (0,\alpha-3/2),$
\begin{align*}
    (k_n n^{\varepsilon})_{n\in \mathbb{N}^*}\in l^2.
\end{align*} 
In particular, we conclude that if $\alpha>3/2$, $(k_n)_{n\in \mathbb{N}^*}\in l^{\infty}.$
In view of \eqref{K_n-expression}, we have $K_n b_n=-(\lambda+k_n)$ and since $(k_n)_{n\in \mathbb{N}^*}$ belongs to $l^{\infty},$ we get that 
\begin{align}
    (K_n b_n)_{n\in \mathbb{N}^*}\in l^{\infty}\quad  \forall \alpha>3/2.\label{Kn-l-infty1}
\end{align} 
Consider now $1<\alpha\leq 3/2.$ In this case the gain of regularity in \eqref{sum-regular} is not enough to deduce the boundedness of $K_n b_n.$ However we are able to show that the equality \eqref{sum-regular} holds in $\tau(\mathcal{H}^{-\varepsilon}),$ for $\varepsilon>3/2-\alpha.$ Indeed, similarly to the previous case, the following holds
\begin{align*}
    \left\|\left(p^{-\varepsilon}\lambda\displaystyle\sum_{n\in \mathbb{N}^*\setminus\{p\}}\frac{1}{\lambda_n-\lambda_p+\lambda}\right)_{n\in \mathbb{N}^*}\right\|^2_{l^2}&=\sum_{p\in \mathbb{N}^*}p^{-2\varepsilon}\lambda^2\left|\displaystyle\sum_{n\in \mathbb{N}^*\setminus\{p\}}\frac{1}{\lambda_n-\lambda_p+\lambda}\right|^2\\
    &\leq \frac{\lambda^2}{C(\lambda)^2}\sum_{p\in \mathbb{N}^*}p^{-2\varepsilon}\left(\displaystyle\sum_{n\in \mathbb{N}^*\setminus\{p\}}\frac{1}{|\lambda_n-\lambda_p|}\right)^2\\
    &\lesssim\sum_{p\in \mathbb{N}^*} p^{-2\varepsilon+2(1-\alpha)}\log^2(p).
\end{align*} This converges for $\varepsilon>3/2-\alpha$ and then the left hand side of \eqref{sum-regular} belongs to $\tau(\mathcal{H}^{-\varepsilon}).$ 
Similarly as before and in view of \eqref{sum-regular}, 
\begin{align}
 (k_n n^{-\varepsilon})_{n\in\mathbb{N}^*}\in l^2 \quad \forall \varepsilon\in \left(\frac{3}{2}-\alpha, \alpha-\frac{1}{2}\right).\label{kn-regular1}  
\end{align}
We can see that we obtain $(\alpha-1)$ gain of regularity between \eqref{kn-regular1} and \eqref{kn-regular0}. So in the following, we will make an iterative principle to gain at each order this $(\alpha-1)$ regularity in order to prove Lemma \ref{Lemma-K_n}. So let us go back to the equation \eqref{TB=B-expression} and expressing $\tau q_n$
\begin{align*}
    \sum_{n\in \mathbb{N}^*}-K_n b_n\left(\frac{b_n\varphi_n}{\lambda}+\sum_{p\in \mathbb{N}^*\setminus\{n\}}\frac{b_p\varphi_p}{\lambda_n-\lambda_p+\lambda}\right)=\sum_{n\in \mathbb{N}^*}b_n\varphi_n.
\end{align*}
Replacing $-K_n b_n$ by $(\lambda+k_n)$ in the first term only and using Fubini's Theorem in $\tau(\mathcal{H}^{-\frac{1}{2}-\varepsilon})$ we get:
\begin{align*}
  \sum_{n\in \mathbb{N}^*}\frac{k_n b_n}{\lambda}\varphi_n-\sum_{p\in \mathbb{N}^*}b_p\varphi_p\sum_{n\in \mathbb{N}^*\setminus\{p\}}\frac{K_n b_n}{\lambda_n-\lambda_p+\lambda}=0.  
\end{align*}
Since $(\varphi_n)_{n\in \mathbb{N}^*}$ is a Riesz basis, we have by identification that
\begin{align*}
  \frac{k_m b_m}{\lambda}=b_m\sum_{n\in \mathbb{N}^*\setminus\{m\}}\frac{K_n b_n}{\lambda_n-\lambda_m+\lambda}\quad \forall m\in \mathbb{N}^*.  
\end{align*}
Thanks to the fact that $b_m\neq 0$ for all $m\in \mathbb{N}^*$ and $K_n b_n=-(\lambda+k_n),$ this implies that 
\begin{align*}
  k_m=-\lambda\sum_{n\in \mathbb{N}^*\setminus\{m\}}\frac{\lambda+k_n}{\lambda_n-\lambda_m+\lambda}\quad \forall m\in \mathbb{N}^*.  
\end{align*}
Let us now set $\lambda=e_n^0$ and $k_n=k_n^0.$ It holds that
\begin{align}
  k_m=-\lambda\sum_{n\in \mathbb{N}^*\setminus\{m\}}\frac{e_n^0+k_n^0}{\lambda_n-\lambda_m+\lambda}\quad \forall m\in \mathbb{N}^*.  \label{kn-def}
\end{align}
Thus $k_m$ can be rewritten as $k_m=e_m^1+k_m^1$ where
\begin{align*}
    e_m^1=-\lambda\sum_{n\in \mathbb{N}^*\setminus\{m\}}\frac{e_n^0}{\lambda_n-\lambda_m+\lambda}, \quad k_m^1=-\lambda\sum_{n\in \mathbb{N}^*\setminus\{m\}}\frac{k_n^0}{\lambda_n-\lambda_m+\lambda} \quad \forall m\in \mathbb{N}^*.
\end{align*}
And using Lemma \ref{lem:tech}, we have that
\begin{align}
    |e_m^1|\lesssim m^{1-\alpha}\log m+m^{-\alpha}\lesssim 1.\label{e-1}
\end{align} So, we focus on the regularity of $k_n^1.$ We define $\varepsilon_0:=\frac{3}{2}-\alpha$ and we have
\begin{align*}
    \|(m^{-\varepsilon} k_m^1)_m\|_{l^2}\lesssim \sum_{m\in \mathbb{N}^*}m^{-2\varepsilon}\left(\sum_{n\in \mathbb{N}^*\setminus\{m\}}\frac{k_n^0}{\lambda_n-\lambda_m+\lambda}\right)^2.
\end{align*}
Using Cauchy-Schwartz inequality and Lemma \ref{lem:lambda}, we get
\begin{align*}
    \|(m^{-\varepsilon} k_m^1)_m\|_{l^2}\lesssim \sum_{m\in \mathbb{N}^*}m^{-2\varepsilon}\left(\sum_{n\in \mathbb{N}^*\setminus\{m\}}\frac{|k_n^0|^2}{|\lambda_n-\lambda_m|}\right)\left(\sum_{n\in \mathbb{N}^*\setminus\{m\}}\frac{1}{|\lambda_n-\lambda_m|}\right).
\end{align*}
Then applying Lemma \eqref{lem:tech} and Fubini's \textcolor{black}{theorem}, we have
\begin{align*}
    \|(m^{-\varepsilon} k_m^1)_m\|_{l^2}&\lesssim \sum_{m\in \mathbb{N}^*}m^{-2\varepsilon+1-\alpha}\log m\left(\sum_{n\in \mathbb{N}^*\setminus\{m\}}\frac{|k_n^0|^2}{|\lambda_n-\lambda_m|}\right)\\
    &\lesssim \sum_{n\in \mathbb{N}^*}|k_n^0|^2\left(\sum_{m\in \mathbb{N}^*\setminus\{n\}}\frac{m^{-2\varepsilon+1-\alpha}\log m}{|\lambda_n-\lambda_m|}\right)\\
    &\lesssim \sum_{n\in \mathbb{N}^*}|k_n^0|^2n^{-2(\varepsilon-1+\alpha)+\sigma},
\end{align*} 
for any $\sigma>0$. As $k_n^0=k_n$ and in view of \eqref{kn-regular1}, for any $\varepsilon-(1-\alpha)\in \left(3/2-\alpha,\alpha-1/2\right)$ (i.e. $\varepsilon\in \left(5/2-2\alpha,1/2\right)$) there exists $\sigma>0$ such that this converges. Thus for $\alpha\in (1,3/2]$, the following holds
\begin{align}
    (n^{-\varepsilon} k_n^1)_{n\in \mathbb{N}^*}\in l^2, \quad \forall \varepsilon\in \left(\frac{5}{2}-2\alpha,\frac{1}{2}\right).\label{kn1-l2}
\end{align}
If $\alpha>\frac{5}{4},$ we have $\varepsilon_1=2\alpha-\frac{5}{2}>0$ and then $(n^{\delta} k_n^1)_{n\in \mathbb{N}^*}\in l^2$ for all $\delta\in[0,\varepsilon_1).$ This ensures that $(k_n^1)_{n\in \mathbb{N}^*}$ is uniformly bounded. And combining with \eqref{e-1}, it ensures that $(k_n)_{n\in \mathbb{N}^*}$ also is uniformly bounded. Therefore $(K_n b_n)_{n\in \mathbb{N}^*}$ is uniformly bounded because $K_n b_n=-(\lambda+k_n).$ \\
If $1<\alpha\leq\frac{5}{4},$ the uniform boundedness of $K_n b_n$ cannot be immediately deduced, but we need to iterate again. So using the definition \eqref{kn-def}, we have from $k_n^0=k_n,$ that
\begin{align*}
     k_m&=-\lambda\sum_{n\in \mathbb{N}^*\setminus\{m\}}\frac{e_n^0+k_n^0}{\lambda_n-\lambda_m+\lambda}\\
     &=-\lambda\sum_{n\in \mathbb{N}^*\setminus\{m\}}\frac{e_n^0}{\lambda_n-\lambda_m+\lambda}-\lambda\sum_{n\in \mathbb{N}^*\setminus\{m\}}\frac{k_n}{\lambda_n-\lambda_m+\lambda}\\
     &=-\lambda\sum_{n\in \mathbb{N}^*\setminus\{m\}}\frac{e_n^0}{\lambda_n-\lambda_m+\lambda}-\lambda\sum_{n\in \mathbb{N}^*\setminus\{m\}}\frac{1}{\lambda_n-\lambda_m+\lambda}\left(-\lambda\sum_{p\in \mathbb{N}^*\setminus\{n\}}\frac{e_p^0+k_p^0}{\lambda_p-\lambda_n+\lambda}\right).\\
\end{align*}
Referring to the definitions of $e_n^1$ and $k_n^1,$ it holds that
\begin{align}
 k_m&=e_m^1-\lambda\sum_{n\in \mathbb{N}^*\setminus\{m\}}\frac{e_n^1}{\lambda_n-\lambda_m+\lambda} -\lambda\sum_{n\in \mathbb{N}^*\setminus\{m\}}\frac{k_n^1}{\lambda_n-\lambda_m+\lambda}\nonumber\\
 &=e_m^1+e_m^2+k_m^2,\label{e1e2k2}
\end{align} where
\begin{align*}
  e_m^2=  -\lambda\sum_{n\in \mathbb{N}^*\setminus\{m\}}\frac{e_n^1}{\lambda_n-\lambda_m+\lambda}, \quad k_m^2=-\lambda\sum_{n\in \mathbb{N}^*\setminus\{m\}}\frac{k_n^1}{\lambda_n-\lambda_m+\lambda} \quad \forall m\in \mathbb{N}^*.
\end{align*}
Since $k_m=e_m^1+k_m^1,$ we have in view of \eqref{e1e2k2} that $k_m^1=e_m^2+k_m^2.$ This means that if $(e_m^2)_{m\in \mathbb{N}}$ and $(k_m^2)_{m\in \mathbb{N}}$ are uniformly bounded, then $(k_m^1)_{m\in \mathbb{N}}$ is and consequently $(k_m)_{m\in \mathbb{N}}.$ Let us first notice that combining \eqref{e-1} and Lemma \ref{lem:tech}, we have
\begin{align*}
    |e_m^2|\lesssim m^{1-\alpha}\log n+m^{-\alpha}\lesssim 1.
\end{align*}
For the $(k_m)_{m\in \mathbb{N}},$ we use again Lemma \eqref{lem:tech}  and Fubini's Theorem to have, similarly as previously:
\begin{align*}
    \|(m^{-\varepsilon} k_m^2)_m\|_{l^2}&\lesssim \sum_{m\in \mathbb{N}^*}m^{-2\varepsilon}\left(\sum_{n\in \mathbb{N}^*\setminus\{m\}}\frac{k_n^1}{\lambda_n-\lambda_m+\lambda}\right)^2\\
    &\lesssim \sum_{n\in \mathbb{N}^*}|k_n^1|^2n^{-2(\varepsilon-1+\alpha)+\sigma},
\end{align*}
for any $\sigma>0$.
In view of \eqref{kn1-l2}, there exists $\sigma>0$ such that this converges for any $\varepsilon -(1-\alpha)\in \left(\frac{5}{2}-2\alpha,\frac{1}{2}\right)$ meaning for $\varepsilon\in \left(\frac{7}{2}-3\alpha,\frac{3}{2}-\alpha\right).$ Thus, for $\alpha\in (1,3/2],$
\begin{align*}
    (n^{-\varepsilon}k_n^2)_{n\in \mathbb{N}^*}\in l^2, \quad \forall \varepsilon\in \left(\frac{7}{2}-3\alpha,\frac{3}{2}-\alpha\right).
\end{align*}

If $\alpha > 7/6,$ then $\varepsilon_2=3\alpha-\frac{7}{2}>0$ and $(n^{\delta}k_n^2)_{n\in \mathbb{N}^*}\in l^2,$ for any $\delta\in [0,\varepsilon_2).$ This ensures that $(k_n^2)_{n\in \mathbb{N}^*}$ is uniformly bounded and then $(k_n^2)_{n\in \mathbb{N}^*}$ is. Thus, $\left(K_n b_n=-(e_n^0+e^1_n+e_n^2+k_n^2)\right)_{n\in \mathbb{N}^*}$ is uniformly bounded.\\

If $1<\alpha\leq \frac{7}{6},$ we can continue the induction and have
$k_n^i=e_n^{i+1}+k_n^{i+1}$ for all $i\in \mathbb{N}$ where
\begin{align*}
  e_n^{i+1}=  -\lambda\sum_{m\in \mathbb{N}^*\setminus\{n\}}\frac{e_m^i}{\lambda_m-\lambda_n+\lambda}, \quad k_n^{i+1}=-\lambda\sum_{m\in \mathbb{N}^*\setminus\{n\}}\frac{k_m^i}{\lambda_m-\lambda_n+\lambda} \quad \forall n\in \mathbb{N}^*, i\in \mathbb{N}.
\end{align*}
And for any $n\in \mathbb{N}^*,$
\begin{align*}
    K_n b_n=-\left(k_n^i+\sum_{j=0}^{i}e_n^j\right), \quad \forall i\in \mathbb{N}
\end{align*}
Thus, based on the previous computations, we obtain by induction the following:
\begin{align}
    |e_n^i|\lesssim n^{1-\alpha}\log n+n^{-\alpha}\lesssim 1,\label{en-i}\\
    \|(n^{-\varepsilon}k_n^{i+1})_{n\in \mathbb{N}^*}\|_{l^2}\lesssim \sum_{n\in \mathbb{N}^*}|k_n^i|^2n^{-2(\varepsilon-1+\alpha)+\sigma}\log^2 n\label{nk-i+1}
\end{align} for any $\sigma>0$ and \eqref{nk-i+1} converges for $\varepsilon\in (-\varepsilon_i,-\varepsilon_{i-2})$ with $\varepsilon_i=(\alpha-1)i-\varepsilon_0$, \textcolor{black}{with $\varepsilon_{0}=1/2$,} for all $i\geq 2.$ It is worth stressing that there exists a finite $i_0\in \mathbb{N}$ such that $\varepsilon_{i_0}>0.$ In this case we have $(n^{\delta}k_n^{i_0+1})_{n\in \mathbb{N}^*}\in l^2$ for any $\delta \in [0,\varepsilon_{i_0})$ meaning that $(k_n^{i_0+1})_{n\in \mathbb{N}^*}\in l^{\infty}.$ Combining this with \eqref{en-i} we conclude that 
\begin{align}
    \left(K_n b_n=-\left(k_n^{i_0}+\sum_{j=0}^{i_0}e_n^j\right)\right)_{n\in \mathbb{N}^*}\in l^{\infty}\quad \forall \alpha\in (1, 3/2]. \label{Kn-l-infty2}
\end{align}
Therefore we conclude Lemma \ref{Lemma-K_n} from \eqref{Kn-l-infty1} and \eqref{Kn-l-infty2}.
\end{proof}
}
{\color{black}
\begin{rmk}
 The control feedback such defined is defined by its action on $(\varphi_{n})_{n\in\mathbb{N}^{*}}$ and one could wonder when this operator can expressed more simply. This could happen in particular cases where an explicit value of the sum can be derived such as in \cite{zhang2022internal}. In other cases, however $K$ might not have a more simple expression, this is a price to pay to have a feedback control $K$ that takes values in a finite-dimensional space rather than an infinite-dimensional space. 
\end{rmk}
}
\textcolor{black}{\begin{rmk}
    Another, but related, interesting question would be to see what would happen with a control feedback $\tilde{K}$ where the coefficients $k_{n}$ is replaced by $0$ after a certain step. This control would only require knowledge of $b_{n}$ and $\lambda_{n}$ up to a certain $N\in\mathbb{N}^{*}$. If the system is parabolic, thanks to the fact that high frequencies are stable, one can expect the system with control feedback law $\tilde{K}$ to still be stable. When the system is not parabolic however, the question is more complicated.
\end{rmk}}
\textcolor{black}{\begin{rmk}
    In \cite{Gagnon2025abstract}, the authors provide quantitative and sharp estimate for the control feedback $K$ in term of $\lambda$ when the operator $\mathcal{A}$ is self-adjoint. It would be interesting to understand if one can recover the same estimate in our case (at least in the case where we select $\gamma$ to be $0$) even if the system is not self-adjoint. 
\end{rmk}}

\subsection{Step 5: {\color{black} $T$ is bounded from $\mathcal{H}^{r}$ to $\mathcal{H}^{r}$ and satisfies the operator equality}}\label{sec:step5*}
In this section, we would like to prove that the \textcolor{black}{$F$-equivalence} transformation $T$ is bounded and the operator equality \eqref{eq:eqop} holds at least in some $\mathcal{H}^s$ space. More precisely, we prove the following Lemma:
\begin{lem}\label{Lemma-T-bounded}
    The operator $T$ given by \eqref{eq:defT} is a bounded operator from $\mathcal{H}^r$ to $\mathcal{H}^r$ for any $r\in(1/2-\alpha+\gamma,\alpha-1/2).$ Moreover, we have the following operator equality,
    \begin{align}
     T(\mathcal{A}+BK)=(\mathcal{A}-\lambda I)T\quad \text{in}\ \mathcal{L}(\mathcal{H}^{\alpha/2+s},\mathcal{H}^{-\alpha/2+s}), \quad \forall s\in \left(-\frac{\alpha-1}{2}+\gamma,\frac{\alpha-1}{2}-\gamma\right).  \label{opeq-space}
    \end{align}
\end{lem}

\begin{proof}
We recall here that for any $r\in \left(\frac{1}{2}-\alpha, \alpha-\frac{1}{2}\right)$ --and therefore for any $r\in \left(\frac{1}{2}-\alpha+\gamma, \alpha-\frac{1}{2}\right)$-- the operator $T$ is defined as $T:\ n^{-r}\varphi_n\mapsto -n^{-r}K_n\tau q_n$ with $\tau$ an isomorphism from $\mathcal{H}^r$ to $\tau(\mathcal{H}^r)$ and $(K_n)_{n\in \mathbb{N}^*}$ uniformly bounded. Since the space $\tau(\mathcal{H}^r)$ is neither equal nor included in $\mathcal{H}^r$ a priori, what we are able to say is that $T$ is defined from $\mathcal{H}^r$ to $\tau(\mathcal{H}^r).$ Here we will exploit the regularity on $K_n b_n$ to first show that $T$ is well defined from $\mathcal{H}^r$ to itself and then deduce that it is bounded.\\

Let $r\in \left(\frac{1}{2}-\alpha+\gamma, \alpha-\frac{1}{2}\right)$ and 
 $a=\displaystyle\sum_{n\in \mathbb{N}^*}a_n n^{-r}\varphi_n\in \mathcal{H}^r,$ the following holds
\begin{align*}
    Ta&=-\sum_{n\in \mathbb{N}^*}a_n K_n n^{-r}\tau q_n\\
    &=-\sum_{n\in \mathbb{N}^*}a_n K_n n^{-r}\sum_{p\in \mathbb{N}^*}\frac{b_p\varphi_p}{\lambda_n-\lambda_p+\lambda}\\
    &=-\frac{1}{\lambda}\sum_{n\in \mathbb{N}^*}a_n K_n n^{-r}b_n\varphi_n-\sum_{n\in \mathbb{N}^*}a_n K_n n^{-r}\sum_{p\in \mathbb{N}^*\setminus\{n\}}\frac{b_p\varphi_p}{\lambda_n-\lambda_p+\lambda}\\
    &=-\frac{1}{\lambda}\sum_{n\in \mathbb{N}^*}a_n K_nb_n n^{-r}\varphi_n-\sum_{n\in \mathbb{N}^*}a_n K_n n^{-r}\tau S_c(\varphi_n).
\end{align*}
By setting 
\begin{align*}
    T_1 a =-\frac{1}{\lambda}\sum_{n\in \mathbb{N}^*}a_n K_nb_n n^{-r}\varphi_n, \quad T_2 a=-\sum_{n\in \mathbb{N}^*}a_n K_n n^{-r}\tau S_c(\varphi_n),
\end{align*} one has $Ta=T_1a+T_2a.$
Notice that since $(a_n)_{n\in \mathbb{N}^*}\in l^2$ and $(K_n b_n)_{n\in \mathbb{N}^*}\in l^{\infty},$ the term $T_1 a$ belongs to $\mathcal{H}^r$ and in particular we have:
\begin{align}
    \|T_1 a\|_{\mathcal{H}^r}\lesssim  \|f\|_{\mathcal{H}^r}.\label{T_1}
\end{align} 
So, it suffices to show that the second term $T_2 a$ belongs to $\mathcal{H}^r$ to conclude that $Ta\in \mathcal{H}^r$. 
We have the following claim.
\begin{claim}\label{Claim1}
    For any $r\in \left(\frac{1}{2}-\alpha,\alpha-\frac{1}{2}\right)$ and $s\geq \gamma$, the space $\tau(\mathcal{H}^{r+s})$ endowed with the norm
    \begin{equation}
    \label{eq:normspacetau}
        \|f\|_{\tau(\mathcal{H}^{r+s})} := \|\tau^{-1}f\|_{\mathcal{H}^{r+s}}
    \end{equation}
    is a Hilbert space, continuously embedded in $\mathcal{H}^{r}$.
\end{claim}
\begin{proof}
For any $q\in\mathbb{R}$, since $\tau$ is an isomorphism from the Hilbert space $\mathcal{H}^{q}$ into $\tau(\mathcal{H}^{q})$, the space $\tau(\mathcal{H}^{r+s})$ is a Hilbert space with the norm \eqref{eq:normspacetau}. Consider $g\in \tau(\mathcal{H}^{r+s}),$ there exists $f=\displaystyle\sum_{n\in \mathbb{N}^*}f_n n^{-r-s}\varphi_n\in \mathcal{H}^{r+s}$ such that $g=\tau(f)$, thus
\begin{align*}
    g
    =\sum_{n\in \mathbb{N}^*}(f_n b_n n^{-s}) n^{-r}\varphi_n.
\end{align*}
From \eqref{control-cond-general} (with $\beta=0$) and the fact that $\gamma\leq s$, we have $b_n
\leq c_2 n^{s},$ and we observe that
\begin{align*}
    \|g\|^2_{\mathcal{H}^r}=\|(f_n b_n n^{-s})_{n\in \mathbb{N}^*}\|^2_{l^2}=\sum_{n\in \mathbb{N}^*}|f_n b_n n^{-s}|^2\lesssim \sum_{n\in \mathbb{N}^*}|f_n|^2=\|g\|^2_{\tau(\mathcal{H}^{r+s})}.
\end{align*} 
and this concludes the Claim \ref{Claim1}.
\end{proof}
Going back to the norm of $T_2a$ now, we have
based on Claim \ref{Claim1}, 
\begin{align*}
\|T_2a\|_{\mathcal{H}^r}=\left\|\sum_{n\in \mathbb{N}^*}a_n K_n n^{-r}\tau S_c(\varphi_n)\right\|_{\mathcal{H}^r}\lesssim \left\|\sum_{n\in \mathbb{N}^*}a_n K_n \tau S_c(n^{-r}\varphi_n)\right\|_{\tau(\mathcal{H}^{r+\gamma})}.
\end{align*}
From \eqref{eq:sc} and Lemma \ref{lemma-compact-sum},  $S_c$ is a continuous operator from $\mathcal{H}^{r}$ to $\mathcal{H}^{r+\varepsilon}$ for any $\varepsilon\in[0,\min\{(\alpha-1)/2,\alpha+r-1/2\})$. Since $r\in(1/2-\alpha+\gamma, \alpha-1/2)$ and $\gamma\in(0,(\alpha-1)/2)$ by assumption, $\gamma$ belongs to $[0,\min\{(\alpha-1)/2,\alpha+r-1/2\})$ and $S_{c}$ is a continuous operator \textcolor{black}{from} $\mathcal{H}^{r}$ to $\mathcal{H}^{r+\gamma}$ and, using the isomorphism property of $\tau$ \textcolor{black}{and \eqref{eq:normspacetau},} one has
\begin{align*}
     \left\|\sum_{n\in \mathbb{N}^*}a_n K_n \tau S_c(n^{-r}\varphi_n)\right\|_{\tau(\mathcal{H}^{r+\gamma})}\lesssim \left\|\sum_{n\in \mathbb{N}^*}a_n K_n \tau(n^{-r} \varphi_n)\right\|_{\tau(\mathcal{H}^{r})}.
\end{align*}
Using that $(\tau(n^{-r} \varphi_n))_{n\in \mathbb{N}^*}$ is a Riesz basis in $\tau(\mathcal{H}^{r})$ and $(K_{n})_{n\in\mathbb{N}^{*}}$ is uniformly bounded, we get
\begin{align*}
    \left\|\sum_{n\in \mathbb{N}^*}a_n K_n \tau(n^{-r} \varphi_n)\right\|_{\tau(\mathcal{H}^{r})}\lesssim\sum_{n\in \mathbb{N}^*}|a_n K_n|^2\lesssim \sum_{n\in \mathbb{N}^*}|a_n|^2=\|a\|_{\mathcal{H}^r}.
\end{align*}
and combining this with \eqref{T_1}, it holds that,
\begin{align*}
    \|Ta\|_{\mathcal{H}^r}\lesssim \|a\|_{\mathcal{H}^r},
\end{align*}
and thus
$T$ is a bounded operator from $\mathcal{H}^r$ to itself.

    It remains now to prove the operator equality \eqref{opeq-space} to conclude the lemma. As, from Lemma \ref{lem:regk}, the equality \eqref{eq:tbb} holds in 
    $\tau(\mathcal{H}^{-\alpha/2+s+\gamma})$ for any $s\in (-(\alpha-1)/2,(\alpha-1)/2-\gamma)$, and in particular, from Claim \ref{Claim1}, it holds in $\mathcal{H}^{-\alpha/2+s}$.
    Thus proving \eqref{opeq-space} amounts to show that for any $s\in \left(-\frac{\alpha-1}{2}+\gamma,\frac{\alpha-1}{2}-\gamma\right)$
    \begin{align}
     T\mathcal{A}+BK=(\mathcal{A}-\lambda I)T\quad \text{in}\ \mathcal{L}(\mathcal{H}^{\alpha/2+s},\mathcal{H}^{-\alpha/2+s}). \label{opeq-space*}
    \end{align}
    Observe first that all terms make sense : indeed, $K: n^{-(\alpha/2+s)}\varphi_n\mapsto n^{-(\alpha/2+s)}K_n$ is a  bounded operator  from $\mathcal{H}^{\alpha/2+s}$ to $\mathbb{C}$ for any $s> -(\alpha-1)/2$. On the other hand, $B\in \mathcal{H}^{-\alpha/2}$ by assumption and in fact in $\mathcal{H}^{-\alpha/2+s}$ for any $s< (\alpha-1)/2-\gamma$ from \eqref{control-cond-general}
    So $B$ can be formally seen as an operator from $\mathbb{C}$ to $\mathcal{H}^{-\alpha/2+s}$ for any $s\in (-(\alpha-1)/2,(\alpha-1)/2-\gamma)$. Thus $BK$ is a bounded operator from $\mathcal{H}^{\alpha/2+s}$ to $\mathcal{H}^{-\alpha/2+s}$ for any $s\in (-(\alpha-1)/2,(\alpha-1)/2-\gamma).$ Similarly, one can show that $\mathcal{A}T$ and $T\mathcal{A}$ are bounded operators from $\mathcal{H}^{\alpha/2+s}$ to $\mathcal{H}^{-\alpha/2+s}$ for any $s\in (-(\alpha-1)/2+\gamma,(\alpha-1)/2-\gamma)$, 
    since $T$ is a bounded operator from $\mathcal{H}^{r}$ into itself for any $r\in(1/2-\alpha+\gamma, \alpha-1/2)$.\\
    
    To show \eqref{opeq-space*}, it suffices to check that it holds against $n^{-\alpha/2-s}\varphi_n$ for any $n\in \mathbb{N}^*$ and $s\in (-(\alpha-1)/2+\gamma,(\alpha-1)/2-\gamma)$ which in turn amounts to show that it holds against any $\varphi_n$ as the operators are linear. From the definition of $T$ (see \eqref{eq:defT}) we have (in $\mathcal{H}^{-\alpha/2+s})$
    \begin{align*}
        [T\mathcal{A}+BK-(\mathcal{A}-\lambda I)T]\varphi_n&=\lambda_n T\varphi_n+BK_n-(\mathcal{A}-\lambda I)T\varphi_n\\
        &=\lambda_n (-K_n)\sum_{p\in \mathbb{N}^*}\frac{b_p\varphi_p}{\lambda_n-\lambda_p+\lambda}+BK_n-(-K_n)\sum_{p\in \mathbb{N}^*}\frac{b_p(\lambda_p-\lambda)\varphi_p}{\lambda_n-\lambda_p+\lambda}\\
        &=(-K_n)\sum_{p\in \mathbb{N}^*}b_p\varphi_p+BK_n\\
        &=0.
    \end{align*}

\begin{rmk}
\label{rmk:extension}
    In fact, the operator equality \eqref{opeq-space*} also holds for $s\in\left(-\frac{\alpha-1}{2}, \textcolor{black}{\frac{\alpha-1}{2}-\gamma}\right).$ This is the consequence of the fact that for any $s\in\left(-\frac{\alpha-1}{2}, -\frac{\alpha-1}{2}+\gamma\right),$  \eqref{opeq-space*} holds in $\mathcal{L}(\mathcal{H}^{\alpha/2+s},\tau(\mathcal{H}^{\alpha/2+s}))$ and both $BK$ and $(\mathcal{A}-\lambda I)T$ belong to $\mathcal{L}(\mathcal{H}^{\alpha/2+s},\mathcal{H}^{-\alpha/2+s})$. As a consequence, one can show that $T$ is in fact a bounded operator from $\mathcal{H}^{r}$ to itself for any $r\in(1/2-\alpha,\alpha-1/2)$ (see Appendix \ref{app:regT}).
\end{rmk}
\end{proof}

\subsection{Step 6: $T$ is an isomorphism from $\mathcal{H}^{r}$ to itself}
\label{sec:step6}
 To show that $T$ is an isomorphism from $\mathcal{H}^{r}$ to itself, we first show that $\ker T=\{0\}$ in $\mathcal{H}^{-\alpha/2}$, then we deduce that $(K_{n}b_n)_{n\in\mathbb{N}}$ is bounded by below and we deduce that an intermediary operator ($\tilde{\tau}$ given below in Lemma \ref{lem-tilde-tau}) is an isomorphism and further that $T$ is an isomorphism from $\mathcal{H}^{r}$ into itself. In other words, we first show the following Lemma
\begin{lem}\label{ker T}
    $\ker T^*=\{0\}\quad \text{in} \ \mathcal{H}^{-\alpha/2}.$
\end{lem}
\begin{proof}
For that aim, we will perform the proof originally used in \cite{Coron2018-rapid} for the Schrödinger equation and more recently in \cite{Gagnon2022-fredholm-laplacien,Gagnon2022-fredholm}. The idea is to prove the following steps 
\begin{itemize}
    \item There exists $\rho\in \mathbb{C}$ such that $\mathcal{A}+BK+\lambda I+\rho I$ and $\mathcal{A}+\rho I$ are invertible operator from $\mathcal{H}^{\alpha/2}$ to $\mathcal{H}^{-\alpha/2}.$
    \item For such $\rho$, if $\ker T^*\neq \{0\},$ then $(\mathcal{A}+\rho I)^{-1}$ has an eigenvector $h$ which belongs to $\ker T^*.$
    \item No eigenvector of $(\mathcal{A}+\rho I)^{-1}$ belong to $\ker T^*.$
\end{itemize}
From the two last steps, it holds that $\ker T^*=\{0\}$ in $\mathcal{H}^{-\alpha/2}.$ \textcolor{black}{Since the same argument is already presented in \cite{Coron2018-rapid,Gagnon2022-fredholm-laplacien,Gagnon2022-fredholm}, we postpone the rigorous proof} in Appendix \ref{sec-Appendix-kerT}.
\end{proof}

Now, we want to show the following Lemma:
\begin{lem} \label{lemma-T-invertible}
    For any $r\in (1/2-\alpha+\gamma,\alpha-1/2),$ $T$ is an isomorphism from $\mathcal{H}^r$ to itself.
\end{lem}
Before proving this, let us first show the following:
\begin{lem}\label{lem-tilde-tau}
    For any $r\in (1/2-\alpha+\gamma,\alpha-1/2),$ the operator $\tilde \tau: n^{-r}\varphi_n\mapsto \frac{1}{b_n}n^{-r}\tau q_n$ is a Fredholm operator of index $0$ and an isomorphism from $\mathcal{H}^r$ to $\mathcal{H}^r.$
\end{lem}
\begin{proof}[Proof of Lemma \ref{lem-tilde-tau}]
    Let $r\in (1/2-\alpha+\gamma,\alpha-1/2),$ and $f=\displaystyle\sum_{n\in \mathbb{N^*}}f_n n^{-r}\varphi_n\in \mathcal{H}^r.$ It holds that
    \begin{align*}
        \tilde \tau(f)&=\sum_{n\in \mathbb{N^*}}\frac{f_n}{b_n} n^{-r}\tau q_n\\
        &=\sum_{n\in \mathbb{N^*}}\frac{f_n}{b_n} n^{-r}\sum_{p\in \mathbb{N}^*}\frac{b_{p}\varphi_p}{\lambda_n-\lambda_p+\lambda}\\
        &=\frac{1}{\lambda}\sum_{n\in \mathbb{N^*}}f_n n^{-r}\varphi_n+\sum_{n\in \mathbb{N^*}}\frac{f_n}{b_n} n^{-r}\sum_{p\in \mathbb{N}^*\setminus\{n\}}\frac{b_{p}\varphi_p}{\lambda_n-\lambda_p+\lambda}\\
        &=\frac{1}{\lambda}f+\tilde\tau_c(f)
    \end{align*}
    where 
    \begin{align*}
        \tilde\tau_c(f):=\sum_{n\in \mathbb{N^*}}\frac{f_n}{b_n} n^{-r}\sum_{p\in \mathbb{N}^*\setminus\{p\}}\frac{b_{p}\varphi_p}{\lambda_n-\lambda_p+\lambda}.
    \end{align*}
    We first show that $\tilde \tau_c$ is a compact operator on $\mathcal{H}^r.$ Using Fubini's Theorem it holds that
    \begin{align*}
        \sum_{n\in \mathbb{N^*}}\frac{f_n}{b_n} n^{-r}\sum_{p\in \mathbb{N}^*\setminus\{n\}}\frac{b_{p}\varphi_p}{\lambda_n-\lambda_p+\lambda}=\sum_{p\in \mathbb{N^*}}\varphi_p\left(\sum_{n\in \mathbb{N}^*\setminus\{p\}}\frac{f_n}{b_n} n^{-r}\frac{b_{p}}{\lambda_n-\lambda_p+\lambda}\right).
    \end{align*}
    Let $\varepsilon>0$ to be selected. So we have
    \begin{align*}
        \|\tilde \tau_c(f)\|^2_{\mathcal{H}^{r+\varepsilon}}&=\left\|\sum_{p\in \mathbb{N^*}}\varphi_p\left(\sum_{n\in \mathbb{N}^*\setminus\{p\}}\frac{f_n}{b_n} n^{-r}\frac{b_{p}}{\lambda_n-\lambda_p+\lambda}\right)\right\|^2_{\mathcal{H}^{r+\varepsilon}}\\
        &=\sum_{p\in \mathbb{N^*}}p^{2r+2\varepsilon}|b_p|^2\left|\sum_{n\in \mathbb{N}^*\setminus\{p\}}\frac{f_n}{b_n} n^{-r}\frac{1}{\lambda_n-\lambda_p+\lambda}\right|^2.
    \end{align*}
    Since $(1/b_n)_{n\in \mathbb{N}^*}$ is uniformly bounded from \eqref{control-cond-general} (recall that here $\beta=0$), we get that
   \begin{align*}
    \|\tilde \tau_c(f)\|_{\mathcal{H}^{r+\varepsilon}}\lesssim  \sum_{p\in \mathbb{N^*}}p^{2r+2\varepsilon}|b_p|^2\left|\sum_{n\in \mathbb{N}^*\setminus\{p\}} \frac{f_n n^{-r}}{\lambda_n-\lambda_p+\lambda}\right|^2.
   \end{align*} 
   As $|b_p|\leq c_2 n^{\gamma}$, still from \eqref{control-cond-general} this yields
   \begin{align*}
    \|\tilde \tau_c(f)\|_{\mathcal{H}^{r+\varepsilon}}\lesssim  \sum_{p\in \mathbb{N^*}}p^{2r+2\varepsilon+2\gamma}\left|\sum_{n\in \mathbb{N}^*\setminus\{p\}} \frac{f_n n^{-r}}{\lambda_n-\lambda_p+\lambda}\right|^2.
   \end{align*}
   Since $r\in (1/2-\alpha+\gamma,\alpha-1/2),$ we have that $\gamma<\alpha+r-1/2$ and consequently $\gamma<\varepsilon_r:=\min\left\{\frac{\alpha-1}{2},\alpha+r-\frac{1}{2}\right\}$ since $\gamma<(\alpha-1)/2$ by assumption. So, considering for instance $\varepsilon=(\varepsilon_r-\gamma)/2,$ we obtain from Lemma \ref{lemma-compact-sum} that
   \begin{align*}
       \sum_{p\in \mathbb{N^*}}p^{2r+2\varepsilon+2\gamma}\left|\sum_{n\in \mathbb{N}^*\setminus\{p\}} \frac{f_n n^{-r}}{\lambda_n-\lambda_p+\lambda}\right|^2\lesssim \sum_{n\in \mathbb{N^*}}|f_n|^2.
   \end{align*}
   This ensures that
   \begin{align*}
      \|\tilde \tau_c(f)\|_{\mathcal{H}^{r+\varepsilon}}\lesssim \|f\|_{\mathcal{H}^r}, 
   \end{align*} 
   which in view of Lemma \ref{lem:embedding} concludes that $\tilde \tau_c$ is compact from $\mathcal{H}^r$ to $\mathcal{H}^r.$ Consequently the operator $\tilde \tau$ is a Fredholm operator of index $0.$ Thus, $\tilde \tau$ is an isomorphism from $\mathcal{H}^r$ to $\mathcal{H}^r$ if and only if $\ker \tilde \tau=\{0\}.$ From the expression of $\tilde{\tau}$,
   \begin{align*}
      f=\displaystyle\sum_{n\in \mathbb{N^*}}f_n n^{-r}\varphi_n\in \ker \tilde \tau\quad \Leftrightarrow \quad \sum_{n\in \mathbb{N^*}}\frac{f_n}{b_n} \tau(n^{-r} q_n)=0.
   \end{align*}
   The equality at the right hand side holds in $\mathcal{H}^r$ and in particular in $\tau(\mathcal{H}^r)$ since $\mathcal{H}^r\subset \tau(\mathcal{H}^r).$ So using the fact that $(\tau(n^{-r} q_n))_{n\in \mathbb{N}^*}$ is a Riesz basis of $\tau(\mathcal{H}^r),$ we have
   \begin{align*}
       \sum_{n\in \mathbb{N^*}}\frac{f_n}{b_n} \tau(n^{-r} q_n)=0&\quad \Leftrightarrow \quad \frac{f_n}{b_n}=0, \quad \forall n\in\mathbb{N}^*\\
       &\quad \Leftrightarrow \quad f_n=0\quad \forall n\in\mathbb{N}^*.
   \end{align*}
   Thus, we have shown that $\ker \tilde \tau=\{0\},$ hence, $\tilde \tau$ is an isomorphism from $\mathcal{H}^r$ to itself for any $r\in (1/2-\alpha+\gamma,\alpha-1/2).$
\end{proof}

{\color{black} We are now able to state the proof of Lemma \ref{lemma-T-invertible}.\\
\begin{proof}[Proof of Lemma \ref{lemma-T-invertible}]
In view of \eqref{eq:defT} and Lemma \ref{lem-tilde-tau}, we may rewrite now $T$ as follows
\begin{align*}
    T:n^{-r}\varphi_n\mapsto -K_n b_n \tilde\tau (n^{-r}\varphi_n).
\end{align*}
Based on Lemma \ref{lem-tilde-tau}, $\tilde \tau$ is an isomorphism from $\mathcal{H}^r$ to $\mathcal{H}^r.$ This ensures that $(\tilde{\tau} (n^{-r}\varphi_n))_{n\in \mathbb{N}^*}$ is a Riesz basis in $\mathcal{H}^r.$  Since $(n^{-r}\varphi_n)_{n\in \mathbb{N}^*}$ is also a Riesz basis in $\mathcal{H}^r$, and $(K_n b_n)_{n\in\mathbb{N}}\in l^{\infty}$ from Lemma \ref{Lemma-K_n}, we just have to prove that $(K_n b_n)_{n\in\mathbb{N}}$ is uniformly bounded from below to conclude that 
$T$ is an isomorphism in $\mathcal{H}^r$. 
From Lemma \ref{ker T}, $\ker T^{*}=\{0\}$ in $\mathcal{H}^{-\alpha/2}$. In fact, from the decomposition \eqref{K_n-expression} and the fact that $\tilde{\tau}$ is a Fredholm operator of order 0, we can deduce that $T$ is a Fredholm operator of order 0 from $\mathcal{H}^{-\alpha/2}$ into itself (see Lemma \ref{lem:Tfredholm} in Appendix \ref{sec-Appendix-kerT}). Hence, $\ker T=\{0\}$ in $\mathcal{H}^{-\alpha/2}$.
Let us assume by contradiction that there exists $n_0\in \mathbb{N}^*$ such that $K_{n_0} b_{n_0}=0.$ We have 
\begin{align*}
    T(n_0^{-r}\varphi_{n_0})=K_{n_0}b_{n_0}\tilde \tau (n_0^{-r}\varphi_{n_0})=0.
\end{align*}
This means that $n_0^{-r}\varphi_{n_0}\in \text{ker} T,$ and since $\text{ker} T=\{0\},$  $n_0^{-r}\varphi_{n_0}=0$ which is a contradiction. Therefore for any $n\in \mathbb{N}^*$  
\begin{align}
    K_n b_n\neq 0.\label{k_nb_n0}
\end{align}
Recall that from Lemma \eqref{Lemma-K_n}, 
\begin{align*}
    |K_n b_n|=|\lambda+k_n|,
\end{align*} and $(k_n n^{\varepsilon})_{n\in \mathbb{N}^*}\in l^{\infty}$ for some $\varepsilon>0,$ which means that $k_n\rightarrow 0$ when $n\rightarrow +\infty.$ Then, there exists $N(\lambda)>0$ such that for any $n\geq N(\lambda),$
\begin{align}
    |K_n b_n|\geq \frac{\lambda}{2}>0.\label{Knbn-up}
\end{align}
Combining \eqref{k_nb_n0} with \eqref{Knbn-up} it holds that for any $n\in \mathbb{N}^*,$
\begin{align*}
    |K_n b_n|\geq \min\left\{\frac{\lambda}{2}, \underset{n\leq N(\lambda)}{\min}|K_n b_n|\right\}>0.
\end{align*}
Therefore, $K_n b_n$ is bounded from below and this concludes
the proof of Lemma \ref{lemma-T-invertible}.
\end{proof}
}

\subsection{Step 7: Well-posedness and proof of Proposition \ref{th:wellposed}}
\label{sec:wellposed}
Since, we have built the isomorphism $T$ and the feedback law $K,$ for which \eqref{eq:tbb} and \eqref{eq:eqoplin}  hold, we need to show the well-posedness of the closed loop system \eqref{sysdepart1}. We will adapt the method given in \cite{Coron2022-stabilization} for the well-posedness of closed loop water tank or in \cite{Gagnon2022-fredholm} for the well-posedness of closed loop systems described by skew-adjoint operators. 

Since, this method relies on the semigroup theory of the operator $\mathcal{A}+BK$, we first provide a description and some properties about its domain:

\begin{lem}\label{Lemma-Dr}
    For any $r\in(1/2-\alpha +\gamma,\alpha-1/2-\gamma),$ the domain $D_r(\mathcal{A}+BK)$ of $\mathcal{A}+BK$ is a Hilbert space, dense in $\mathcal{H}^r$ and 
    \begin{equation}
        D_r(\mathcal{A}+BK)=T^{-1}(\mathcal{H}^{r+\alpha}).\label{D_r-expression}
    \end{equation}
\end{lem}
\begin{proof}
Recall that, what we call $D_r(\mathcal{A}+BK)$ for $r\in (1/2-\alpha +\gamma,\alpha-1/2-\gamma)$ is the following:
\begin{align*}
  D_r(\mathcal{A}+BK):=\{f\in \mathcal{H}^r, (\mathcal{A}+BK)f\in \mathcal{H}^r\}.  
\end{align*}
 We start by showing \eqref{D_r-expression}. For any $f\in D_r(\mathcal{A}+BK)\subset{\mathcal{H}^r},$ we have $\mathcal{A}f\in \mathcal{H}^{r-\alpha}$ and then by definition of $D_{r}(\mathcal{A}+BK)$ this implies that $BKf\in\mathcal{H}^{r-\alpha}$. As $r<\alpha-1/2-\gamma,$ there exist $\varepsilon >0$ such that $BKf\in\mathcal{H}^{-1/2-\gamma-\varepsilon}$. 
  This shows that $Kf\in \mathbb{C}$ because $B\in \mathcal{H}^{-1/2-\gamma-\varepsilon}$ for any $\varepsilon>0.$ Since $Kf\in \mathbb{C},$ it holds that $BKf\in\mathcal{H}^{-1/2-\gamma-\varepsilon}$
  for any $\varepsilon>0$. Combining this with the definition of $D_r(\mathcal{A}+BK)$ ensures that {\color{black} either
  \begin{itemize}
      \item $r< -1/2-\gamma-\varepsilon$ and $\mathcal{A}f\in \mathcal{H}^{r}$,
      \item or $r\geq -1/2-\gamma-\varepsilon$ and $\mathcal{A}f\in \mathcal{H}^{-1/2-\gamma-\varepsilon}$.
  \end{itemize}
  In the first case, this gives 
  $f\in \mathcal{H}^{r+\alpha}$ and since $r+\alpha \in (1/2,\alpha-1/2-\gamma-\varepsilon)\subset(1/2-\alpha+\gamma,\alpha-1/2)$, $\mathcal{H}^{r+\alpha}= T^{-1}(\mathcal{H}^{r+\alpha})$ from Lemma \ref{lemma-T-invertible} which ends the proof. In the second case,} 
  $\mathcal{A}f\in \mathcal{H}^{-1/2-\gamma-\varepsilon},$ for any $\varepsilon>0$. \textcolor{black}{If $\mathcal{A}$ is invertible, we get directly}
$f=\mathcal{A}^{-1}\mathcal{A}f\in \mathcal{H}^{\alpha-1/2-\gamma-\varepsilon}$ for any $\varepsilon>0.$ \textcolor{black}{If $\mathcal{A}$ is not invertible, then there exists a unique $n_{0}\in \mathbb{N}^{*}$ such that $\lambda_{n_{0}}=0$ and, defining $\mathcal{A}^{-1}$ on $\overline{\text{Span}_{n\neq n_{0}}(n^{-r}\varphi_{n})}$, we get $f=\langle f,n_{0}^{-r}\varphi_{n_{0}}\rangle_{\mathcal{H}^{r}}n_{0}^{-r}\varphi_{n_{0}}+\mathcal{A}^{-1}\mathcal{A}f\in \mathcal{H}^{\alpha-1/2-\gamma-\varepsilon}$. In both cases, t}his shows that $D_r(\mathcal{A}+BK)\subset\mathcal{H}^{\alpha-1/2-\gamma-\varepsilon}.$ So in view of the operator equality \eqref{opeq-space}, for any $\varepsilon\in (0,(\alpha-1)-2\gamma)$, $\alpha/2-1/2-\gamma-\varepsilon \in (\gamma-(\alpha-1)/2,(\alpha-1)/2-\gamma)$ and therefore
 it holds that
 \begin{align*}
     T(\mathcal{A}+BK)f=(\mathcal{A}-\lambda I)Tf\in \mathcal{H}^{-1/2-\gamma-\varepsilon},\quad \forall f\in D_r(\mathcal{A}+BK).
 \end{align*}
 Notice that by the property of $T$ and the definition of $D_r(\mathcal{A}+BK),$ $T(\mathcal{A}+BK)f\in \mathcal{H}^r$ for any $f\in D_r(\mathcal{A}+BK),$ and hence $(\mathcal{A}-\lambda I)Tf\in \mathcal{H}^{r}.$ So the operator equality holds in the following appropriate setting: for any $r\in (1/2-\alpha+\gamma,\alpha-1/2-\gamma),$
 \begin{align}
     T(\mathcal{A}+BK)f=(\mathcal{A}-\lambda I)Tf\in \mathcal{H}^{r},\quad \forall f\in D_r(\mathcal{A}+BK).\label{opeq-space-appro}
 \end{align}
 Consider now $f\in D_r(\mathcal{A}+BK),$ we get from \eqref{opeq-space-appro} that $(\mathcal{A}-\lambda I)Tf\in \mathcal{H}^r,$ and then $Tf\in \mathcal{H}^{r+\alpha}.$ This implies that $f\in T^{-1}(\mathcal{H}^{r+\alpha}),$ and thus show that $D_r(\mathcal{A}+BK)\subset T^{-1}(\mathcal{H}^{r+\alpha}).$\\
 Conversely, let us consider $f\in T^{-1}(\mathcal{H}^{r+\alpha})\subset \mathcal{H}^r.$ Then, since $r<\alpha-1/2-\gamma,$ there exists $\bar{\varepsilon} >0$ such that for any $\varepsilon\in (0,\bar{\varepsilon}),$ $f\in \mathcal{H}^{\alpha-1/2-\varepsilon-\gamma}.$ So by considering $\varepsilon$ small enough we can apply the operator equality \eqref{opeq-space} to $f$ and get 
 \begin{align}
    T(\mathcal{A}+BK)f=(\mathcal{A}-\lambda I)Tf.\label{eq:T-1}
 \end{align} Notice that $f\in T^{-1}(\mathcal{H}^{r+\alpha})$ implies that $Tf\in \mathcal{H}^{r+\alpha},$ and thus $(\mathcal{A}-\lambda I)Tf\in \mathcal{H}^{r}.$ This ensures from \eqref{eq:T-1} that $T(\mathcal{A}+BK)f\in \mathcal{H}^{r}.$ Since $T$ is an isomorphism from $\mathcal{H}^{r}$ to itself, it then holds that $(\mathcal{A}+BK)f\in \mathcal{H}^{r}.$ Hence $f\in D_r(\mathcal{A}+BK),$ meaning that $T^{-1}(\mathcal{H}^{r+\alpha})\subset D_r(\mathcal{A}+BK).$ This concludes the proof of \eqref{D_r-expression}.

 Since $T$ is an isomorphism from $\mathcal{H}^{r}$ to itself, and from the density of $\mathcal{H}^{r+\alpha}$ in $\mathcal{H}^{r},$ it immediately holds based on \eqref{D_r-expression} that $D_r(\mathcal{A}+BK)$ is dense in $\mathcal{H}^{r}$ \textcolor{black}{(note that for any $f\in\mathcal{H}^{r}$ there exists $g\in \mathcal{H}^{r}$ such that $f = T^{-1}g$).}

It remains now to prove that $D_r(\mathcal{A}+BK)$ is a Hilbert space to conclude Lemma \ref{Lemma-Dr}. For that, we just need to prove that $D_r(\mathcal{A}+BK)$ is complete since $D_r(\mathcal{A}+BK)\subset \mathcal{H}^r$ which is a Hilbert space. So, let consider any Cauchy sequence $(f_n)_n\subset D_r(\mathcal{A}+BK).$ It holds that
\begin{align*}
    \|f_n-f_p\|_{D_r(\mathcal{A}+BK)}\underset{n,p\rightarrow+\infty}{\longrightarrow} 0.
\end{align*}
Recall that $\|f_n-f_p\|_{D_r(\mathcal{A}+BK)}:=\|f_n-f_p\|_{\mathcal{H}^r}+\|(\mathcal{A}+BK)(f_n-f_p)\|_{\mathcal{H}^r},$ we have
\begin{align}
    \|f_n-f_p\|_{\mathcal{H}^r}+\|(\mathcal{A}+BK)(f_n-f_p)\|_{\mathcal{H}^r}\underset{n,p\rightarrow+\infty}{\longrightarrow} 0.\label{eq:cauchy}
\end{align}
Since $(f_n)_n\subset D_r(\mathcal{A}+BK),$ by using the operator equality \eqref{opeq-space-appro}, this implies that
\begin{align*}
    \|(\mathcal{A}-\lambda I)(Tf_n-Tf_p)\|_{\mathcal{H}^r}=\|T(\mathcal{A}+BK)(f_n-f_p)\|_{\mathcal{H}^r}\lesssim \|(\mathcal{A}+BK)(f_n-f_p)\|_{\mathcal{H}^r}\underset{n,p\rightarrow+\infty}{\longrightarrow} 0.
\end{align*}
This combining with \eqref{eq:cauchy}, ensures that 
\begin{align*}
  \|\mathcal{A}(Tf_n-Tf_p)\|_{\mathcal{H}^r}  \underset{n,p\rightarrow+\infty}{\longrightarrow} 0,\quad \text{and}\quad \|Tf_n-Tf_p\|_{\mathcal{H}^r}  \underset{n,p\rightarrow+\infty}{\longrightarrow} 0.
\end{align*} Thus, we obtain
\begin{align*}
   \|Tf_n-Tf_p\|_{\mathcal{H}^{r+\alpha}}  \underset{n,p\rightarrow+\infty}{\longrightarrow} 0.
\end{align*}
This shows that $(Tf_n)_n$ is a Cauchy sequence in $\mathcal{H}^{r+\alpha}.$  Then, since $\mathcal{H}^{r+\alpha}$ is complete,
there exists $g\in \mathcal{H}^{r+\alpha}$ such that $T f_n
\longrightarrow g$ in $\mathcal{H}^{r+\alpha}$.  As $T$ is an isomorphism on $\mathcal{H}^{r}$, and $\mathcal{H}^{r+\alpha}\subset \mathcal{H}^{r}$ there exists a unique $f\in \mathcal{H}^{r}$ such that $Tf=g$ and $\|T f_n-Tf\|_{\mathcal{H}^{r+\alpha}}
\longrightarrow 0,$ which in turn implies, from the compact embedding of $\mathcal{H}^{r+\alpha}$ in $\mathcal{H}^{r}$ and the isomorphism property of $T$ in $\mathcal{H}^{r}$
\begin{align}
    \|f_n-f\|_{\mathcal{H}^{r}}\underset{n\rightarrow+\infty}{\longrightarrow} 0.\label{eq:conv1}
\end{align}
We also have, \textcolor{black}{since $f\in T^{-1}\mathcal{H}^{r+\alpha}=D_{r}(\mathcal{A}+BK)$ and using \eqref{opeq-space-appro}}
\begin{align}
    \|(\mathcal{A}+BK)(f_n-f)\|_{\mathcal{H}^{r}}&\lesssim \|T(\mathcal{A}+BK)(f_n-f)\|_{\mathcal{H}^{r}}\nonumber\\
    &=\|(\mathcal{A}-\lambda I)T(f_n-f)\|_{\mathcal{H}^{r}}\lesssim \|T f_n-Tf\|_{\mathcal{H}^{r+\alpha}}\underset{n\rightarrow+\infty}{\longrightarrow} 0.\label{eq:conv2}
\end{align}
Combining \eqref{eq:conv1} and \eqref{eq:conv2}, it holds that:
\begin{align*}
    \|f_n-f\|_{D_r(\mathcal{A}+BK)}=\|f_n-f\|_{\mathcal{H}^{r}}+\|(\mathcal{A}+BK)(f_n-f)\|_{\mathcal{H}^{r}}\underset{n\rightarrow+\infty}{\longrightarrow} 0.
\end{align*}
Hence, $D_r(\mathcal{A}+BK)$ is complete and therefore a Hilbert space.
\end{proof}
As we know now that the domain $D_r(\mathcal{A}+BK)$ is a Hilbert space and it is not an empty set, we are able to state the following lemma:
\begin{lem}\label{Lemma-Dr-semigroup}
    The operator $\mathcal{A}+BK$ with domain $D_r(\mathcal{A}+BK)$ generates a $C^0$ semigroup on $\mathcal{H}^r.$
\end{lem}
\begin{proof}
    Let us first notice that for any $r\in (1/2-\alpha+\gamma,\alpha-1/2-\gamma),$ 
    $D_r(\mathcal{A}-\lambda I):=\{f\in \mathcal{H}^r,\ (\mathcal{A}-\lambda I)f\in \mathcal{H}^r\}=\mathcal{H}^{r+\alpha}.$ Since $\mathcal{A}$ generates a $C^0$ semigroup on $\mathcal{H}^r,$ then $\mathcal{A}-\lambda I$ generates also a $C^0$ semigroup which we denote by $e^{t(\mathcal{A}-\lambda I)}.$ Let us now define the following $C^0$ semigroup on $\mathcal{H}^r,$
    \begin{align}
    \label{eq:semi}
        \mathcal{S}(t):=T^{-1}e^{t(\mathcal{A}-\lambda I)}T, \quad \forall t\geq 0.
    \end{align}
    Since $T$ is an isomorphism \textcolor{black}{on $\mathcal{H}^{r}$}, \textcolor{black}{the fact that} $D_r(\mathcal{A}-\lambda I)=\mathcal{H}^{r+\alpha},$ and Lemma \ref{Lemma-Dr}, it holds that
    \begin{align*}
        \left\{f\in \mathcal{H}^r,\ \underset{t\rightarrow 0^+}{\lim}\frac{\mathcal{S}(t)f-f}{t}\ \text{exists in }\mathcal{H}^{r}\right\}
        &=\left\{f\in \mathcal{H}^r,\ \underset{t\rightarrow 0^+}{\lim}T^{-1}\left(\frac{e^{t(\mathcal{A}-\lambda I)}Tf-Tf}{t}\right)\ \text{exists in }\mathcal{H}^{r}\right\}\\
        &=\left\{(f=T^{-1}g)\in \mathcal{H}^r,\ \underset{t\rightarrow 0^+}{\lim}T^{-1}\left(\frac{e^{t(\mathcal{A}-\lambda I)}g-g}{t}\right)\ \text{exists in }\mathcal{H}^{r}\right\}\\
        &=\left\{(f=T^{-1}g)\in \mathcal{H}^r,\ \underset{t\rightarrow 0^+}{\lim}\left(\frac{e^{t(\mathcal{A}-\lambda I)}g-g}{t}\right)\ \text{exists in }\mathcal{H}^{r}\right\}\\
        &=\left\{(f=T^{-1}g),\ g\in D_r(A-\lambda I)=\mathcal{H}^{r+\alpha}\right\}\\
        &=T^{-1}(\mathcal{H}^{r+\alpha})=D_r(\mathcal{A}+BK).
    \end{align*}
    This shows that the domain of the operator which generates $\mathcal{S}(t)$ is $D_r(\mathcal{A}+BK).$ 
    
    Moreover, for any $f\in D_r(\mathcal{A}+BK),$ we have $Tf\in \mathcal{H}^{r+\alpha}$ and then it holds that
    \begin{align*}
     \underset{t\rightarrow 0^+}{\lim}\frac{\mathcal{S}(t)f-f}{t}&=\underset{t\rightarrow 0^+}{\lim}T^{-1}\left(\frac{e^{t(\mathcal{A}-\lambda I)}Tf-Tf}{t}\right) \\
     &=T^{-1}(\mathcal{A}-\lambda I)Tf.
    \end{align*}
    So using the operator equality \eqref{opeq-space-appro} \textcolor{black}{(since $f\in D_{r}(\mathcal{A}+BK)$)} we get that
    \begin{align*}
        \underset{t\rightarrow 0^+}{\lim}\frac{\mathcal{S}(t)f-f}{t}&=(\mathcal{A}+BK)f.
    \end{align*}
  Hence, the operator $\mathcal{A}+BK$  with domain $D_r(\mathcal{A}+BK)$ is the generator of $C^0$ semigroup $\mathcal{S}(t)$ and this concludes Lemma \ref{Lemma-Dr-semigroup}.
\end{proof}
\begin{proof}[Proof of Proposition \ref{th:wellposed}]
   Based on Lemma \ref{Lemma-Dr-semigroup}, we know that for any $i\in \{1,\cdots, m\},$ the operator $\mathcal{A}+B_iK_i$ generates a $C^0$ semigroup $\mathcal{S}_i(t)$ on $\textcolor{black}{\mathcal{H}^{r_{i}}_i}$ for any $\textcolor{black}{r_{i}}\in (1/2-\alpha_i+\gamma_{i},\alpha_i-1/2-\gamma_{i})$. Note that $(\mathcal{S}_{i}(t))_{t\geq 0}$ refers here to what was referred to as $(\mathcal{S}(t))_{t\geq 0}$ in the proof of \textcolor{black}{Lemma \ref{Lemma-Dr-semigroup}} since we do not drop the index $i$ anymore for the proof of Proposition \ref{th:wellposed}. Now, we would like to show that the operator $\mathcal{A}+BK=\mathcal{A}+\displaystyle\sum_{i=1}^mB_iK_i$ with domain $\textcolor{black}{D_{ \vec r}}(\mathcal{A}+BK):=\displaystyle\sum_{i=1}^m \textcolor{black}{D_{r_{i}}}(\mathcal{A}+B_iK_i)$ generates the $C^0$ semigroup 
   \begin{equation}
       \mathcal{S}(t) := \displaystyle\sum_{i=1}^m \mathcal{S}_i(t)
   \end{equation}
   on \textcolor{black}{$\mathcal{H}^{\vec{r}}:=\mathcal{H}^{r_{1}}_1+\cdots+\mathcal{H}^{r_{m}}_m.$} Recall that $\mathcal{S}_i(t)=T^{-1}_ie^{t(\mathcal{A}-\lambda I)}T_i,$ and since $T_i\varphi_n^j=0$ for any $i\neq j,$ it holds that $\mathcal{S}_i(t)\varphi_n^j=0$ for any $i\neq j.$ So for any $f^j\in \mathcal{H}^{r_j},$ we have that $\mathcal{S}_i(t)f^j=0$ for any $i\neq j.$ This ensures in particular that $\mathcal{S}_i(t)\mathcal{S}_j(s)f^k=0$ for any $i\neq j$, $k\in\{1,...,m\}$ and $s,t\geq 0$. Combining this with the fact that  $(\mathcal{S}_{i}(t))_{t\geq0}$ is a $C^{0}$ semigroup on $\mathcal{H}^{\textcolor{black}{r_{i}}}_{\textcolor{black}{i}},$ it holds that $\mathcal{S}(t)$ is a $C^0$ semi-group of $\textcolor{black}{\mathcal{H}^{\vec r}}$. Looking at its domain:
   \begin{align*}
       \left\{ f=\sum_{j=1}^m f^j\in \textcolor{black}{\mathcal{H}^{\vec r}}, \ \lim_{t\rightarrow 0}\frac{\mathcal{S}(t)f-f}{t}\ \text{exists}\right\}&=\left\{ f=\sum_{j=1}^m f^j\in \textcolor{black}{\mathcal{H}^{\vec r}}, \ \lim_{t\rightarrow 0}\frac{\sum_{i=1}^m \mathcal{S}_i(t)f-f}{t}\ \text{exists}\right\}\\
       &=\left\{ f=\sum_{j=1}^m f^j\in \textcolor{black}{\mathcal{H}^{\vec r}}, \ \lim_{t\rightarrow 0}\sum_{i=1}^m\frac{\mathcal{S}_i(t)f^i-f^i}{t}\ \text{exists}\right\}\\
       &=\left\{ f=\sum_{j=1}^m f^j\in \textcolor{black}{\mathcal{H}^{\vec r}}, \ \lim_{t\rightarrow 0}\frac{\mathcal{S}_i(t)f^i-f^i}{t}\ \text{exists}\right.\\
       &\left.\forall i\in \{1,\cdots,m\}\right\}\\
       &=\sum_{i=1}^m \textcolor{black}{D_{r_i}}(\mathcal{A}+B_iK_i)=\textcolor{black}{D_{\vec r}}(\mathcal{A}+BK).
   \end{align*}
   This shows that $\textcolor{black}{D_{\vec r}}(\mathcal{A}+BK)$ is the domain of the corresponding generator to $\mathcal{S}(t)$.
   Moreover, for any $f\in \textcolor{black}{D_{\vec r}}(\mathcal{A}+BK),$ there exists $(f^1,\cdots,f^m)\in \textcolor{black}{\mathcal{H}^{r}_1\times \cdots \times \mathcal{H}^{r}_m}$ such that $f=\sum_{i=1}^m f^i,$ and 
   \begin{align*}
       \lim_{t\rightarrow 0^+}\frac{ \mathcal{S}(t)f-f}{t}&=\lim_{t\rightarrow 0^+}\frac{\sum_{i=1}^m \mathcal{S}_i(t)f^i-\sum_{i=1}^m f^i}{t}\\
       &=\sum_{i=1}^m\lim_{t\rightarrow 0^+}\frac{\mathcal{S}_i(t)f^i-f^i}{t}\\
       &=\sum_{i=1}^m(\mathcal{A}+B_iK_i)f^i=(\mathcal{A}+BK)f
   \end{align*}
   Therefore, we conclude that the operator $\mathcal{A}+BK$ generates a $C^0$ semigroup on $\textcolor{black}{\mathcal{H}^{\vec r}}$ and this concludes the proof of Proposition \ref{th:wellposed}.
\end{proof}

\subsection{Proof of Theorem \ref{Theo_main_general}}
\label{sec:thm}
Theorem \ref{Theo_main} follows from all the previous steps and the following Lemma, which is an immediate consequence from the fact that $\mathcal{N}$ is countable.

\begin{lem}
\label{lem:lambda*}
    For any $\lambda_{0}>0$ there exists $\lambda>\lambda_{0}$ such that $\lambda\in\mathbb{R}_{+}\setminus\mathcal{N}$.
\end{lem}

\begin{proof}[Proof of Theorem \ref{Theo_main_general}]
Let $\lambda_{0}>0$, and $\lambda>0$ given by Lemma \ref{lem:lambda}. From Proposition \ref{prop:main} there exists $K_{i}\in\mathcal{L}(\mathcal{H}^{\beta_{i}+1/2+\varepsilon};\mathbb{C})$ and $T_{i}\in\mathcal{L}(\mathcal{H}^{\textcolor{black}{\beta_{i}+}r_{i}})$ such that $T_{i}$ is an isomorphism from $\mathcal{H}^{\textcolor{black}{\beta_{i}+}r_{i}}$ into itself for any $\textcolor{black}{r_{i}\in(1/2-\alpha_{i}+\gamma_{i},\alpha_{i}-1/2-\gamma_{i})}$ and maps the system \eqref{sysdepart1} to \eqref{sysarrive1}. Setting
\begin{equation}
K = (K_{1},...,K_{m})^{T},\;\;T=T_{1}+...+T_{m},
\end{equation}
we know from Proposition \ref{th:wellposed} that the system \eqref{closed-loop-system} is well-posed in \textcolor{black}{$\mathcal{H}^{\vec{\beta}+\vec{r}}$} for any \textcolor{black}{$\vec{r}$} satisfying \eqref{eq:condr}. It remains to show that $T$ is an isomorphism from $\textcolor{black}{\mathcal{H}^{\vec{\beta}+\vec{r}}}$ to itself which maps the system \eqref{closed-loop-system} to \eqref{target-system} and that, as a consequence, the system \eqref{closed-loop-system} is exponentially stable. $T$ belongs clearly to $\mathcal{L}(\textcolor{black}{\mathcal{H}^{\vec{\beta}+\vec{r}}})$. We define $T^{-1}:=(T_{1}^{-1}+...+T_{m}^{-1})\in\mathcal{L}(\textcolor{black}{\mathcal{H}^{\vec{\beta}+\vec{r}}})$. Recall that $T_{i}\varphi_{n}^{j} = 0$, $T_{i}^{-1}\varphi_{n}^{j} = 0$ and $K_{i}\varphi_{n}^{j}=0$ for any $j\neq i$ and $n\in\mathbb{N}^{*}$. Thus, for any $f = \sum\limits_{i=1}^{m}\sum\limits_{n\in\mathbb{N}^{*}}f_{n}^{i}\varphi_{n}^{i}\in\textcolor{black}{\mathcal{H}^{\vec{\beta}+\vec{r}}}$, 
\begin{equation}
    T^{-1}Tf = T^{-1}\left(\sum\limits_{i=1}^{m}\sum\limits_{n\in\mathbb{N}^{*}}f_{n}^{i}T_{i}\varphi_{n}^{i}\right) = \sum\limits_{i=1}^{m}\sum\limits_{n\in\mathbb{N}^{*}}f_{n}^{i}\varphi_{n}^{i} = f,
\end{equation}
hence $T$ is an isomorphism from $\textcolor{black}{\mathcal{H}^{\vec{\beta}+\vec{r}}}$ into itself with inverse $T^{-1}$.  To show that $T$ maps \eqref{closed-loop-system} to \eqref{target-system} it suffices to show that the operator equality \eqref{opeq-space} --which currently holds in $\mathcal{L}(\mathcal{H}_{i}^{\alpha_{i}/2+s_{i}},\mathcal{H}_{i}^{-\alpha_{i}/2+s_{i}})$ with $T_{i}$, $B_{i}$ and $K_{i}$ for any $i\in\{1,...,m\}$-- in fact also holds in 
$\mathcal{L}(\mathcal{H}^{\vec \alpha/2+\vec s},\mathcal{H}^{-\vec \alpha/2+\vec s})$ with $T$, $B$ and $K$ where
\begin{equation}
\label{eq:api}
    \alpha = (\alpha_{1},...,\alpha_{m}),\;\;s = (s_{1},...,s_{m}),\;\;\gamma=(\gamma_{1},...,\gamma_{m}),\;\;s_{i}\in\left(-\frac{(\alpha_{i}-1)}{2}+\gamma_{i},\frac{\alpha_{i}-1}{2}-\gamma_{i}\right).
\end{equation}
This is a consequence of the fact that, by definition, $T_{i}\varphi_{n}^{j}=0$, $K_{i}\varphi_{n}^{j}=0$ for any $i\neq j$ and any $n\in\mathbb{N}^{*}$: let $f = \sum_{i=1}^{m}f^{i}$ with $f^{i} = \sum_{n\in\mathbb{N}}f_{n}^{i}\varphi_{n}^{i}$ and consider its smooth approximation $\tilde{f}^{i} = \sum_{n=1}^{N}f_{n}^{i}\varphi_{n}^{i}$. Observe that $K\tilde{f} =(K_{1}\tilde{f}^{1},...,K_{m}\tilde{f}^{m})^{T}$ and $BK\tilde{f} = \sum\limits_{i=1}^{m}B_{i}K_{i}\tilde{f}^{i}$, thus we have
\begin{equation}
    T(\mathcal{A}+BK)\tilde{f} = T\left(\sum\limits_{i=1}^{m}\mathcal{A}\tilde{f}^{i}+B_{i}K_{i}\tilde{f}^{i}\right) = \sum\limits_{i=1}^{m} T_{i}\mathcal{A}\tilde{f}^{i}+B_{i}K_{i}\tilde{f}^{i}
\end{equation}
where we used that, for any $i\in\{1,...,m\}$ and $s\in\mathbb{R}$, we have $\mathcal{A} \in\mathcal{L}(\mathcal{H}_{i}^{s+\alpha_{i}},\mathcal{H}^{s}_{i})$. Thus using \eqref{opeq-space} for each $i\in\{1,...,m\}$,
\begin{equation}
T(\mathcal{A}+BK)\tilde{f} = \sum\limits_{i=1}^{m} T_{i}(\mathcal{A}+B_{i}K_{i})\tilde{f}^{i} = \sum\limits_{i=1}^{m} (\mathcal{A}-\lambda) T_{i}\tilde{f}_{i} = (\mathcal{A}-\lambda) \sum\limits_{i=1}^{m} T_{i}\tilde{f}_{i}=(\mathcal{A}-\lambda I)T \tilde{f}.
\end{equation}
Note that one can pass to the limit (in $\mathcal{H}^{-\vec \alpha/2+s}$) $N\rightarrow +\infty$ in each of this terms since $f\in \mathcal{H}^{\textcolor{black}{\vec \alpha/2}+s}$ and the operator equality \eqref{opeq-space} holds in $\mathcal{L}(\mathcal{H}_{i}^{\textcolor{black}{\alpha_{i}/2}+s},\mathcal{H}_{i}^{-\alpha_{i}/2+s})$ for any $i\in\{1,...,m\}$. As a consequence, 
\begin{equation}
\label{eq:opeq-spacetot}
T(\mathcal{A}+BK) = (\mathcal{A}-\lambda I)T\text{ in }\mathcal{L}(\mathcal{H}^{\vec \alpha/2+s},\mathcal{H}^{-\vec \alpha/2+s})\text{ with \eqref{eq:api}}.
\end{equation}

Finally it remains to show that for any $\mu>0$, $\lambda_{0}$ can be chosen large enough such that the system \eqref{closed-loop-system} is exponentially stable with decay rate $\mu$. Since $\mathcal{A}$ generates a $C^{0}$ semigroup with finite growth bound on $\textcolor{black}{\mathcal{H}^{\vec{\beta}+\vec{r}}}$ (see Lemma \ref{lem:semigroup} in Appendix \ref{app:semigroup}), there exists $M\geq 1$ and $\omega\in\mathbb{R}$ such that for any $f\in\textcolor{black}{\mathcal{H}^{\vec{\beta}+\vec{r}}}$,
\begin{equation}
\|e^{t\mathcal{A}}f\|_{\textcolor{black}{\mathcal{H}^{\vec{\beta}+\vec{r}}}}\leq M e^{\omega t}\|f\|_{\textcolor{black}{\mathcal{H}^{\vec{\beta}+\vec{r}}}},\;\forall\;t\in[0,+\infty),
\end{equation}
and in particular for $\lambda>\lambda_{0}>\mu+\omega$,
\begin{equation}
\|e^{t(\mathcal{A}-\lambda I)}f\|_{\textcolor{black}{\mathcal{H}^{\vec{\beta}+\vec{r}}}}\leq M e^{-\mu t}\|f\|_{\textcolor{black}{\mathcal{H}^{\vec{\beta}+\vec{r}}}},\;\forall\;t\in[0,+\infty).
\end{equation}
Let $u_{0}\in \textcolor{black}{\mathcal{H}^{\vec{\beta}+\vec{r}}}$, the (unique) solution in $C^{0}([0,+\infty);\textcolor{black}{\mathcal{H}^{\vec{\beta}+\vec{r}}})$ to the closed loop system \eqref{closed-loop-system} is $u\;:\;t\rightarrow \mathcal{S}(t)u_{0}$ and, using \eqref{eq:semi}
\begin{equation}
    \|\mathcal{S}(t)u_{0}\|_{\textcolor{black}{\mathcal{H}^{\vec{\beta}+\vec{r}}}}=
    \|T^{-1}e^{t(\mathcal{A}-\lambda I)}T u_{0}\|_{\textcolor{black}{\mathcal{H}^{\vec{\beta}+\vec{r}}}}\leq \|T^{-1}\|_{\mathcal{L}(\textcolor{black}{\mathcal{H}^{\vec{\beta}+\vec{r}}})}\|T\|_{\mathcal{L}(\textcolor{black}{\mathcal{H}^{\vec{\beta}+\vec{r}}})}Me^{-\mu t}\|u_{0}\|_{\textcolor{black}{\mathcal{H}^{\vec{\beta}+\vec{r}}}},\;\forall\;t\in[0,+\infty).
\end{equation}
This ends the proof of Theorem \ref{Theo_main}.
\end{proof}
\label{sec:step7}

\section{Conclusion and perspectives}
\label{sec:perspectives}
 In this paper, we extended the $F$-equivalence approach (\textcolor{black}{sometimes called} generalized backstepping), to a wide class of linear systems. This includes systems that are not necessarily skew-adjoint or self-adjoint provided that the differential operator involved has a Riesz basis of eigenvectors (hence not necessarily orthonormal) and is of order larger than one. We obtain conditions that are less restrictive than the ones proposed in \cite{Gagnon2022-fredholm,Gagnon2022-fredholm-laplacien} and we show that we can get nice rapid stabilization properties even when the system is not exactly controllable and the control operator is not admissible \textcolor{black}{in the space of stabilization}. 
We illustrate these results with the rapid stabilization of several examples such as the Schrödinger equation, a general diffusion equation, \textcolor{black}{Burgers equation}, and a system that is neither self nor skew adjoint. For Schrödinger equation and Burgers' equation, in particular, we obtain less restrictive conditions on the control operator than what was obtained in \cite{Coron2018-rapid,nguyen2024rapid,Gagnon2022-fredholm-laplacien}.
This result makes the $F$-equivalence one step closer to a general theory. Nevertheless, it is still a young approach and there are many questions that remain open and that would be interesting to consider
\begin{itemize}
    \item If $\mathcal{A}$ only has a generalized basis of eigenvectors (and not a Riesz basis of eigenvectors which we assume here), is it still possible to extend this method? This is interesting in particular for systems of PDE, where $\mathcal{A}$ does not always have a Riesz basis of eigenvectors.
    \item What happens if the eigenvalues $\lambda_{n}$ are infinite-dimensional? In particular, this could be useful for multidimensional systems. While Theorem \ref{Theo_main} and \ref{Theo_main_general} are not limited to 1D systems, in several multidimensional systems the eigenvalues $\lambda_{n}$ may have infinite multiplicity (think of the heat equation for instance).
       What happens if $\alpha=1$ ? In this case, only a few $F$-equivalence results exist and only in particular cases \cite{zhang2022internal,Coron2022-stabilization}, while there is a large literature on the controllability and stabilization of several such systems, for instance the wave equations \cite{jaffard1998singular,valein2009stabilization}.
   \item Can these linear results be extended to (very) nonlinear (i.e. quasi-linear) systems? This is all the more important that for infinite-dimensional systems the stability of the linearized system does not necessarily imply the stability of the nonlinear system, even locally \cite{coron2015dissipative} when they are at least quasi-linear, and there is so far no general method for stabilizing these systems.
    \item Is it possible to further relax the conditions on $B$? 
    \item Can the method be used for systems with a bilinear control (such as in \cite{alabau2021superexponential}) ?
    \item \textcolor{black}{What happens if we have only partial knowledge of the eigenvectors and eigenvalues and we can only construct in practice $\hat{K}$ that depends only on a finite-number of eigenvalues and $b_{n}$? If the system is parabolic, because high frequencies are stable, one can expect the system with control feedback law $\hat{K}$ to still be stable. When the system is not parabolic however, the question is more complicated.}
\end{itemize}

\section*{Acknowledgments}
The authors wish to thank the PEPS program JCJC 2022 of CNRS-INSMI as well as the Young Talent Program France-China (JTFC 2024) and the ANR StarPDE (ANR-24-ERCS-0010) for their support as well as the institute of mathematics of Peking University for their hospitality and particularly the NSF for China under grant 12301562.
\appendix
\section{Appendix}

\subsection{Proof of Lemma \ref{lem:structure}}
\label{app:Riesz}

We start by proving that $\|\cdot\|_{\mathcal{H}^{r}}$ is a norm on $\mathcal{H}^{r}$.  Let $f\in\mathcal{H}^{r}$, by definition there exists $(f_{n})_{n\in\mathbb{N}^{*}}\in l^{2}$ 
    such that $f = \sum\limits_{n\in\mathbb{N}^{*}}f_{n} \varphi_{n}$ 
    and therefore $\|f\|_{\mathcal{H}^{r}}$ is well-defined. 
    Besides, since $(\varphi_{n})_{n\in\mathbb{N}}$ is a Riesz basis, $f = 0$ is equivalent to $f_{n}=0$ for any $n\in\mathbb{N}^{*}$ and consequently to $\|f\|_{\mathcal{H}^{r}}=0$. Also for any $\mu\in\mathbb{C}$
    \begin{equation}
        \|\mu f\|_{\mathcal{H}^{r}} =\| \sum\limits_{n\in\mathbb{N}^{*}}\mu f_{n}\varphi_{n}\|_{\mathcal{H}^{r}} = 
|\mu|\left(\sum\limits_{n\in\mathbb{N}^{*}}n^{2r}|f_{n}|^{2}\right)^{2} = |\mu|\| f\|_{\mathcal{H}^{r}}.
    \end{equation}
    Finally, since $(\sum_{n}n^{2r}|f_{n}+g_{n}|^{2})^{1/2}\leq (\sum_{n}n^{2r}|f_{n}|^{2})^{1/2}+(\sum_{n}n^{2r}|g_{n}|^{2})^{1/2}$ for any $(g_{n})_{n\in\mathbb{N}^{*}}\in l^{2}$, $\|\cdot\|_{\mathcal{H}^{r}}$ is indeed a norm. One can easily check that it derives from the inner product given by \eqref{eq:inner}. 
    Since there is an isomorphism between $l^{2}(\mathbb{N}^{*})$ and $\mathcal{H}^{r}$ given by 
    \begin{equation}
       \iota: (f_{n})_{n\in\mathbb{N}^{*}}\mapsto \sum\limits_{n\in\mathbb{N}^{*}}f_{n}n^{-r}\varphi_{n},
    \end{equation}
    such that for any $(f_{n})_{n\in\mathbb{N}^{*}}$
        $\|\iota (f_{n})_{n\in\mathbb{N}^{*}}\|_{\mathcal{H}^{r}}=\|f_{n}\|_{l^{2}}$,
    then $\mathcal{H}^{r}$ is also complete and hence a Banach space. Finally, since $\|\cdot\|_{\mathcal{H}^{r}}$ is associated to the inner product $\langle \cdot , \cdot\rangle_{\mathcal{H}^{r}}$ it is a Hilbert space.

\subsection{$(q_n)_{n\in \mathbb{N}^*}$ is either ${\mathcal{H}^0}$-dense or $\omega$ independent in ${\mathcal{H}^0}$}\label{sec-appendix-q}
The following argument is classical. It was first provided in \cite{Coron2018-rapid} for the Schrödinger equation, and recently in \cite{Gagnon2022-fredholm-laplacien} and \cite{Gagnon2022-fredholm} for the heat equation and skew-adjoint operators respectively. We detail it here for completeness. We assume that $\lambda_n\neq 0.$ The proof can be easily adapted to the case where $\lambda_1=0,$ since the resolvent $\mathcal{A}_{\lambda}$ defined below is well defined and invertible on ${\mathcal{H}^0}.$

Recall that $q_n\in \mathcal{H}^{\alpha-1/2-\varepsilon}$ for any $\varepsilon>0.$ We define $r_n:=(\lambda_n+\lambda/2)^{-1}$ and we obtain by definition of $q_n$ in \eqref{q_n} that
\begin{equation*}
    (\mathcal{A}-\lambda-\lambda_n)q_n=-\sum_{p\in \mathbb{N}^*}\varphi_p=h,\quad \text{in}\ \mathcal{H}^{\alpha-1/2-\varepsilon}
\end{equation*} which becomes by assuming without loss of generality that $\lambda/2$ belongs to the resolvent set of $\mathcal{A},$
\begin{align*}
    \mathcal{A}_{\lambda}q_n=r_n q_n-r_n\mathcal{A}_{\lambda}h\quad \text{in}\ \mathcal{H}^{\alpha-1/2-\varepsilon}
\end{align*} where
\begin{align*}
    \mathcal{A}_{\lambda}:=(\mathcal{A}+\lambda/2)^{-1}.
\end{align*}
Let us now assume that the family $(q_n)_{n\in \mathbb{N}^*}$ is not $\omega$ independent in ${\mathcal{H}^0}$. Then, there exists non trivial $(c_n)_{n\in \mathbb{N}^*}\in l^2$  such that 
\begin{align*}
    \sum_{n\in \mathbb{N}^*}c_n q_n=0\quad \text{in}\ {\mathcal{H}^0}.
\end{align*}
By applying $\mathcal{A}_{\lambda}$ to this equation, we get
\begin{align*}
    \sum_{n\in \mathbb{N}^*}c_n r_n q_n=\left(\sum_{n\in \mathbb{N}^*}c_n r_n \right)\mathcal{A}_{\lambda}h\quad \text{in}\ {\mathcal{H}^0}
\end{align*} which is well defined since $(r_n)_n\in l^2.$
Applying again $\mathcal{A}_{\lambda},$ we get
\begin{align*}
    \sum_{n\in \mathbb{N}^*}c_n r_n^2 q_n=\left(\sum_{n\in \mathbb{N}^*}c_n r_n^2 \right)\mathcal{A}_{\lambda}h+\left(\sum_{n\in \mathbb{N}^*}c_n r_n \right)\mathcal{A}_{\lambda}^2h\quad \text{in}\ {\mathcal{H}^0}
\end{align*}
By induction, it holds that
\begin{align}
    \sum_{n\in \mathbb{N}^*}c_nr_n^m q_n=\sum_{i=1}^{m}\left(\sum_{n\in \mathbb{N}^*}c_nr_n^{m+1-i}\right)\mathcal{A}_{\lambda}^i h=\sum_{i=1}^{m}C_{m+1-i} \mathcal{A}_{\lambda}^i h\quad \text{in}\ {\mathcal{H}^0},\label{S-induc}
\end{align} where 
\begin{align*}
    C_j:=\sum_{n\in \mathbb{N}^*}c_nr_n^j<+\infty, \quad j\in \mathbb{N}^*.
\end{align*}
Let us now distinguish two cases:
\begin{itemize}
    \item The $\{C_m\}$ are identically zero.
    Then we define the complex variable function 
    \begin{align*}
        G(z):=\sum_{n\in \mathbb{N}^*}c_nr_ne^{r_n z}.
    \end{align*} This function is holomorphic and since $\{C_m\}$ are identically zero, we have
    \begin{align*}
        G^{(m)}(0)=0, \quad \forall m\in \mathbb{N}. 
    \end{align*} 
    Thus $G\equiv 0,$ and therefore $c_n=0$ for all $n\in \mathbb{N}^*.$ This is in contradiction with the definition of $(c_n)_{n\in \mathbb{N}^*}.$ Hence, $(q_n)_{n\in \mathbb{N}^*}$ is $\omega$ independent in $\mathcal{H}^0.$
    \item The $\{C_m\}$ are not identically zero. Let us denote $m_0:=\inf\{n\in \mathbb{N}^*, C_n\neq 0\}.$ Then, starting with $m=m_0,$ we have based on \eqref{S-induc} by induction 
    \begin{align}
        \mathcal{A}_{\lambda}^m\in \text{span}\{q_n\}_{n\in \mathbf{N}^*},\quad \forall m\in \mathbb{N}^*.
    \end{align}
    Suppose that the family $(q_n)_{n\in \mathbb{N}^*}$ is not 
    ${\mathcal{H}^0}$-dense. Then there exists a nonzero $d=\displaystyle\sum_{n\in \mathbb{N}^*}d_n\varphi_n\in {\mathcal{H}^0}$ such that 
    \begin{align*}
        \langle g,d\rangle_{\mathcal{H}^0}=0, \quad \forall g\in \text{span}\{q_n\}_{n\in \mathbb{N}^*}
    \end{align*} which in particular yields,
    \begin{align*}
        \langle \mathcal{A}_{\lambda}^m h,d\rangle_{\mathcal{H}^0}=0, \quad \forall m\in \mathbb{N}^*.
    \end{align*} 
    Noticing that $h=-\displaystyle\sum_{n\in \mathbb{N}^*}\varphi_n\in \mathcal{H}^{-1},$ we get that 
    \begin{align}
        \sum_{n\in \mathbb{N}^*}\bar d_n r_n^m=0,\quad \forall m\in \mathbb{N}^*.\label{S1}
    \end{align}
    We define the complex variable function,
    \begin{align*}
        \tilde G(z):=\sum_{n\in \mathbb{N}^*}\bar d_nr_ne^{r_n z}, \quad \forall z\in \mathbb{C}.
    \end{align*}
    This function is holomorphic and from \eqref{S1}, $\tilde G^{(m)}(0)=0$ for all $m\in \mathbb{N}^*.$ Thus $\tilde G\equiv 0,$ and further $d_n=0$ for any $n\in \mathbb{N}^*$ which is a contradiction. Therefore the family $(q_n)_{n\in \mathbb{N}^*}$ is either ${\mathcal{H}^0}$-dense or $\omega$ independent.
\end{itemize}

{\color{black}
\subsection{Extended regularity of $T$}
\label{app:regT}

 From Lemma \ref{Lemma-T-bounded} and the first part of Remark \ref{rmk:extension}, we can deduce the following:
\begin{lem}
    \label{Lemma-T-bounded-bis}
   For any $r\in (1/2-\alpha,\alpha-1/2),$ $T$ is a bounded operator from $\mathcal{H}^r$ to itself.  
\end{lem}
\begin{proof}
From Lemma \ref{Lemma-T-bounded}, the only cases that need to be dealt with are $r\in(1/2-\alpha,1/2-\alpha+\gamma]$. Let $r\in (1/2-\alpha,-1/2-\gamma)$, and set $f\in \mathcal{H}^{r}$. Since the eigenvalues are simple (recall that we are working in the whole proof on one of the subspaces of the decomposition \eqref{eq:sumHi}) there exists at most one $n_{0}\in\mathbb{N}^{*}$ such that $\lambda_{n_{0}}=0$. Assume that this is the case (the proof is the same otherwise, with $\langle f, \varphi_{n_{0}}\rangle_{\mathcal{H}^{0}}=0$ instead), 
\begin{equation}
    f = \langle f, \varphi_{n_{0}}\rangle_{\mathcal{H}^{0}}\varphi_{n_{0}}+f^{1},
\end{equation}
where $f_{1}\in \{f\in\mathcal{H}^{r}\;|\;\langle f,\varphi_{n_{0}}\rangle =0\}$. Note that one can define $\mathcal{A}^{-1}$ on this space and $g:=\mathcal{A}^{-1}f^{1}\in \mathcal{H}^{r+\alpha}$. Since $r+\alpha \in(1/2,\alpha-1/2-\gamma)$, one can apply \eqref{opeq-space*} to $g$ and have, using the definition of $g$,
\begin{equation}
Tf^{1} = -BK\mathcal{A}^{-1}f^{1} +(\mathcal{A}-\lambda I)T\mathcal{A}^{-1}f^{1},
\end{equation}
and this equality holds in $\mathcal{H}^{r}$. In particular
\begin{equation}
\begin{split}
    \|Tf^{1}\|_{\mathcal{H}^{r}}\leq &\left(\|B\|_{\mathcal{H}^{r}}\|K\|_{\mathcal{L}(\mathcal{H}^{r+\alpha},\mathbb{R})}\|\mathcal{A}^{-1}\|_{\mathcal{L}(\mathcal{H}^{r},\mathcal{H}^{r+\alpha})}\right.\\
    &\left.+\|(\mathcal{A}-\lambda I)\|_{\mathcal{L}(\mathcal{H}^{r+\alpha},\mathcal{H}^{r})}\|T\|_{\mathcal{H}^{r+\alpha}}\|\mathcal{A}^{-1}\|_{\mathcal{L}(\mathcal{H}^{r},\mathcal{H}^{r+\alpha})}
    \right)\|f^{1}\|_{\mathcal{H}^{r}}.
    \end{split}
\end{equation}
where we used Lemma \ref{Lemma-T-bounded} the fact that $T$ is bounded on $\mathcal{H}^{r+\alpha}$ since $r+\alpha\in(1/2,\alpha-1/2-\gamma)$, as well as the fact that $r< -1/2-\gamma$ and therefore $B\in\mathcal{H}^{r}$. Since $X$ is continuously embedded in $\mathcal{H}^{r}$ for $r<-1/2-\gamma$, $\|T \varphi_{n_{0}}\|_{\mathcal{H}^{r}}\leq \|T \varphi_{n_{0}}\|_{X}$, hence, there exists a constant $C$ (possibly depending on $r$) such that
\begin{equation}
    \|T f\|_{\mathcal{H}^{r}}\leq C\left(n_{0}^{r}|\langle f, \varphi_{n_{0}}\rangle_{\mathcal{H}^{0}}|+\|f^{1}\|_{\mathcal{H}^{r}}\right)\leq C\|f\|_{\mathcal{H}^{r}}.
\end{equation}
Thus $T$ is bounded from $\mathcal{H}^{r}$ into itself for any $r\in(1/2-\alpha,-1/2-\gamma)$ and since $-1/2-\gamma >1/2-\alpha-\gamma$ since $\gamma\in[0,(\alpha-1)/2)$, this ends the proof of the Lemma.
\end{proof}
}

\subsection{Proof of Lemma \ref{ker T}}\label{sec-Appendix-kerT}
We proceed as presented in Section \ref{sec:step6}. 
{\color{black} Before doing that let us first show that 
\begin{lem}
\label{lem:Tfredholm}
    For any $r\in (1/2-\alpha+\gamma, \alpha-1/2),$ the operator $T$ is a Fredholm operator of index $0$ from $\mathcal{H}^{r}$ to itself. 
\end{lem} 
\begin{proof}
   Looking at the definition of $T$ given by \eqref{eq:decompT2} and \eqref{K_n-expression}, we have for any $n\in\mathbb{N}^*$
 \begin{align}
 T\varphi_{n} &= \frac{\lambda+k_{n}}{b_{n}}\sum\limits_{p\in\mathbb{N}^{*}}\frac{b_p \varphi_{p}}{\lambda_{n}-\lambda_{p}+\lambda}\nonumber\\
 &= \lambda\tilde \tau \varphi_{n}+k\circ\tilde\tau \varphi_n,\label{T-fred}
 \end{align}
 where $\tilde\tau$ is given by Lemma \ref{lem-tilde-tau} and $k$ is defined on $\mathcal{H}^r$ by $k:n^{-r}\tilde \tau \varphi_n\mapsto k_n n^{-r}\tilde \tau \varphi_n$. Notice that the operator $k$ is a compact operator. Indeed, let $f\in \mathcal{H}^r.$ Since, $\tilde \tau$ is an isomorphism from $\mathcal{H}^r$ to itself, it holds that $(n^{-r}\tilde \tau \varphi_n)_{n\in \mathbb{N}^*}$ is a Riesz basis of $\mathcal{H}^r$ and in particular there exists $(f_n)_{n\in \mathbb{N}^*}\in l^2$ such that $f=\displaystyle\sum_{n\in \mathbb{N}^*}f_n n^{-r}\tilde \tau \varphi_n.$ Let $\varepsilon>0$ to be chosen, 
 \begin{align*}
     k(f)=\sum_{n\in \mathbb{N}^*}k_n f_n n^{-r}\tilde \tau \varphi_n=\sum_{n\in \mathbb{N}^*}k_n n^{\varepsilon} f_n n^{-r-\varepsilon}\tilde \tau \varphi_n.
 \end{align*}
 Since from Lemma \ref{Lemma-K_n}, there exists $\varepsilon>0$ such that $(k_n n^{\varepsilon})_{n\in \mathbb{N}^*}\in l^{\infty}$ , we have that $(f_n k_n n^{\varepsilon})_{n\in \mathbb{N}^*}\in l^{2}.$ And this ensures that $k(f)\in \mathcal{H}^{r+\varepsilon}$ and, in particular, 
 \begin{align*}
     \|k(f)\|_{\mathcal{H}^{r+\varepsilon}}=\left\|\sum_{n\in \mathbb{N}^*}k_n n^{\varepsilon} f_n n^{-r-\varepsilon}\tilde \tau \varphi_n\right\|_{\mathcal{H}^{r+\varepsilon}}\lesssim \|k_n n^{\varepsilon} f_n\|_{l^2}\lesssim \|f_n\|_{l^2}\lesssim\|f\|_{\mathcal{H}^{r}}.
 \end{align*}
 Thus, $k$ is compact in $\mathcal{H}^r$ from Lemma \ref{lem:embedding}. Combining this with the fact that $\tilde \tau$ is a Fredholm operator of index $0$ (Lemma \ref{lem-tilde-tau}) and thanks to the expression \eqref{T-fred}, $T$ is a Fredholm operator of index $0$ in $\mathcal{H}^r$ thanks to  \cite[p. 169]{Brezis2011-functional}. 
\end{proof}
}

\begin{itemize}
    \item Let us define $z:=\lambda+\rho,$ where $\rho\in \mathbb{C}$. We assume without loss of generality that the operator $\mathcal{A}$ is invertible. Otherwise, since its spectrum is countable, there exists sufficiently small $\delta\neq 0$ such that $\mathcal{A}+\delta I$ is invertible. So showing that $\mathcal{A}+BK+zI$ is invertible amounts to show that $I+\mathcal{A}^{-1}BK+z\mathcal{A}^{-1}$ is invertible. We now consider the following two distinct cases:
    \begin{enumerate}
        \item [(a)] $K(\mathcal{A}^{-1}B)\neq -1.$\\
        We can check that for any $f\in\mathcal{H}^{\alpha/2},$ the function $\varphi\in \mathcal{H}^{\alpha/2}$ defined by
        \begin{align*}
            \varphi:=f-\frac{\mathcal{A}^{-1}B(Kf)}{1+K(\mathcal{A}^{-1}B)}
        \end{align*} is the unique solution to 
        \begin{align*}
            (I+\mathcal{A}^{-1}BK)\varphi=f.
        \end{align*} 
        Thus, the operator $I+\mathcal{A}^{-1}BK$ is invertible. Since $\mathcal{A}$ is a differential operator, then $\mathcal{A}^{-1}$ is a continuous operator. And then, thanks to the openness of invertible operators, there exists $\varepsilon>0$ such that for any $|z|<\varepsilon$
        \begin{align}
        I+\mathcal{A}^{-1}BK+z\mathcal{A}^{-1}
        \end{align} is invertible in $\mathcal{H}^{-\alpha/2}.$
        \item[(b)] $K(\mathcal{A}^{-1}B)= -1.$\\
        For any $v\in \mathcal{H}^{\alpha/2},$ it holds that
        \begin{align*}
            (I+\mathcal{A}^{-1}BK)v=0&\Leftrightarrow v=-\mathcal{A}^{-1}B(Kv)\in \text{span}\{\mathcal{A}^{-1}B\}.
        \end{align*} This ensures that $0$ is an eigenvalue of $I+\mathcal{A}^{-1}BK$ with multiplicity $1$ and the corresponding eigenspace is generated by $\mathcal{A}^{-1}B.$ Thus, there exist small open neighborhoods $\Omega$ and $\tilde \Omega$ of $0$ in $\mathbb{C}$ satisfying:
        \begin{align}
            &(I+\mathcal{A}^{-1}BK+z\mathcal{A}^{-1})y(z)=\lambda(z)y(z)\label{I-appendix}\\
            &y:z\mapsto y(z)\in \mathcal{H}^{\alpha/2},\ \text{is holomorphic}\\
            &\lambda:z\mapsto \lambda(z)\in \tilde\Omega,\ \text{is holomorphic}\\
            &\lambda(0)=0, \quad y_0:=y(0)=\mathcal{A}^{-1}B        
            \end{align} in such fashion that for any $z\in \Omega,$ $\lambda(z)$ is the unique eigenvalue inside $\tilde \Omega.$ Since $\lambda(0)=0,$ either $\lambda$ is identically $0$ in $\Omega$ or there exists small neighborhood $\omega$ such that $\lambda(z)\neq 0$ for any $z\in \omega\smallsetminus\{0\}.$ Let us show that $\lambda$ is not identically $0.$ For that, we assume by contradiction that it is. From the holomorphy property of $y,$ there exists a sequence $(y_k)_{k\in \mathbb{N}^*}\in \mathcal{H}^{\alpha/2}$ such that
            \begin{align*}
                y(z)=\sum_{k\in \mathbb{N}^*}y_k z^k,
            \end{align*} with $y_0=\mathcal{A}^{-1}B.$
    Then, in view of \eqref{I-appendix} and from assumption that $\lambda\equiv0,$ it holds that
    \begin{align*}
       (I+\mathcal{A}^{-1}BK+z\mathcal{A}^{-1}) \sum_{k\in \mathbb{N}^*}y_k z^k=0, \ \text{in}\ \mathcal{H}^{\alpha/2}
    \end{align*} which ensures by the unicity of the development in entire function that
    \begin{align}
        (I+y_0K)y_k+\mathcal{A}^{-1}y_{k-1}=0,\ \text{in}\ \mathcal{H}^{\alpha/2}, \ \forall k\in \mathbb{N}^*.\label{I-induc}
    \end{align}
    By applying $K$ to this equation and noticing that $Ky_0=-1,$ we get
    \begin{align*}
        K(\mathcal{A}^{-1}y_{k-1})=0, \quad \forall k\in \mathbb{N}^*\Leftrightarrow K(\mathcal{A}^{-1}y_{k})=0, \quad \forall k\in \mathbb{N}
    \end{align*}
    Applying now $K\mathcal{A}^{-1}$ to the same equation \eqref{I-induc}, we get
    \begin{align*}
        K(\mathcal{A}^{-1}y_k)+K(\mathcal{A}^{-1}y_0)(Ky_k)+K(\mathcal{A}^{-2}y_{k-1})=0, \quad \forall k\in \mathbb{N}^*
    \end{align*} and then
    \begin{align*}
        K(\mathcal{A}^{-2}y_{k-1})=0, \quad \forall k\in \mathbb{N}^*\Leftrightarrow K(\mathcal{A}^{-2}y_{k})=0, \quad \forall k\in \mathbb{N}
    \end{align*}
    By induction, applying $K\mathcal{A}^{-(n-1)}$ to \eqref{I-induc}, we get that
    \begin{align*}
        K(\mathcal{A}^{-n}y_k)=0,\quad \forall k\in \mathbb{N}, n\in\mathbb{N}^*.
    \end{align*}
    In particular
    \begin{align*}
        K(\mathcal{A}^{-n}y_0)=0,\quad \forall n\in\mathbb{N}^*,
    \end{align*} which implies that 
    \begin{align*}
        \sum_{m\in\mathbb{N}^*}\frac{b_m K_m}{\lambda_m^l}=0, \quad \forall l\geq 2.
    \end{align*}
    Using the holomorphic function technique in Appendix \ref{sec-appendix-q}, we get that $b_mK_m=0,$ which is a contradiction. Therefore, there exists $\varepsilon>0$ such that $\lambda(z)\neq 0$ for any $z\in \Omega$ and $|z|<\varepsilon.$ Since $\lambda(z)$ is the unique eigenvalue inside $\tilde \Omega,$ $I+\mathcal{A}^{-1}BK+z\mathcal{A}^{-1}$ is invertible.
    \end{enumerate}
    In the two cases, there exists at least a sequence of $(z_k)_{k\in \mathbb{N}}$ converging to $0$ such that $I+\mathcal{A}^{-1}BK+z_k\mathcal{A}^{-1}$ is invertible from $\mathcal{H}^{\alpha/2}$ to $\mathcal{H}^{-\alpha/2}.$  Since the spectrum of $\mathcal{A}+\rho I$ is discrete, we can find $\rho:=z_k-\lambda$ such that both $\mathcal{A}+\rho I$ and $\mathcal{A}+BK+\lambda I+\rho I=\mathcal{A}(I+\mathcal{A}^{-1}BK+z_k \mathcal{A}^{-1})$ are invertible operators from $\mathcal{H}^{\alpha/2}$ to $\mathcal{H}^{-\alpha/2}$.
    \item We assume here that $\ker T^*\neq \{0\}$ and we want to show that $(\mathcal{A}^*+\bar \rho I)^{-1}$ has an eigenvector $h$ in $\ker T^*$ with $\rho$ defined above. Recall the operator equality
    \begin{align*}
        T(\mathcal{A}+BK)=(\mathcal{A}-\lambda I)T.
    \end{align*} Based on Lemma \ref{Lemma-T-bounded}, this equality holds at least where the operator are seen as acting  on $\mathcal{H}^{\alpha/2}$ to $\mathcal{H}^{-\alpha/2}.$ Then, for any $\rho$ as in above, the following holds:
    \begin{align*}
        T(\mathcal{A}+BK+\lambda I+\rho I)=(\mathcal{A}+\rho I)T.
    \end{align*} Since, $\mathcal{A}+BK+\lambda I+\rho I$ and $\mathcal{A}+\rho I$ are invertible, we get
    \begin{align}
        (\mathcal{A}+\rho I)^{-1}T=T(\mathcal{A}+BK+\lambda I+\rho I)^{-1}.\label{A+r}
    \end{align}
    As $\ker T^*\neq \{0\},$ we can select $h\neq 0$ such that $h\in \ker T^*$ and $h\in \mathcal{H}^{-\alpha/2}.$ So, from \eqref{A+r}, for any $\varphi\in \mathcal{H}^{-\alpha/2},$ it holds that
    \begin{align*}
        0&=\langle (\mathcal{A}+\rho I)^{-1}T\varphi-T(\mathcal{A}+BK+\lambda I+\rho I)^{-1}\varphi,h\rangle_{\mathcal{H}^{-\alpha/2}}\\
        &=\langle \varphi, T^*(\mathcal{A}^*+\bar \rho I)^{-1}h\rangle_{\mathcal{H}^{-\alpha/2}}-\langle (\mathcal{A}+BK+\lambda I+\rho I)^{-1}\varphi,T^*h\rangle_{\mathcal{H}^{-\alpha/2}} \\
        &=\langle \varphi, T^*(\mathcal{A}^*+\bar \rho I)^{-1}h\rangle_{\mathcal{H}^{-\alpha/2}}.
    \end{align*} This ensures that $T^*(\mathcal{A}^*+\bar \rho I)^{-1}h=0$ in $\mathcal{H}^{-\alpha/2},$ and thus $(\mathcal{A}^*+\bar \rho I)^{-1}h\in \ker T^*.$ We have then shown that for any $f\in \ker T^*,$ $(\mathcal{A}^*+\bar \rho I)^{-1}f\in \ker T^*.$\\
    Because $\ker T^*$ is of finite dimension (recall that $T$ is Fredholm, hence $T^*$ is) and not reduced to $\{0\}$ therefore the restriction of $(\mathcal{A}^*+\bar \rho I)^{-1}$ to $\ker T^*$ belongs to $\mathcal{L}(\ker T^*)$ and is an operator on a space of finite-dimension, therefore   
    there exists an eigenfunction $h\in \ker T^*$ of $(\mathcal{A}^*+\bar\rho I)^{-1}$, associated to an eigenvalue $\mu\neq 0$ (since the operator $(\mathcal{A}+\bar \rho I)^{-1}$ is invertible). Then we have
    \begin{align*}
        (\mathcal{A}+\bar \rho I)^{-1}h=\mu h,
    \end{align*} which in particular implies that
    \begin{align}
        \mathcal{A}^*h=\frac{1-\bar \rho \mu}{\mu}h,\label{Eig}
    \end{align} and $h\in \mathcal{H}^{\alpha/2}.$ This ensures that $h$ is an eigenfunction of $\mathcal{A}^*$ associated to the eigenvalue $(1-\bar \rho \mu)/\mu.$
    As the subspaces of $\mathcal{H}^{-\alpha/2}$ which are the  eigenspaces of $\mathcal{A}^*$  have dimension $1,$ the dimension of eigenspace associated to $(1-\bar \rho \mu)/\mu$ is one. Since $(n^{\alpha/2}\varphi_n)_{n\in \mathbb{N}^*}$ forms a Riesz basis of $\mathcal{H}^{-\alpha/2},$ there exists $n_0\in\mathbb{N}^*$ such that
    $\text{Eig}((1-\bar \rho \mu)/\mu)=\text{span}
    \{n_0^{\alpha/2}\varphi_{n_0}\}.$ In view of \eqref{Eig}, $h\in \text{Eig}((1-\bar \rho \mu)/\mu)$, then there exists $C>0$ such that
    $h=Cn_0^{\alpha/2}\varphi_{n_0}.$ So, we have finally shown that if $\ker T^*\neq \{0\},$ there exist $n_0\in \mathbb{N}^*, C>0$ such that $Cn_0^{\alpha/2} \varphi_{n_0}\in \ker T^*$ and $Cn_0^{\alpha/2} \varphi_{n_0}$ is an eigenfunction of $(\mathcal{A}^*+\bar\rho I).$
    \item From what is done above, if $\ker T^*\neq \{0\},$ there exist $n_0\in \mathbb{N}^*, C>0$ such that $Cn_0^{\alpha/2} \varphi_{n_0}\in \ker T^*.$ So for any $\varphi\in \mathcal{H}^{-\alpha/2},$
    \begin{align*}
        0=\langle T\varphi,n_0^{\alpha/2} \varphi_{n_0}\rangle_{\mathcal{H}^{-\alpha/2}}=\langle \varphi,T^*(n_0^{\alpha/2} \varphi_{n_0})\rangle_{\mathcal{H}^{-\alpha/2}}
    \end{align*}
    In particular for $\varphi=B=\displaystyle\sum_{n\in \mathbb{N}^*}b_n \varphi_n$ we have
    \begin{align*}
        0=\langle TB,n_0^{\alpha/2} \varphi_{n_0}\rangle_{\mathcal{H}^{-\alpha/2}}
    \end{align*} which combining with the fact that $TB=B$ in $\mathcal{H}^{-\alpha/2},$ ensures that
    \begin{align*}
        0=\langle B ,n_0^{\alpha/2} \varphi_{n_0}\rangle_{\mathcal{H}^{-\alpha/2}}=\frac{b_{n_0}}{n_0^{\alpha/2}}.
    \end{align*}
    This is a contradiction due to assumption \eqref{control_cond}. Hence, $\ker T^*=\{0\}.$
\end{itemize}

\subsection{$\mathcal{A}$ generates a semigroup on $\mathcal{H}^{\vec r}$}
\label{app:semigroup}
We show the following
\begin{lem}
\label{lem:semigroup}
Under the assumption of Theorem \ref{Theo_main_general}, $\mathcal{A}$ generates a $C^{0}$ semigroup on $\mathcal{H}^{\vec r}$ for any $r=(r_{1},...,r_{m})\in \mathbb{R}^{m}$. 
\end{lem}
This is due to the intrinsic link between the definition of the spaces $\mathcal{H}^{\vec r}$ and $\mathcal{A}$ (see Section \ref{sec:notations}). We explicit this below. Given the definition of $\mathcal{H}^{\vec r}$ it suffices to show that for any $i\in\{1,...,m\}$ and $r\in\mathbb{R}$, $\mathcal{A}$ generates a $C^{0}$ semigroup on $\mathcal{H}^{r}_{i}$.
We define for any $t\geq0$
\begin{equation}
\begin{split}
     \mathcal{H}^{r}_{i}&\rightarrow \mathcal{H}^{r}_{i}\\
    \mathcal{S}_{i,r}(t) \;:\; \sum\limits_{n\in\mathbb{N}^{*}}f_{n}\varphi_{n}^{i} &\mapsto\sum\limits_{n\in\mathbb{N}^{*}}e^{\lambda_{n}^{i}t}f_{n}\varphi_{n}^{i}.
    \end{split}
\end{equation}
Since $\mathcal{A}$ generates a $C^{0}$ semigroup with finite growth bound, there exists $\omega>0$ such that $|e^{\lambda_{n}^{i} t}|\leq e^{\omega t}$ for any $n\in\mathbb{N}^{*}$ and therefore $S_{i,r}(t)\in\mathcal{L}(\mathcal{H}^{r})$ for any $t\geq 0$. Clearly, $S_{i,r}(0) = Id$ and for $(t,s)\in\mathbb{R}_{+}$, 
    $S_{i,r}(t+s) = S_{i,r}(t)S_{i,r}(s)$. 
    Finally, using the fact that $\tau_{r}: \varphi_{n}^{i}\rightarrow n^{-r}\varphi_{n}^{i}$ is an isomorphism from $\mathcal{H}^{0}_{i}$ to $\mathcal{H}^{r}_{i}$ which commutes with $\mathcal{A}$ and satisfies
    \begin{equation}
        S_{i,r}\tau_{r} = \tau_{r} S_{i,0},
    \end{equation}
    and the fact that $\mathcal{A}$ is the infinitesimal generator of $S_{i,0}$ on $\mathcal{H}_{i}^{0}$, we deduce directly that $S_{i,r}$ is a $C^{0}$ semigroup and that $\mathcal{A}$ is the infinitesimal generator of $S_{i,r}$ with domain $D_{r}(\mathcal{A}):=\tau_{r}(D_{0}(\mathcal{A})) = \mathcal{H}^{r+\alpha_{i}}_{i}$.

\subsection{Extension to the case $\beta\neq 0$}
\label{app:beta}
Suppose that Proposition \ref{prop:main} holds for $\beta=0$, then it also holds for $\beta\neq0$ exactly as in \cite{Gagnon2022-fredholm}. To do so, it suffices to define $\tilde B= \sum_n n^{\beta} b_n \varphi_n$, apply Proposition \ref{prop:main} with $\tilde{B}$ to obtain a feedback operator $\tilde{K}$ and an isomorphism $\tilde{T}$. Then Proposition \ref{prop:main} holds for $\beta\neq 0$ with $T:= M^{-1} \tilde T M$, and $K = \tilde K M$ where $M$ is the isomorphism from $\mathcal{H}^{\beta+ s}$ to $\mathcal{H}^{s}$ for any $s\in\mathbb{R}$ defined by
\begin{equation}
    M : n^{-\beta} \varphi_n\mapsto \varphi_n.
\end{equation}
We refer to \cite{Gagnon2022-fredholm} for more details.

\subsection{Proof of Corollary \ref{corollary-diffusion}}
\begin{proof}
    The corresponding operator $\mathcal{A}$ for the system \eqref{eq:sysdiff1} is the following Sturm-Liouville operator
    \begin{align*}
        \mathcal{A}:=\partial_x(a\partial_x)+b,
    \end{align*} 
    defined on the domain
    \begin{equation}
     D(\mathcal{A}) =\{f\in H^{2}([0,L])\;|\; 
         c_{1}f(0) = -c_{2}f'(0),\;\;
    c_{3}f(L) = -c_{4}f'(L)
     \}.
    \end{equation}

    This is classically a self-adjoint operator and the family $(u_n)_{n\in\mathbb{N}^*}$ of its eigenvectors forms an orthonormal basis in $L^2(0,L)$ and its eigenvalues $(\lambda_n)_{n\in \mathbb{N}^*}$ are simple, real, discrete and forms a non-decreasing sequence 
    \cite{Fulton1994-eigenvalue}.

    We choose $\textcolor{black}{X}=L^{2}(0,L)$. Since the eigenvalues are simple $\textcolor{black}{X}_{1}=\textcolor{black}{X}:=L^{2}$.

     Let us now determine the eigenvalues of the operator $\mathcal{A}$ and show that they verify the assumptions \eqref{ln+1_cond} and \eqref{ln-lp_cond}. For any eigenpair $(\lambda_n,u_n)$ of $\mathcal{A},$ it holds that
     \begin{align*}
         \mathcal{A}u_n=\lambda_n u_n
    &\Leftrightarrow\partial_x(a\partial_x u_n)+bu_n=\lambda_n u_n\\
      &\Leftrightarrow\partial_x(a\partial_x u_n)+(b-\lambda_n)u_n=0.
     \end{align*}
     By the following change of variables
     \begin{align}
     \label{change-variables}
         y(x):=\int_0^x \frac{1}{\sqrt{a(s)}}ds, \quad \varphi_n(y):=a(x)^{1/4}u_n(x),\quad \forall x\in [0,L],
     \end{align} we get that
     \begin{align*}
         \partial_x(a(x)\partial_x u_n(x))+(b(x)-\lambda_n)u_n(x)=0,\quad \forall x\in [0,L]\quad \Leftrightarrow \quad &\partial_y^2 \varphi_n(y)+Q(y)\varphi_n(y)=\lambda_n \varphi_n(y),\\
         &\forall y\in [0,M],
     \end{align*} where
     \begin{align}
         Q(y):=b(x(y))-\frac{\partial_y^2 [a(x(y))^{1/4}]}{a(x(y))^{1/4}}, \quad M:=\int_0^L \frac{1}{\sqrt{a(s)}}ds.\label{Q-M}
     \end{align} With the change of variables \eqref{change-variables} and \eqref{eq:bcsturm}, we have the following boundary conditions on $\varphi_n$:
     \begin{align}
         \begin{split}
            \tilde c_1\varphi_n(0) + \tilde c_2\partial_{y}\varphi_n(0)=0\\
    \tilde c_3\varphi_n(M) + \tilde c_4\partial_{y}\varphi_n(M)=0,
         \end{split}
     \end{align} where
     \begin{align*}
         \tilde c_1:=c_{1}a(0)^{-1/4}-\frac{c_{2}a'(0)}{4a(0)^{5/4}}, \ \tilde{c}_2:=\frac{c_2}{a(0)^{3/4}},\ \tilde c_3:=c_{3}a(L)^{-1/4}-\frac{c_{4}a'(L)}{4a(L)^{5/4}}, \ \tilde{c}_4:=\frac{c_4}{a(L)^{3/4}}.
     \end{align*}
     Note that since $a$ has only positive values, $\tilde{c}_{1}^{2}+\tilde{c}_{2}^{2}>0$ and $\tilde{c}_{3}^{2}+\tilde{c}_{4}^{2}>0$.
     Let us define the following operator
      \begin{align*}
          \tilde {\mathcal{A}}:=\partial_y^2+Q
      \end{align*} on the domain
      \begin{align*}
          D(\tilde{\mathcal{A}}):=\{f\in H^2([0,M]) \;|\; \tilde c_1f(0)+\tilde c_2 \partial_y f(0)=0, \;\;  \tilde c_3f(M)+\tilde c_4 \partial_y f(M)=0\}.
      \end{align*} The eigenvalues $\lambda_n$ are invariant under the change of variables \eqref{change-variables}. In \textcolor{black}{other} words,  any eigenvalue $\lambda_n$ of $\mathcal{A}$ is also an eigenvalue of the operator $\tilde{\mathcal{A}}$ and conversely, only 
     the eigenfunctions change. Thus, it suffices to know the eigenvalues of $\tilde{\mathcal{A}}$ to know the ones of $\mathcal{A}.$ From \cite{Fulton1994-eigenvalue}, the eigenvalues satisfy the following asymptotic estimate
     \begin{align}
         \lambda_n= -cn^2+O(1)\quad \forall n\in \mathbb{N}^*,\label{ln-expression}
     \end{align} where $c$ is a positive constant independent of $n.$ Thus, the assumption \eqref{ln+1_cond} immediately holds for $\alpha_i=\alpha=2$. 
     
     Notice that the assumption \eqref{ln-lp_cond} holds for any $n=p\in \mathbb{N}^*.$ 
     
     Then, we consider now, $n\neq p\in \mathbb{N}^*,$ and we have
     \begin{align*}
         |\lambda_n-\lambda_p|&=|cn^2-cp^2+O(1)|.
     \end{align*} 
Since
     \begin{equation}         \lim\limits_{p^{2}+n^{2}\rightarrow+\infty, p\neq n} |n^{2}-p^{2}| \geq \lim\limits_{p^{2}+n^{2}\rightarrow+\infty} |p+n| =   +\infty,
     \end{equation} 
     there exists $\eta>0$ such that
     \begin{equation}
         |\lambda_n-\lambda_p|\geq \eta|n^{2}-p^{2}|, \forall n,p\geq 1.
     \end{equation}
     
     This ensures that
     \begin{align*}
         \forall n,p\geq 1, \quad |\lambda_n-\lambda_p|&\geq \eta|(n+p)(n-p)|\\
         &\geq \eta n|n-p|,
     \end{align*}
      which makes \eqref{ln-lp_cond} fulfilled for $\alpha_i=\alpha=2$. 
\end{proof}

\subsection{Proof of Corollary \ref{cor:burgers}}\label{sec:appendix-burgers}

Corollary \ref{cor:burgers} follows from using the same transformation and feedback as in the linear system and dealing with the nonlinearities as in \cite{Gagnon2022-fredholm-laplacien}. Let consider first the system \eqref{eq:burgers} without its nonlinear part $u\partial_x u.$ Then, it becomes exactly the heat equation \eqref{eq:heat}  and in view of Corollary \ref{cor:expstab} ($\alpha_i=\alpha=2,$ $B_i=\phi_i,$ $r_i=r,$ $\gamma_i=\gamma$) since \eqref{eq:condburgers} holds, there exist an invertible operator $T\in \mathcal{L}(L^2(\mathbb{T}),L^2(\mathbb{T}))$ and feedback operator $K\in \mathcal{L}(H^{3/4}(\mathbb{T}),\mathbb{R}^2)$
which transform the linear system \eqref{eq:heat} to \eqref{eq:heatstab} and then make it exponentially stable in $L^2(\mathbb{T})$.
Due to \eqref{eq:condburgers}, $\phi=(\phi_1,\phi_2)\in H^{-1}(\mathbb{T}).$ This combining with the fact that $K\in \mathcal{L}(H^{3/4}(\mathbb{T}),\mathbb{R}^2),$ allows to apply \cite[Lemma 5.5]{Gagnon2022-fredholm-laplacien}, to the nonlinear closed loop system
\begin{align}
    \partial_t u-\Delta u+u\partial_x u=\phi K(u),\label{eq:burgerclosed}
\end{align} and then concludes that it has a unique solution $u$ for any initial condition $u_0$ in $L^2(\mathbb{T}).$ In particular there exists $\omega\in \mathbb{R}$ such that,
\begin{align*}
    u\in C^0([0,+\infty);L^2(\mathbb{T}))\cap L^2_{loc}((0,+\infty),H^1(\mathbb{T}))\cap H^1_{loc}((0,+\infty);H^{-1}(\mathbb{T})),
\end{align*} and 
\begin{align}
    \|u(t,\cdot)\|_{L^2(\mathbb{T})}\leq e^{\omega t}\|u_0\|_{L^2(\mathbb{T})}, \quad \forall t\geq 0.\label{ini-burg}
\end{align}
Let $\tau>0$ and $u_0\in L^2(\mathbb{T})$ with $\|u_0\|_{L^2(\mathbb{T}}\leq \delta,$ ($\delta$ will be chosen later). Applying the transformation $T$ to \eqref{eq:burgerclosed} yields
\begin{align}
    \partial_t Tu-T\Delta u+T(u\partial_x u)=T\phi K(u).\label{eq:applyT}
\end{align}
As \eqref{eq:condburgers} holds, $\phi\in H^{-1}(\mathbb{T}),$ and by construction the operators $T$ and $K$ satisfy the operator equality \eqref{opeq-space} as
\begin{align*}
    T\Delta +T\phi K=(\Delta-\lambda I)T.
\end{align*} 
At least, this equality holds in $\mathcal{L}(H^{1}(\mathbb{T}), H^{-1}(\mathbb{T})).$ Since $u\in C^0([0,\tau); L^2(\mathbb{T}))\cap L^{2}([0,\tau);H^{1}(\mathbb{T}))$ 
then the following equality makes sense in $L^{2}([0,\tau);H^{-1}(\mathbb{T}))$
\begin{align*}
    T\Delta u+T\phi K(u)=\Delta Tu -\lambda Tu.
\end{align*} Then, it follows from \eqref{eq:applyT} that for almost every $t\in (0,\tau),$
\begin{align*}
    \partial_t Tu(t,\cdot)-\Delta Tu(t,\cdot) +\lambda Tu(t,\cdot)+T(u(t,\cdot)\partial_x u(t,\cdot))=0\ \text{in}\  H^{-1}(\mathbb{T}).
\end{align*} By setting $z(t,\cdot):=Tu(t,\cdot)\in H^1(\mathbb{T}),$ we get for almost every $t\in (0,\tau)$ that
\begin{align*}
    \partial_t z(t,\cdot)-\Delta z(t,\cdot) +\lambda z(t,\cdot)+T(T^{-1}z(t,\cdot)\partial_x T^{-1}z(t,\cdot))=0\ \text{in}\  H^{-1}(\mathbb{T}).
\end{align*}
Thus, we have
\begin{align*}
    \frac{1}{2}\partial_t \|z(t,\cdot)\|^2_{L^2(\mathbb{T})}&=\langle z(t, \cdot), \partial_t z(t,\cdot)\rangle\\
    &=\langle z, \Delta z -\lambda z-T(T^{-1}z\partial_x T^{-1}z)\rangle
\end{align*}
Since $\langle z, \Delta z\rangle=-\langle \nabla z, \nabla z\rangle=-\|\nabla z\|^2_{L^2(\mathbb{T})}=\|z\|^2_{L^2(\mathbb{T})}-\|z\|^2_{H^1(\mathbb{T})},$ it follows that
\begin{align*}
\frac{1}{2}\partial_t \|z(t,\cdot)\|^2_{L^2(\mathbb{T})}&= \|z\|^2_{L^2(\mathbb{T})}-\|z\|^2_{H^1(\mathbb{T})}-\lambda\|z\|_{L^2(\mathbb{T})}-\langle z, T(T^{-1}z\partial_x T^{-1}z)\rangle   \\
&\leq -(\lambda-1)\|z\|^2_{L^2(\mathbb{T})}-\|z\|^2_{H^1(\mathbb{T})}+C\|z\|_{L^2(\mathbb{T})}\|T^{-1}z\partial_x T^{-1}z\|_{L^2(\mathbb{T})}\\
& \leq -(\lambda-1)\|z\|^2_{L^2(\mathbb{T})}-\|z\|^2_{H^1(\mathbb{T})}+\frac{C}{2}\|z\|_{L^2(\mathbb{T})}\|\partial_x ((T^{-1}z)^2)\|_{L^2(\mathbb{T})}\\
&\leq -(\lambda-1)\|z\|^2_{L^2(\mathbb{T})}-\|z\|^2_{H^1(\mathbb{T})}+\frac{C}{2}\|z\|_{L^2(\mathbb{T})}\| (T^{-1}z)^2\|_{H^1(\mathbb{T})}\\
&\leq -(\lambda-1)\|z\|^2_{L^2(\mathbb{T})}-\|z\|^2_{H^1(\mathbb{T})}+\frac{C}{2}
\|z\|_{L^2(\mathbb{T})}\|z\|^2_{H^1(\mathbb{T})},
\end{align*}
where $C$ is a constant that can change between lines but does not depend on $z$ or $t$.
Thus, considering $\|z\|_{L^2(\mathbb{T})}$ small enough on $[0,\tau)$ (for instance if $\sup_{[0,\tau)}\|z(t,\cdot)\|_{L^2(\mathbb{T})}\leq 1/C$), we get that
\begin{align*}
 \frac{1}{2}\partial_t \|z(t,\cdot)\|^2_{L^2(\mathbb{T})} \leq -(\lambda-1)\|z(t,\cdot)\|^2_{L^2(\mathbb{T})}, 
\end{align*} and in particular
\begin{align*}
    \|z(t,\cdot)\|_{L^2(\mathbb{T})} \leq e^{-(\lambda-1)t}\|z_0\|_{L^2(\mathbb{T})}, \quad \forall t\in [0,\tau)
\end{align*} with $\|z\|_{L^2(\mathbb{T})}$ small enough. Using the isomorphism property of $T,$ we get,
\begin{align}
    \|u(t,\cdot)\|_{L^2(\mathbb{T})} \lesssim e^{-(\lambda-1)t}\|u_0\|_{L^2(\mathbb{T})}, \quad \forall t\in [0,\tau)\label{eq:stab-burg-estim}
\end{align} with $\|u\|_{L^2(\mathbb{T})}$ small enough. Since the solution $u$ is unique and satisfies \eqref{ini-burg}, it suffices to consider $\|u_0\|$ sufficiently small to have $\|u\|_{L^2(\mathbb{T})}$ small enough. Thus, there exists $\delta_0(\tau,\lambda)$ such that for any $\delta\in (0,\delta_0),$ we have $\|u_0\|_{L^2(\mathbb{T})}$ and the solution of \eqref{eq:burgerclosed} satisfies the exponential stability estimate \eqref{eq:stab-burg-estim}. Actually, using a classical argument (see for instance \cite[Section 6.2]{Gagnon2022-fredholm-laplacien}),  the estimate \eqref{eq:stab-burg-estim} can be extended to $[0,+\infty)$ in the following sense: there exists small $\delta_{1}(\lambda)>0$ such that for any $\delta\in[0,\delta_{1})$, if $\|u_0\|_{L^2(\mathbb{T})}\leq \delta$ then
\begin{align*}
    \|u(t,\cdot)\|_{L^2(\mathbb{T})} \lesssim e^{-(\lambda-1)t}\|u_0\|_{L^2(\mathbb{T})}, \quad \forall t\in [0,+\infty).
\end{align*} This concludes the proof.

\bibliography{References}

@article{kunisch2025frequency,
  title={Frequency-domain criterion on the stabilizability for infinite-dimensional linear control systems},
  author={Kunisch, Karl and Wang, Gengsheng and Yu, Huaiqiang},
  journal={Journal de Math{\'e}matiques Pures et Appliqu{\'e}es},
  volume={196},
  pages={103690},
  year={2025},
  publisher={Elsevier}
}

@article{astrom1995pid,
  title={PID controllers: theory, design, and tuning},
  author={Astrom, Karl J},
  journal={The international society of measurement and control},
  year={1995}
}

@article{kellett2019feedback,
  title={Feedback, dynamics, and optimal control in climate economics},
  author={Kellett, Christopher M and Weller, Steven R and Faulwasser, Timm and Gr{\"u}ne, Lars and Semmler, Willi},
  journal={Annual Reviews in Control},
  volume={47},
  pages={7--20},
  year={2019},
  publisher={Elsevier}
}

@article{delle2019feedback,
  title={Feedback control algorithms for the dissipation of traffic waves with autonomous vehicles},
  author={Delle Monache, Maria Laura and Liard, Thibault and Rat, Ana{\"\i}s and Stern, Raphael and Bhadani, Rahul and Seibold, Benjamin and Sprinkle, Jonathan and Work, Daniel B and Piccoli, Benedetto},
  journal={Computational Intelligence and Optimization Methods for Control Engineering},
  pages={275--299},
  year={2019},
  publisher={Springer}
}

@ARTICLE{Goatin_stabilization_2024,
  author={Goatin, Paola},
  journal={IEEE Control Systems Letters}, 
  title={Dissipation of Stop-and-Go Waves in Traffic Flows Using Controlled Vehicles: A Macroscopic Approach}, 
  year={2024},
  volume={},
  number={},
  pages={1-1},
  doi={10.1109/LCSYS.2024.3401092}}

@article{hayat2023dissipation,
  title={Dissipation of traffic jams using a single autonomous vehicle on a ring road},
  author={Hayat, Amaury and Piccoli, Benedetto and Truong, Sydney},
  journal={SIAM Journal on Applied Mathematics},
  volume={83},
  number={3},
  pages={909--937},
  year={2023},
  publisher={SIAM}
}

@article{xiang2024quantitative,
  title={Quantitative rapid and finite time stabilization of the heat equation},
  author={Xiang, Shengquan},
  journal={ESAIM: Control, Optimisation and Calculus of Variations},
  volume={30},
  pages={40},
  year={2024},
  publisher={EDP Sciences}
}

@article{Coron2022-stabilization,
	title={Stabilization of the linearized water tank system},
	author={Coron, Jean-Michel and Hayat, Amaury and Xiang, Shengquan and Zhang, Christophe},
	journal={Archive for Rational Mechanics and Analysis},
	volume={244},
	number={3},
	pages={1019--1097},
	year={2022},
	publisher={Springer}
}

@article{Coron2014-local,
	title={Local rapid stabilization for a {K}orteweg-de {V}ries equation with a {N}eumann boundary control on the right},
	author={Coron, Jean-Michel and L{\"u}, Qi},
	journal={Journal de Math{\'e}matiques Pures et Appliqu{\'e}es},
	volume={102},
	number={6},
	pages={1080--1120},
	year={2014},
	publisher={Elsevier}
}

@article{breiten2014riccati,
  title={Riccati-based feedback control of the monodomain equations with the Fitzhugh--Nagumo model},
  author={Breiten, Tobias and Kunisch, Karl},
  journal={SIAM Journal on Control and Optimization},
  volume={52},
  number={6},
  pages={4057--4081},
  year={2014},
  publisher={SIAM}
}

@article{coron2017finite,
  title={Finite-time boundary stabilization of general linear hyperbolic balance laws via Fredholm backstepping transformation},
  author={Coron, Jean-Michel and Hu, Long and Olive, Guillaume},
  journal={Automatica},
  volume={84},
  pages={95--100},
  year={2017},
  publisher={Elsevier}
}

@article{deutscher2019fredholm,
  title={Fredholm backstepping control of coupled linear parabolic PDEs with input and output delays},
  author={Deutscher, Joachim and Gabriel, Jakob},
  journal={IEEE Transactions on Automatic Control},
  volume={65},
  number={7},
  pages={3128--3135},
  year={2019},
  publisher={IEEE}
}

@article{alabau2021superexponential,
  title={Superexponential stabilizability of evolution equations of parabolic type via bilinear control},
  author={Alabau-Boussouira, Fatiha and Cannarsa, Piermarco and Urbani, Cristina},
  journal={Journal of Evolution Equations},
  volume={21},
  pages={941--967},
  year={2021},
  publisher={Springer}
}

@article{redaud2022stabilizing,
  title={Stabilizing output-feedback control law for hyperbolic systems using a Fredholm transformation},
  author={Redaud, Jeanne and Auriol, Jean and Niculescu, Silviu-Iulian},
  journal={IEEE transactions on automatic control},
  volume={67},
  number={12},
  pages={6651--6666},
  year={2022},
  publisher={IEEE}
}

@article{hu2015control,
  title={Control of homodirectional and general heterodirectional linear coupled hyperbolic PDEs},
  author={Hu, Long and Di Meglio, Florent and Vazquez, Rafael and Krstic, Miroslav},
  journal={IEEE Transactions on Automatic Control},
  volume={61},
  number={11},
  pages={3301--3314},
  year={2015},
  publisher={IEEE}
}

@article{raymond2007feedback,
  title={Feedback boundary stabilization of the three-dimensional incompressible Navier--Stokes equations},
  author={Raymond, J-P},
  journal={Journal de math{\'e}matiques pures et appliqu{\'e}es},
  volume={87},
  number={6},
  pages={627--669},
  year={2007},
  publisher={Elsevier}
}

@article{badra2014fattorini,
  title={On the Fattorini criterion for approximate controllability and stabilizability of parabolic systems},
  author={Badra, Mehdi and Takahashi, Tak{\'e}o},
  journal={ESAIM: Control, Optimisation and Calculus of Variations},
  volume={20},
  number={3},
  pages={924--956},
  year={2014},
  publisher={EDP Sciences}
}

@article{banks1984linear,
  title={The linear regulator problem for parabolic systems},
  author={Banks, HT and Kunisch, Karl},
  journal={SIAM Journal on Control and Optimization},
  volume={22},
  number={5},
  pages={684--698},
  year={1984},
  publisher={SIAM}
}

@article{CoronLu15,
  title={Fredholm transform and local rapid stabilization for a {K}uramoto-{S}ivashinsky equation},
  author={Coron, Jean-Michel and L{\"u}, Qi},
  journal={Journal of Differential Equations},
  volume={259},
  number={8},
  pages={3683--3729},
  year={2015},
  publisher={Elsevier}
}

@article{slemrod1974note,
  title={A note on complete controllability and stabilizability for linear control systems in {H}ilbert space},
  author={Slemrod, Marshall},
  journal={SIAM Journal on Control},
  volume={12},
  number={3},
  pages={500--508},
  year={1974},
  publisher={SIAM}
}

@article{badra2020local,
  title={Local feedback stabilization of time-periodic evolution equations by finite dimensional controls},
  author={Badra, Mehdi and Mitra, Debanjana and Ramaswamy, Mythily and Raymond, Jean-Pierre},
  journal={ESAIM: Control, Optimisation and Calculus of Variations},
  volume={26},
  pages={101},
  year={2020},
  publisher={EDP Sciences}
}

@article{badra2011stabilization,
  title={Stabilization of parabolic nonlinear systems with finite dimensional feedback or dynamical controllers: Application to the Navier--Stokes system},
  author={Badra, Mehdi and Takahashi, Tak{\'e}o},
  journal={SIAM journal on control and optimization},
  volume={49},
  number={2},
  pages={420--463},
  year={2011},
  publisher={SIAM}
}

@incollection{jakubczyk1993remarks,
  title={Remarks on equivalence and linearization of nonlinear systems},
  author={Jakubczyk, B},
  booktitle={Nonlinear Control Systems Design 1992},
  pages={143--147},
  year={1993},
  publisher={Elsevier}
}

@article{fliess1999lie,
  title={A lie-backlund approach to equivalence and flatness of nonlinear systems},
  author={Fliess, Michel and L{\'e}vine, Jean and Martin, Philippe and Rouchon, Pierre},
  journal={IEEE Transactions on automatic control},
  volume={44},
  number={5},
  pages={922--937},
  year={1999},
  publisher={IEEE}
}

@inproceedings{martin2001flat,
  title={Flat systems, equivalence and feedback},
  author={Martin, Philippe and Murray, Richard M and Rouchon, Pierre},
  booktitle={Advances in the control of nonlinear systems},
  pages={5--32},
  year={2001},
  organization={Springer}
}

@article{cheng2005feedback,
  title={On feedback equivalence to port controlled Hamiltonian systems},
  author={Cheng, Daizhan and Astolfi, Alessandro and Ortega, Romeo},
  journal={Systems \& control letters},
  volume={54},
  number={9},
  pages={911--917},
  year={2005},
  publisher={Elsevier}
}

@article{nguyen2024stabilization,
  title={Stabilization of control systems associated with a strongly continuous group},
  author={Nguyen, Hoai-Minh},
  journal={arXiv preprint arXiv:2402.07560},
  year={2024}
}

@misc{CoronCDF,
  author       = {Coron, Jean-Michel},
  title        = {Stabilisation en temps fini},
  year         = {2017},
  month        = {jan},
  howpublished = {Collège de France seminar},
}

@article{zhang2022internal,
  title={Internal rapid stabilization of a 1-{D} linear transport equation with a scalar feedback},
  author={Zhang, Christophe},
  journal={Mathematical Control and Related Fields},
  volume={12},
  number={1},
  pages={169--200},
  year={2022},
  publisher={Mathematical Control and Related Fields}
}

@article{Gagnon2022-fredholm,
	title={Fredholm backstepping for critical operators and application to rapid stabilization for the linearized water waves},
	author={Gagnon, Ludovick and Hayat, Amaury and Xiang, Shengquan and Zhang, Christophe},
	journal={to appear in Ann. Inst. Fourier},
	year={2022}
}

@article {backsteppingadapt,
    AUTHOR = {Coron, Jean-Michel and d'Andr\'{e}a-Novel, Brigitte},
     TITLE = {Stabilization of a rotating body beam without damping},
   JOURNAL = {IEEE Trans. Automat. Control},
  FJOURNAL = {Institute of Electrical and Electronics Engineers.
              Transactions on Automatic Control},
    VOLUME = {43},
      YEAR = {1998},
    NUMBER = {5},
     PAGES = {608--618},
      ISSN = {0018-9286},
   MRCLASS = {93C85 (74H55 74M05 93D15)},
  MRNUMBER = {1618052},
       DOI = {10.1109/9.668828},
       URL = {https://doi.org/10.1109/9.668828},
}

@article {backstepping1,
    AUTHOR = {Byrnes, Christopher I. and Isidori, Alberto},
     TITLE = {New results and examples in nonlinear feedback stabilization},
   JOURNAL = {Systems Control Lett.},
  FJOURNAL = {Systems \& Control Letters},
    VOLUME = {12},
      YEAR = {1989},
    NUMBER = {5},
     PAGES = {437--442},
      ISSN = {0167-6911},
   MRCLASS = {93C10 (93D15)},
  MRNUMBER = {1005310},
MRREVIEWER = {Lin Huang},
       DOI = {10.1016/0167-6911(89)90080-7},
       URL = {https://doi.org/10.1016/0167-6911(89)90080-7},
}

@article{backstepping2,
  title={Adaptive techniques for mechanical systems},
  author={Koditschek, Daniel E},
  journal={Proc. 5th. Yale University
Conference},
PAGES = {pp. 259–265},
  year={1987}
}

@article {backstepping3,
    AUTHOR = {Tsinias, John},
     TITLE = {Sufficient {L}yapunov-like conditions for stabilization},
   JOURNAL = {Math. Control Signals Systems},
  FJOURNAL = {Mathematics of Control, Signals, and Systems},
    VOLUME = {2},
      YEAR = {1989},
    NUMBER = {4},
     PAGES = {343--357},
      ISSN = {0932-4194},
   MRCLASS = {93D15 (93C10)},
  MRNUMBER = {1015672},
MRREVIEWER = {A. Halanay},
       DOI = {10.1007/BF02551276},
       URL = {https://doi.org/10.1007/BF02551276},
}

@article {BBK2003,
    AUTHOR = {Bo\v{s}kovi\'{c}, Dejan M. and Balogh, Andras and Krsti\'{c}, Miroslav},
     TITLE = {Backstepping in infinite dimension for a class of parabolic
              distributed parameter systems},
   JOURNAL = {Math. Control Signals Systems},
  FJOURNAL = {Mathematics of Control, Signals, and Systems},
    VOLUME = {16},
      YEAR = {2003},
    NUMBER = {1},
     PAGES = {44--75},
      ISSN = {0932-4194},
   MRCLASS = {93D15 (93B51 93C20)},
  MRNUMBER = {1993488},
MRREVIEWER = {Emmanuelle Cr\'{e}peau},
}

@article{BaloghKrstic,
  title={Infinite dimensional backstepping-style feedback transformations for a heat equation with an arbitrary level of instability},
  author={Balogh, Andras and Krstic, Miroslav},
  journal={European journal of control},
  volume={8},
  number={2},
  pages={165--175},
  year={2002},
  publisher={Elsevier}
}

@book {Kato-book,
    AUTHOR = {Kato, Tosio},
     TITLE = {Perturbation theory for linear operators},
    SERIES = {Classics in Mathematics},
      NOTE = {Reprint of the 1980 edition},
 PUBLISHER = {Springer-Verlag, Berlin},
      YEAR = {1995},
     PAGES = {xxii+619},
      ISBN = {3-540-58661-X},
   MRCLASS = {47A55 (46-00 47-00)},
  MRNUMBER = {1335452},
}

@article{beauchard2014local,
  title={Local controllability of 1{D} {S}chr{\"o}dinger equations with bilinear control and minimal time},
  author={Beauchard, Karine and Morancey, Morgan},
  journal={Mathematical Control and Related Fields},
  volume={4},
  number={2},
  pages={125--160},
  year={2014}
}

@article{beauchard2010local,
  title={Local controllability of 1{D} linear and nonlinear {S}chr{\"o}dinger equations with bilinear control},
  author={Beauchard, Karine and Laurent, Camille},
  journal={Journal de math{\'e}matiques pures et appliqu{\'e}es},
  volume={94},
  number={5},
  pages={520--554},
  year={2010},
  publisher={Elsevier}
}

@book{KrsticBook,
	Author = {Krstic, Miroslav and Smyshlyaev, Andrey},
	Doi = {10.1137/1.9780898718607},
	Isbn = {978-0-89871-650-4},
	Mrclass = {93C20 (35B37 74M05 76D55 93C40)},
	Mrnumber = {2412038},
	Pages = {x+192},
	Publisher = {Society for Industrial and Applied Mathematics (SIAM), Philadelphia, PA},
	Series = {Advances in Design and Control},
	Title = {Boundary {C}ontrol of {PDE}s: A {C}ourse on {B}ackstepping {D}esigns},
	Url = {http://dx.doi.org/10.1137/1.9780898718607},
	Volume = {16},
	Year = {2008}}

@article{vest2013rapid,
  title={Rapid stabilization in a semigroup framework},
  author={Vest, Ambroise},
  journal={SIAM Journal on Control and Optimization},
  volume={51},
  number={5},
  pages={4169--4188},
  year={2013},
  publisher={SIAM}
}

@article {Coron2018-rapid,
    AUTHOR = {J.-M. Coron and L. Gagnon and M. Morancey},
     TITLE = {Rapid stabilization of a linearized bilinear 1-{D}
              {S}chr\"odinger equation},
   JOURNAL = {J. Math. Pures Appl. (9)},
  FJOURNAL = {Journal de Math\'ematiques Pures et Appliqu\'ees. Neuvi\`eme S\'erie},
    VOLUME = {115},
      YEAR = {2018},
     PAGES = {24--73},
      ISSN = {0021-7824},
   MRCLASS = {35Q55 (93B52 93D15)},
  MRNUMBER = {3808341},
       DOI = {10.1016/j.matpur.2017.10.006},
       URL = {https://doi.org/10.1016/j.matpur.2017.10.006},
}

@article{urquiza2005rapid,
  title={Rapid exponential feedback stabilization with unbounded control operators},
  author={Urquiza, Jose Manuel},
  journal={SIAM journal on control and optimization},
  volume={43},
  number={6},
  pages={2233--2244},
  year={2005},
  publisher={SIAM}
}

@book{zabczyk2020mathematical,
  title={Mathematical control theory},
  author={Zabczyk, Jerzy},
  year={2020},
  publisher={Springer}
}

@article{pouchol2018global,
  title={Global stability with selection in integro-differential {L}otka-{V}olterra systems modelling trait-structured populations},
  author={Pouchol, Camille and Tr{\'e}lat, Emmanuel},
  journal={Journal of Biological Dynamics},
  volume={12},
  number={1},
  pages={872--893},
  year={2018},
  publisher={Taylor \& Francis}
}

@article{ma2023feedback,
  title={Feedback law to stabilize linear infinite-dimensional systems},
  author={Ma, Yaxing and Wang, Gengsheng and Yu, Huaiqiang},
  journal={Mathematical Control and Related Fields},
  volume={13},
  number={3},
  pages={1160--1183},
  year={2023},
  publisher={Mathematical Control and Related Fields}
}

@article{liu2022characterizations,
  title={Characterizations of complete stabilizability},
  author={Liu, Hanbing and Wang, Gengsheng and Xu, Yashan and Yu, Huaiqiang},
  journal={SIAM Journal on Control and Optimization},
  volume={60},
  number={4},
  pages={2040--2069},
  year={2022},
  publisher={SIAM}
}

@article{slemrod1972linear,
  title={The linear stabilization problem in {H}ilbert space},
  author={Slemrod, Marshall},
  journal={Journal of Functional Analysis},
  volume={11},
  number={3},
  pages={334--345},
  year={1972},
  publisher={Elsevier}
}

@article{nguyen2024rapid,
  title={Rapid stabilization and finite time stabilization of the bilinear {S}chr\" odinger equation},
  author={Nguyen, Hoai-Minh},
  journal={arXiv preprint arXiv:2405.10002},
  year={2024}
}

@Book{LT1,
 Author = {Lasiecka, Irena and Triggiani, Roberto},
 Title = {Control theory for partial differential equations: continuous and approximation theories. 1: {Abstract} parabolic systems},
 FSeries = {Encyclopedia of Mathematics and Its Applications},
 Series = {Encycl. Math. Appl.},
 ISSN = {0953-4806},
 Volume = {74},
 ISBN = {0-521-43408-4},
 Year = {2000},
 Publisher = {Cambridge: Cambridge University Press},
 Language = {English},
 zbMATH = {1447315},
 Zbl = {0942.93001}
}

@article{valein2009stabilization,
  title={Stabilization of the wave equation on 1-{D} networks},
  author={Valein, Julie and Zuazua, Enrique},
  journal={SIAM Journal on control and Optimization},
  volume={48},
  number={4},
  pages={2771--2797},
  year={2009},
  publisher={SIAM}
}

@article{dehman2006stabilization,
  title={Stabilization and control for the nonlinear {S}chr{\"o}dinger equation on a compact surface},
  author={Dehman, B and G{\'e}rard, P and Lebeau, G},
  journal={Mathematische Zeitschrift},
  volume={254},
  number={4},
  pages={729--749},
  year={2006},
  publisher={Springer}
}

@article{jaffard1998singular,
  title={Singular internal stabilization of the wave equation},
  author={Jaffard, St{\'e}phane and Tucsnak, Marius and Zuazua, Enrique},
  journal={journal of differential equations},
  volume={145},
  number={1},
  pages={184--215},
  year={1998},
  publisher={Elsevier}
}

@article{zuazua1990uniform,
  title={Uniform stabilization of the wave equation by nonlinear boundary feedback},
  author={Zuazua, Enrike},
  journal={SIAM Journal on Control and Optimization},
  volume={28},
  number={2},
  pages={466--477},
  year={1990},
  publisher={SIAM}
}

@article{machtyngier1994stabilization,
  title={Stabilization of the {S}chrodinger equation},
  author={Machtyngier, Elaine and Zuazua, Enrique},
  journal={Portugaliae Mathematica},
  volume={51},
  number={2},
  pages={243--256},
  year={1994},
  publisher={Lisboa, Gazeta de matematica [etc.]}
}

@inproceedings{coron2015stabilization,
  title={Stabilization of control systems and nonlinearities},
  author={Coron, Jean-Michel},
  booktitle={Proceedings of the 8th International Congress on Industrial and Applied Mathematics},
  pages={17--40},
  year={2015},
  organization={Higher Ed. Press Beijing}
}

@book{barbu2018controllability,
  title={Controllability and stabilization of parabolic equations},
  author={Barbu, Viorel},
  year={2018},
  publisher={Springer}
}

@book{coron2007control,
  title={Control and nonlinearity},
  author={Coron, Jean-Michel},
  number={136},
  year={2007},
  publisher={American Mathematical Soc.}
}

@article{gagnon2021fredholm,
  title={A Fredholm transformation for the rapid stabilization of a degenerate parabolic equation},
  author={Gagnon, Ludovick and Lissy, Pierre and Marx, Swann},
  journal={SIAM Journal on Control and Optimization},
  volume={59},
  number={5},
  pages={3828--3859},
  year={2021},
  publisher={SIAM}
}

@article{lissy2023rapid,
  title={Rapid stabilization of a degenerate parabolic equation using a backstepping approach: the case of a boundary control acting at the degeneracy},
  author={Lissy, Pierre and Moreno, Claudia},
  journal={Mathematical Control and Related Fields},
  pages={0--0},
  year={2023},
  publisher={Mathematical Control and Related Fields}
}

@article{brunovsky1970classification,
  title={A classification of linear controllable systems},
  author={Brunovsk{\`y}, Pavol},
  journal={Kybernetika},
  volume={6},
  number={3},
  pages={173--188},
  year={1970},
  publisher={Institute of Information Theory and Automation AS CR}
}

@article{HS,
  title={A quadratic {L}yapunov function for {S}aint-{V}enant equations with arbitrary friction and space-varying slope},
  author={Hayat, Amaury and Shang, Peipei},
  journal={Automatica},
  volume={100},
  pages={52--60},
  year={2019},
  publisher={Elsevier}
}

@book{tucsnak2009observation,
  title={Observation and control for operator semigroups},
  author={Tucsnak, Marius and Weiss, George},
  year={2009},
  publisher={Springer Science \& Business Media}
}

@book{lions1971optimal,
  title={Optimal control of systems governed by partial differential equations},
  author={Lions, Jacques Louis},
  volume={170},
  year={1971},
  publisher={Springer}
}

@Book{LT2,
 Author = {Lasiecka, Irena and Triggiani, Roberto},
 Title = {Control theory for partial differential equations: continuous and approximation theories. 2: {Abstract} hyperbolic-like systems over a finite time horizon},
 FSeries = {Encyclopedia of Mathematics and Its Applications},
 Series = {Encycl. Math. Appl.},
 ISSN = {0953-4806},
 Volume = {75},
 ISBN = {0-521-58401-9},
 Year = {2000},
 Publisher = {Cambridge: Cambridge University Press},
 Language = {English},
 zbMATH = {1535627},
 Zbl = {0961.93003}
}

@book{BastinCoronBook,
	Author = {Bastin, Georges and Coron, Jean-Michel},
	Publisher = {Springer International},
	Series = {Number 88 in Progress in Nonlinear Differential Equations and Their Applications},
	Title = {Stability and Boundary Stabilisation of 1-{D} Hyperbolic Systems},
	Year = {2016}}

@article{viel1997global,
  title={Global stabilization of exothermic chemical reactors under input constraints},
  author={Viel, F and Jadot, Fabrice and Bastin, Georges},
  journal={Automatica},
  volume={33},
  number={8},
  pages={1437--1448},
  year={1997},
  publisher={Elsevier}
}

@article{caponigro2013sparse,
  title={Sparse stabilization and optimal control of the {C}ucker-{S}male model},
  author={Caponigro, Marco and Fornasier, Massimo and Piccoli, Benedetto and Tr{\'e}lat, Emmanuel},
  journal={Mathematical Control and Related Fields},
  volume={3},
  number={4},
  pages={447--466},
  year={2013}
}

@article{hansen1997new,
  title={New results on the operator {C}arleson measure criterion},
  author={Hansen, Scott and Weiss, George},
  journal={IMA Journal of Mathematical Control and Information},
  volume={14},
  number={1},
  pages={3--32},
  year={1997},
  publisher={OUP}
}

@book{PazyBook,
    AUTHOR = {Pazy, A.},
     TITLE = {Semigroups of linear operators and applications to partial
              differential equations},
    SERIES = {Applied Mathematical Sciences},
    VOLUME = {44},
 PUBLISHER = {Springer-Verlag, New York},
      YEAR = {1983},
     PAGES = {viii+279},
      ISBN = {0-387-90845-5},
   MRCLASS = {47D05 (34Gxx 35Fxx 35Gxx 47H20)},
  MRNUMBER = {710486},
MRREVIEWER = {H. O. Fattorini},
       DOI = {10.1007/978-1-4612-5561-1},
       URL = {https://doi.org/10.1007/978-1-4612-5561-1},
}

@article{komornik1997rapid,
  title={Rapid boundary stabilization of linear distributed systems},
  author={Komornik, Vilmos},
  journal={SIAM journal on control and optimization},
  volume={35},
  number={5},
  pages={1591--1613},
  year={1997},
  publisher={SIAM}
}

@article{coron2015dissipative,
  title={Dissipative boundary conditions for nonlinear 1-{D} hyperbolic systems: sharp conditions through an approach via time-delay systems},
  author={Coron, Jean-Michel and Nguyen, Hoai-Minh},
  journal={SIAM Journal on Mathematical Analysis},
  volume={47},
  number={3},
  pages={2220--2240},
  year={2015},
  publisher={SIAM}
}

@article{weiss2011eigenvalues,
  title={Eigenvalues and eigenvectors of semigroup generators obtained from diagonal generators by feedback},
  author={Weiss, George and Xu, Cheng-Zhong},
  journal={Communications in Information and Systems},
  volume={11},
  number={1},
  pages={71--104},
  year={2011},
  publisher={International Press of Boston}
}

@article{russell1994general,
  title={A general necessary condition for exact observability},
  author={Russell, David L and Weiss, George},
  journal={SIAM Journal on Control and Optimization},
  volume={32},
  number={1},
  pages={1--23},
  year={1994},
  publisher={SIAM}
}

@article{trelat2019characterization,
  title={Characterization by observability inequalities of controllability and stabilization properties},
  author={Tr{\'e}lat, Emmanuel and Wang, Gengsheng and Xu, Yashan},
  journal={Pure and Applied Analysis},
  volume={2},
  number={1},
  pages={93--122},
  year={2019},
  publisher={Mathematical Sciences Publishers}
}

@article{Gagnon2022-fredholm-laplacien,
	title={Fredholm transformation on {L}aplacian and rapid stabilization for the heat equation},
	author={Gagnon, Ludovick and Hayat, Amaury and Xiang, Shengquan and Zhang, Christophe},
	journal={Journal of Functional Analysis},
	volume={283},
	number={12},
	pages={109664},
	year={2022},
	publisher={Elsevier}
}

@article{Fulton1994-eigenvalue,
  title={Eigenvalue and eigenfunction asymptotics for regular {S}turm-{L}iouville problems},
  author={Fulton, Charles T and Pruess, Steven A},
  journal={Journal of Mathematical Analysis and Applications},
  volume={188},
  number={1},
  pages={297--340},
  year={1994},
  publisher={Elsevier}
}

@article{Feki2014-riesz,
  title={On a {R}iesz basis of eigenvectors of a nonself-adjoint analytic operator and applications},
  author={Feki, Ines and Jeribi, Aref and Sfaxi, Ridha},
  journal={Linear and Multilinear Algebra},
  volume={62},
  number={8},
  pages={1049--1068},
  year={2014},
  publisher={Taylor \& Francis}
}

@article{Aimar1999-On-an-unconditional,
  title={On an unconditional basis of generalized eigenvectors of the nonself-adjoint {G}ribov operator in {B}argmann space},
  author={Aimar, Marie-Therese and Intissar, Abdelkader and Jeribi, Aref},
  journal={Journal of mathematical analysis and applications},
  volume={231},
  number={2},
  pages={588--602},
  year={1999},
  publisher={Elsevier}
}

@article{Gribov1968-A-reggeon,
  title={A reggeon diagram technique},
  author={Gribov, VN},
  journal={Sov. Phys. JETP},
  volume={26},
  number={2},
  pages={414--423},
  year={1968}
}

@book{Intissar1997-analyse,
  title={Analyse fonctionnelle \& th{\'e}orie spectrale},
  author={Intissar, A},
  year={1997},
  publisher={C{\'e}padues-{\'e}ditions}
}

@book{Brezis2011-functional,
  title={Functional analysis, {S}obolev spaces and partial differential equations},
  author={Brezis, Haim},
  volume={2},
  number={3},
  year={2011},
  publisher={Springer}
}

@article{Kong1996-eigenvalues,
  title={Eigenvalues of regular {S}turm-{L}iouville problems},
  author={Kong, Q and Zettl, A},
  journal={Journal of differential equations},
  volume={131},
  number={1},
  pages={1--19},
  year={1996},
  publisher={Elsevier}
}

@article{Bailey1996-regular,
  title={Regular and singular {S}turm-{L}iouville problems with coupled boundary conditions},
  author={Bailey, PB and Everitt, WN and Zettl, A},
  journal={Proceedings of the Royal Society of Edinburgh Section A: Mathematics},
  volume={126},
  number={3},
  pages={505--514},
  year={1996},
  publisher={Royal Society of Edinburgh Scotland Foundation}
}

@article{Moller1999-unboundedness,
  title={On the unboundedness below of the {S}turm—{L}iouville operator},
  author={M{\"o}ller, Manfred},
  journal={Proceedings of the Royal Society of Edinburgh Section A: Mathematics},
  volume={129},
  number={5},
  pages={1011--1015},
  year={1999},
  publisher={Royal Society of Edinburgh Scotland Foundation}
}

@book{Weidmann2006-spectral,
  title={Spectral theory of ordinary differential operators},
  author={Weidmann, Joachim},
  volume={1258},
  year={2006},
  publisher={Springer}
}

@article{Gagnon2025abstract,
  title={An abstract setting for the Fredholm backstepping transformation: self-adjoint case},
  author={Gagnon, Ludovick and Hayat, Amaury and Marx, Swann and Zhang, Christophe and Xiang, Shengquan},
  year={2025}
}

@article{Boulard2025f,
  title={F-equivalence for parabolic systems and applications to the stabilization of nonlinear PDE},
  author={Boulard, Vincent and Hayat, Amaury},
  journal={arXiv preprint arXiv:2508.21605},
  year={2025}
}

@inproceedings{Rebarber2007frequency,
  title={Frequency domain methods for proving the uniform stability of vibrating systems},
  author={Rebarber, Richard},
  booktitle={Analysis and Optimization of Systems: State and Frequency Domain Approaches for Infinite-Dimensional Systems: Proceedings of the 10th International Conference Sophia-Antipolis, France, June 9--12, 1992},
  pages={366--377},
  year={2007},
  organization={Springer}
}

@article{Pruss1984spectrum,
  title={On the spectrum of $C_0$-semigroups},
  author={Pr{\"u}ss, Jan},
  journal={Transactions of the American Mathematical Society},
  volume={284},
  number={2},
  pages={847--857},
  year={1984}
}

@article{Huang1985characteristic,
  title={Characteristic conditions for exponential stability of linear dynamical systems in {H}ilbert spaces},
  author={Huang, Falun},
  journal={Ann. of Diff. Eqs.},
  volume={1},
  pages={43--56},
  year={1985}
}
\bibliographystyle{plain}
\end{document}